\documentclass[a4paper]{amsart}
\usepackage[francais,english]{babel}
\usepackage{amsmath,amssymb,amsthm}
\usepackage{amsfonts}
\usepackage{vmargin}
\usepackage{pifont}
\usepackage{euscript}
\usepackage{upgreek}
\usepackage{verbatim}
\RequirePackage{mathrsfs}

\DeclareMathOperator{\Hom}{Hom}
\DeclareMathOperator{\Ker}{Ker}
\DeclareMathOperator{\GL}{GL}
\DeclareMathOperator{\SL}{SL}
\DeclareMathOperator{\SO}{SO}
\DeclareMathOperator{\End}{End}
\DeclareMathOperator{\spec}{Spec}
\DeclareMathOperator{\card}{card}
\DeclareMathOperator{\codim}{codim}
\DeclareMathOperator{\jet}{jet}

\DeclareMathOperator{\IM}{Im}
\DeclareMathOperator{\RE}{Re}

\DeclareMathOperator{\gal}{Gal}
\DeclareMathOperator{\sh}{sh}
\DeclareMathOperator{\covol}{covol}
\DeclareMathOperator{\NS}{NS}
\DeclareMathOperator{\Pic}{Pic}

\theoremstyle{plain}
\newtheorem{theo}{Th\'eor\`eme}[section]
\newtheorem{lemma}[theo]{Lemme}
\newtheorem{prop}[theo]{Proposition}
\newtheorem{coro}[theo]{Corollaire}

\newtheorem{theoMW}{Th{\'e}or{\`e}me (Masser et W{\"u}stholz, 1993)}

\theoremstyle{definition}
\newtheorem{defi}[theo]{D\'efinition}
\newtheorem{rema}[theo]{Remarque}
\newtheorem{remas}[theo]{Remarques}
\newtheorem{proprietes}[theo]{Propri{\'e}t{\'e}s}

\title[Th\'eor\`eme des p\'eriodes et
isog\'enies]{Th\'eor\`eme des p\'eriodes et degr\'es
minimaux d'isog\'enies}
\author{\'Eric Gaudron et Ga\"el R\'emond}
\date{\today}
\begin{document}

\selectlanguage{francais}

\footnotetext{MSC~$2010$: 11G10 (11J86, 14G40, 14K02).}
\footnotetext{\textbf{Mots-clefs}: Lemme matriciel, th{\'e}or{\`e}me
des p{\'e}riodes, isog{\'e}nie minimale elliptique, probl{\`e}me
d'uniformit{\'e} de Serre, m{\'e}thode de la section auxiliaire,
pente d'Arakelov, lemme d'interpolation analytique. }
\begin{abstract} Nous donnons de nouvelles versions effectives du
th\'eor\`eme des p\'eriodes de Masser et W\"ustholz. Nos \'enonc\'es
sont totalement explicites et 
permettent de raffiner les applications aux th\'eor\`emes
d'isog\'enies elliptiques. Celles-ci entra\^\i nent \`a leur tour la
r\'esolution du probl\`eme d'uniformit\'e de Serre dans le cas des
sous-groupes de Cartan d\'eploy\'es, en conjonction avec les
travaux de Bilu, Parent et Rebolledo.
\\[.2cm] \textsc{Abstract}. We give a new, sharpened version of the period
theorem of Masser and W{\"u}stholz, which is moreover totally
explicit. We also present a new formulation involving all archimedean
places. We then derive new bounds for elliptic isogenies, improving
those of Pellarin. The small numerical constants obtained allow an
application to Serre's uniformity problem in the split Cartan case,
thanks to the work of Bilu, Parent and Rebolledo.

\end{abstract}
\maketitle
\textbf{Coordonn{\'e}es des auteurs~:}

\vspace{0,3cm}

\noindent {\'E}ric \textsc{Gaudron}\\ \ding{41} Universit{\'e}
Grenoble I, Institut Fourier.\\ UMR $5582$, BP $74$\\ $38402$
Saint-Martin-d'H{\`e}res Cedex, France.\\
Courriel~: \texttt{Eric.Gaudron@ujf-grenoble.fr}\\
\ding{37} (33) 04 76 51 45 72\\
Page internet~: \texttt{http://www-fourier.ujf-grenoble.fr/\~{}gaudron}

\vspace{0,3cm}

\noindent Ga{\"e}l \textsc{R{\'e}mond}\\ \ding{41} Universit{\'e}
Grenoble I, Institut Fourier.\\ UMR $5582$, BP $74$\\ $38402$
Saint-Martin-d'H{\`e}res Cedex, France.\\
Courriel~: \texttt{Gael.Remond@ujf-grenoble.fr}\\
\ding{37} (33) 04 76 51 49 86

\newpage
\section{Introduction}
Dans ce texte, nous revisitons le th\'eor\`eme des p\'eriodes de
Masser et W\"ustholz et ses applications aux degr\'es minimaux
d'isog\'enies entre courbes elliptiques. Notre pr\'esentation du
th\'eor\`eme lui-m\^eme diff\`ere des versions ant\'erieures et nous
expliquons ci-dessous ce qui nous a conduit \`a cette
formulation, notamment en lien avec le lemme matriciel, dont nous
utilisons une nouvelle version due {\`a} Autissier. Nous donnerons ensuite les
\'enonc\'es ainsi qu'une application au probl\`eme d'uniformit\'e de
Serre qui repose sur les travaux de Bilu, Parent et Rebolledo.

Dans tout ce texte $A$ est une vari\'et\'e ab\'elienne de dimension
$g$ sur un corps de nombres $k$. Pour parler de p\'eriodes, nous
fixons un plongement complexe $\sigma\colon k\hookrightarrow\mathbf{C}$ et
consid\'erons la vari\'et\'e ab\'elienne complexe $A_\sigma$ obtenue
par extension des scalaires, son espace tangent \`a l'origine
$t_{A_\sigma}$ et son r\'eseau des p\'eriodes $\Omega_{A_\sigma}$.

En $1985$~\cite{lnmmasser}, David Masser a d\'emontr\'e une majoration
des coefficients d'une matrice de p\'eriodes en fonction de la
hauteur d'une vari\'et\'e ab\'elienne principalement polaris\'ee. La
paragraphe de son texte consacr\'e \`a cette estimation portait le nom
de \emph{lemme matriciel} et cette terminologie est rest\'ee pour d\'esigner
ce type d'\'enonc\'es. Une nouvelle approche a \'et\'e introduite par
Bost~\cite{bost2,bost4} en termes de hauteur de Faltings et des versions
effectives ont \'et\'e donn\'ees par Graftieaux \cite{graftieaux}, David et Philippon~\cite{sinnoupph} et le premier auteur \cite{gaudrondeux}.

Si l'on veut s'affranchir de l'hypoth\`ese de polarisation principale,
il est pr\'ef\'erable de consid\'erer qu'un lemme matriciel donne une
minoration de la plus petite p\'eriode non nulle d'une vari\'et\'e
ab\'elienne (en une place donn\'ee). Il s'agit donc d'un premier
prototype d'un th\'eor\`eme des p\'eriodes puisqu'il s'agit de relier
la norme d'une p\'eriode (en l'occurrence la plus petite) \`a divers
invariants de la vari\'et\'e ab\'elienne (ici essentiellement la
hauteur de Faltings).

Un v\'eritable th\'eor\`eme des p\'eriodes, au sens attach\'e \`a ce
terme depuis les travaux fondateurs de Masser et W\"ustholz, doit, lui,
faire intervenir de plus de mani\`ere essentielle un terme de
degr\'e g{\'e}om{\'e}trique. Traditionnellement on l'\'ecrit comme une
majoration du degr\'e de la plus petite sous-vari\'et\'e ab\'elienne
$A_\omega$ de $A$ dont l'espace tangent contient une p\'eriode
donn\'ee $\omega$ en fonction de la norme de $\omega$ et de la hauteur
de $A$.

Nous utilisons dans nos \'enonc\'es la hauteur de Faltings stable
$h_{F}(A)$ d'une vari\'et\'e ab\'elienne sur un corps de nombres. Nous
fixons une polarisation $L$ sur $A$. La forme de Riemann de $L_\sigma$
munit $t_{A_\sigma}$ d'une norme hermitienne que nous notons
$\|\cdot\|_{L,\sigma}$ (les d{\'e}finitions pr{\'e}cises de ces objets
sont donn{\'e}es dans la partie suivante). On tire alors de
\cite{masserwustholz1993} l'\'enonc\'e suivant.

\begin{theoMW} Il existe une constante $c>0$, qui ne
d{\'e}pend que de $g$, $[k:\mathbf{Q}]$ et $\deg_LA$, et une constante
$\kappa>0$, qui ne d{\'e}pend que de $g$, telles que
$$\deg_LA_{\omega}\le c\max{(1,h_F(A),\|\omega\|_{L,\sigma}^2)}^\kappa.$$
De plus l'on peut choisir $\kappa=(g-1)4^{g}g!$ et
$c=c_0[k:\mathbf{Q}]^{\kappa}(\deg_{L}A)^{1+g\kappa}$ o\`u $c_0$
est une constante qui ne d\'epend que de $g$.\end{theoMW}

Nous proposons ici de voir un tel th\'eor\`eme plut\^ot comme une
minoration de la norme de $\omega$ en fonction du degr\'e de
$A_\omega$ et de la hauteur. En fait, dans cette approche, la
vari\'et\'e ab\'elienne $A_\omega$ joue le r\^ole principal et la
vari\'et\'e $A$ initiale est rel\'egu\'ee au second plan. Si nous
l'oublions compl\`etement, nous sommes en train de dire qu'un
th\'eor\`eme des p\'eriodes n'est autre que la minoration de la norme
de la plus petite p\'eriode $\omega$ de $A$ telle que $A_\omega=A$. En
d'autres termes encore, nous consid\'erons le minimum des normes des
p\'eriodes de $A$ qui ne sont p\'eriodes d'aucune sous-vari\'et\'e
ab\'elienne stricte de $A$.

Vu ainsi, le lemme matriciel devient une minoration d'un minimum
absolu $\rho$ du r\'eseau des p\'eriodes (minimum sur tous les \'el\'ements
non nuls) tandis que le th\'eor\`eme des p\'eriodes vise \`a minorer
un minimum essentiel $\updelta$ de ce m\^eme r\'eseau (minimum sur les
\'el\'ements transverses ou non d\'eg\'en\'er\'es au sens des
sous-vari\'et\'es ab\'eliennes). Bien entendu, ici $\rho\le\updelta$ et il
ne faut pas perdre de vue que la minoration souhait\'ee de $\updelta$ est
plus grande que celle de $\rho$ puisque sa caract\'eristique
principale est de cro{\^\i}tre avec le degr\'e de $A$. Notons aussi que
dans cette approche il est possible que $\updelta$ soit infini : cela
signifie simplement que la vari\'et\'e ab\'elienne $A$ consid\'er\'ee
n'est pas de la forme $(A')_{\omega'}$ pour un couple $(A',\omega')$.

Ce nouvel \'eclairage sur le th\'eor\`eme des p\'eriodes pr\'esente
plusieurs avantages. D'une part, on sait depuis Bost que l'on peut
exprimer naturellement le lemme matriciel en faisant intervenir toutes
les places. De mani\`ere pr\'ecise, on note $\rho(A_\sigma,L_\sigma)$
la valeur minimale de $\|\omega\|_{L,\sigma}$ pour
$\omega\in\Omega_{A_\sigma}$ non nul et nous utilisons dans ce texte la
forme suivante du lemme matriciel (qui raffine les versions
\'evoqu\'ees plus haut), tir{\'e}e du travail d'Autissier~\cite{autissier}.

\begin{theo}\label{matint} Si $(A,L)$ est une vari\'et\'e ab\'elienne
polaris\'ee de dimension $g$ sur un corps de nombres $k$ nous
avons $$\frac{1}{[k:\mathbf{Q}]}\sum_{\sigma\colon k\hookrightarrow\mathbf{C}}
\rho(A_\sigma,L_\sigma)^{-2}\le14\max(1,h_F(A),\log\deg_LA).$$\end{theo}

Ainsi il devient naturel
de formuler notre th\'eor\`eme des p\'eriodes comme une majoration
d'une moyenne de la forme $(1/D)\sum_\sigma1/\updelta_\sigma^2$ (par,
r\'ep\'etons-le, 
une puissance n\'egative du degr\'e de $A$) et nous constatons
effectivement que c'est une telle quantit\'e qui appara\^it
dans la preuve. Conform\'ement \`a ce qui pr\'ec\`ede, notons
$\updelta_0(A_\sigma,L_\sigma)$ la valeur minimale de $\|\omega\|_{L,\sigma}$ pour
$\omega\in\Omega_{A_\sigma}\setminus\bigcup_{\mathsf B}\Omega_{\mathsf B}$ o\`u
l'union porte sur les sous-vari\'et\'es ab\'eliennes strictes $\mathsf
B$ de $A_\sigma$. Si cet ensemble est vide, nous posons
$\updelta_0(A_\sigma,L_\sigma)=+\infty$. Nous verrons en fait bient\^ot
qu'une quantit\'e $\updelta(A_\sigma,L_\sigma)$, toujours finie et plus
petite que $\updelta_0(A_\sigma,L_\sigma)$, peut intervenir. La formulation
suivante, forme simplifi\'ee du r\'esultat principal de cet article,
est valable pour les deux variantes.

\begin{theo}\label{perint} Si $(A,L)$ est une vari\'et\'e ab\'elienne
polaris\'ee de dimension $g$ sur un corps de nombres $k$ nous
avons $$\frac{1}{[k:\mathbf{Q}]}\sum_{\sigma\colon k\hookrightarrow\mathbf{C}}
{(\deg_LA)^{1/g}\over\updelta(A_\sigma,L_\sigma)^2}\le50g^{2g+6}
\max(1,h_F(A),\log\deg_LA).$$\end{theo}

Un autre avantage, un peu plus technique, de notre pr\'esentation, est
de mettre en lumi\`ere le r\^ole des sous-vari\'et\'es ab\'eliennes
auxiliaires qui interviennent dans la d\'emonstration. On s'aper\c
coit en effet que la condition sur $\omega$ n'est utilis\'ee que pour
une vari\'et\'e en particulier. Ceci nous conduit \`a affiner la
d\'efinition en introduisant pour une sous-vari\'et\'e ab\'elienne
$\mathsf B$ de $A_\sigma$ le \emph{minimum d'\'evitement} de $\mathsf B$ not\'e
$\updelta(A_\sigma,L_\sigma,\mathsf B)$ : la plus petite distance non
nulle d'une p\'eriode de $A_\sigma$ \`a l'espace tangent de $\mathsf
B$. Nous \'ecrivons alors la preuve avec cette quantit\'e et le
th\'eor\`eme ci-dessus d\'ecoule d'un choix particulier de $\mathsf B$
en chaque place (techniquement celui qui minimise la quantit\'e
$(\deg_{L_\sigma}\mathsf B/\deg_LA)^{1/\codim\mathsf B}$). M\^eme si
nous n'avons pas d'application pour d'autres choix de $\mathsf B$, il nous
semble tout de m\^eme plus int\'eressant d'\'ecrire la majoration sous
cette forme (voir le th\'eor\`eme~\ref{thmimportant}) : d'une part
cela renforce encore les liens avec le lemme matriciel (dont le
minimum absolu correspond maintenant simplement au choix de $\sf B=0$),
ensuite nous ne manipulons pas de quantit\'e infinie et donc nous
obtenons toujours un r\'esultat m\^eme si $A$ ne s'\'ecrit pas sous la
forme $(A')_{\omega'}$ et enfin cela affine le
th\'eor\`eme~\ref{perint} : il est valable avec
$$\updelta(A_\sigma,L_\sigma)=\sup_\mathsf B\updelta(A_\sigma,L_\sigma,\mathsf
B)$$ o\`u, comme ci-dessus, $\mathsf B$ parcourt les sous-vari\'et\'es
ab\'eliennes strictes de $A_\sigma$. Accessoirement la quantit\'e
$\updelta(A_\sigma,L_\sigma)$ (toujours finie) est plus facile \`a majorer que
$\updelta_0(A_\sigma,L_\sigma)$ (lorsque celle-ci est finie, voir
proposition~\ref{majmu}).

Disons enfin qu'il est un cas o\`u th\'eor\`eme des p\'eriodes et
lemme matriciel deviennent identiques~: c'est celui des courbes
elliptiques. En effet on a toujours $\rho=\updelta$
(autrement dit seul $\sf B=0$ intervient) et comme toute polarisation est
puissance de la polarisation principale le degr\'e n'intervient pas
(voir aussi le paragraphe~\ref{redell}).

Nous pouvons maintenant d\'eduire du th\'eor\`eme~\ref{perint} un
\'enonc\'e ayant la forme de celui de Masser et W\"ustholz. Nous
revenons pour cela au cadre o\`u $\omega$ est une p\'eriode, pour un
plongement fix\'e $\sigma_0$, de la vari\'et\'e ab\'elienne $A$ et
nous appliquons notre th\'eor\`eme \`a la vari\'et\'e ab\'elienne
$A_\omega$. Nous en d\'eduisons alors facilement le th\'eor\`eme
suivant, o\`u l'on note $k'$ une extension de $k$
sur laquelle est d\'efinie $A_\omega$ ; on sait que l'on peut choisir
$[k':k]\le3^{16g^4}$.

\begin{theo}\label{thmintro} Si $\omega\ne0$, nous avons
$$(\deg_LA_{\omega})^{1/\dim A_{\omega}}\le
195g^{2g+9}[k':\mathbf{Q}]\|\omega\|_{L,\sigma_{0}}^{2}\max{(1,h_F(A),\log
[k':\mathbf{Q}]\|\omega\|_{L,\sigma_{0}}^{2})}.$$\end{theo}

Il convient de signaler qu'\`a l'occasion d'un
cours donn\'e \`a une \'ecole d'\'et\'e en 2009 \`a
Rennes~\cite{drennes}, David a pr{\'e}sent{\'e} une version de cet
\'enonc\'e dans le cas d'une polarisation principale et sans
expliciter la d\'ependance en $g$. En particulier, on lui doit le
premier r\'esultat avec une constante $\kappa$ optimale
(en rempla\c cant $L$ par une puissance dans le
th\'eor\`eme~\ref{thmintro} on voit que $\kappa<\dim A_\omega$ est
impossible).

Comme dernier th\`eme abord\'e dans cet article, nous nous
int\'eressons \`a l'application du th\'eor\`eme des p\'eriodes aux
th\'eor\`emes d'isog\'enie. Nous nous limitons ici au cas
elliptique. Le probl\`eme consiste alors, \'etant donn\'e deux
courbes elliptiques $E_1$ et $E_2$ isog\`enes, toutes deux d\'efinies
sur un corps de nombres $k$, \`a majorer le degr\'e minimal d'une
isog\'enie entre $E_1$ et $E_2$. On peut faire remonter cette question
aux travaux des fr\`eres Chudnovsky~\cite{chud} (cas d'un corps r\'eel
; on consultera \`a ce sujet l'historique pr\'esent\'e par Pellarin
dans~\cite{torino}) mais elle trouve toute son importance depuis
l'article de Masser et W\"ustholz~\cite{masserwustholz1990} qui ont
donn\'e une borne de la forme $c\max(1,h_F(E_1))^4$ pour une constante
$c$ non explicit\'ee. David~\cite{david} puis Pellarin~\cite{pellarin}
ont obtenu les premiers r\'esultats explicites. Ce dernier
d\'emontre l'existence d'une isog\'enie de degr\'e au plus
$$10^{78}[k:\mathbf{Q}]^4\max(\log[k:\mathbf{Q}],1)^2\max(h_F(E_1),1)^2.$$

Nous am\'eliorons ici \`a la fois l'exposant du degr\'e et la
constante num\'erique. \'Etant donn\'e un corps
$k$, on note $\overline{k}$ une cl\^oture alg\'ebrique de $k$.

\begin{theo}\label{isoprin} Soient $k$ un corps de nombres, $E_1$ et
$E_2$ deux courbes elliptiques d\'efinies sur $k$. Si $E_1$ et $E_2$
sont isog\`enes (sur $\overline k$), il existe une isog\'enie entre
elles (sur $\overline k$) de degr\'e au
plus $$10^7[k:\mathbf{Q}]^2\left(\max(h_F(E_1),985)+4\log[k:\mathbf{Q}]\right)^2$$
ce que l'on peut majorer
par $$10^{13}[k:\mathbf{Q}]^2\max(h_F(E_1),\log[k:\mathbf{Q}],1)^2.$$ Lorsque
$E_1$ (et donc $E_2$) admet des multiplications complexes, la borne
ci-dessus peut \^etre remplac\'ee par
$$3,4\times10^4[k:\mathbf{Q}]^2\max\left(h_F(E_1)+\frac{1}{2}\log[k:\mathbf{Q}],1\right)^2.$$
Si $E_1$ et $E_2$ n'ont pas de multiplications complexes et si $k$ a
une place r\'eelle, elle peut \^etre remplac\'ee par
$$3583[k:\mathbf{Q}]^2\max\left(h_F(E_1),\log[k:\mathbf{Q}],1\right)^2.$$
\end{theo}

Dans le cas g\'en\'eral, ce th\'eor\`eme s'obtient en appliquant le
th\'eor\`eme des p\'eriodes~\ref{perint} \`a la vari\'et\'e
ab\'elienne $E_1^2\times E_2^2$ tandis que, pour les deux derni\`eres
bornes, les hypoth\`eses suppl\'ementaires permettent d'utiliser
$A=E_1\times E_2$ : dans ce cas, $A_\omega$ est une courbe elliptique
et le th\'eor\`eme des p\'eriodes se r\'eduit \`a un lemme matriciel
(de la forme du th\'eor\`eme~\ref{matint}).

Avec \cite{parentbilu,parentbilurebolledo} le
th\'eor\`eme~\ref{isoprin} s'applique au probl\`eme d'uniformit\'e de
Serre~\cite{Serredeux} (ci-dessous $E[p]$ d{\'e}signe le groupe
des points de $p$-torsion de la courbe $E$).

\begin{coro}\label{pabi} Pour tout nombre premier $p>p_0=3,1\times10^6$
et toute courbe elliptique $E$ d\'efinie sur $\mathbf{Q}$ sans
multiplications complexes, l'image de la repr\'esentation galoisienne
naturelle $\rho_{E,p}\colon\gal(\overline{\mathbf{Q}}/\mathbf{Q})\to\GL(E[p])$
n'est pas contenue dans le normalisateur d'un sous-groupe de Cartan
d\'eploy\'e.\end{coro}

Signalons qu'\`a partir de ce r\'esultat et avec des calculs
informatiques pour les petits premiers ($p<p_0$), Bilu, Parent et
Rebolledo \cite{parentbilurebolledo} montrent que l'\'enonc\'e
pr\'ec\'edent vaut en fait pour tout $p\not\in\{2,3,5,7,13\}$. Nous
renvoyons \`a leur texte pour les d\'etails.
\vskip3mm
{\bf Remerciements.}
Nous remercions Yuri Bilu, Pierre Parent et Marusia Rebolledo pour
nous avoir signal\'e l'application des th\'eor\`emes d'isog\'enies au
probl\`eme de Serre. Leur commande fut notre principale motivation
pour obtenir des constantes num\'eriques aussi petites que possible.
Nous remercions Pascal Autissier pour ses remarques sur une premi\`ere
version de ce texte et pour nous avoir communiqu\'e son remarquable lemme matriciel. Nous remercions aussi Mathilde Herblot
et Guillaume Maurin de nous avoir fourni leurs notes du cours de
Sinnou David \cite{drennes}.

\tableofcontents

Passons \`a pr\'esent en revue rapidement les ingr\'edients principaux
de notre preuve et les travaux dont elle s'inspire.
Les m\'ethodes que nous employons remontent pour une grande part
au s\'eminaire Bourbaki~\cite{bost2} dans lequel
Bost a jet\'e un nouvel \'eclairage sur le th\'eor{\`e}me des
p\'eriodes en introduisant la \emph{m\'ethode des pentes}. Ce
travail, rendu un peu plus explicite par Viada~\cite{viada}, apporte
l'effectivit\'e de la constante $c_0$ dans le th\'eor\`eme de Masser
et W\"ustholz. Il a eu \'egalement \'enorm\'ement d'impact sur la
mani{\`e}re de pr\'esenter la d\'emonstration en conservant au maximum
l'aspect intrins{\`e}que des donn\'ees. Il a aussi ouvert un champ
d'application naturel {\`a} la g\'eom\'etrie d'Arakelov. Par exemple,
les m\'ethodes de Bost ont permis au premier auteur d'obtenir des
minorations de formes lin\'eaires de logarithmes de vari\'et\'es
ab\'eliennes, totalement explicites, pour des logarithmes qui ne
sont pas des p\'eriodes (en un sens assez fort)~\cite{gaudrondeux}.
D'un autre c\^ot\'e, la preuve de David est extraite de la
d\'emonstration g\'en\'erale pour les formes lin\'eaires de
logarithmes (m\'ethode de Philippon et
Waldschmidt~\cite{pph-miw}). {\`A} cela rien de surprenant puisque
nous sommes dans les m\^emes conditions~: on dispose d'un
logarithme $\omega$ d'un point alg\'ebrique $0_{A}\in A(k)$ et d'un
sous-espace vectoriel $t_{A_{\omega}}$ de $t_{A}$. Comme $\omega\in
t_{(A_{\omega})_{\sigma_{0}}}$, nous sommes dans le \og{}cas
d\'eg\'en\'er\'e\fg{} o{\`u} le logarithme appartient au
sous-espace. La d\'emonstration de Philippon et Waldschmidt
fonctionne encore dans ce cas mais au lieu de fournir une minoration
de la distance du logarithme au sous-espace, elle montre l'existence
d'une sous-vari\'et\'e ab\'elienne stricte $B$ de $A$ avec
$\omega\in t_{B_{\sigma_{0}}}$ et $\deg_{L}B$ major\'e
essentiellement comme dans le th\'eor{\`e}me de David (des
bornes pour $\deg_{L}B$ se trouvent par exemple
dans~\cite{david,gaudronun,villanithese}). Pour assurer
$B=A_\omega$, l'id\'ee de David est de travailler avec
$A_\omega$ d\`es le d\'epart. La sous-vari\'et\'e $B$ ne
peut pas exister (par minimalit\'e de $A_\omega$ relativement
\`a l'hypoth\`ese $\omega\in t_{(A_{\omega})_{\sigma_{0}}}$) mais
la d\'emonstration donne malgr\'e tout quelque chose, \`a savoir
une majoration de $\deg_LA_\omega$. Cette observation a permis
\`a David d'obtenir en une seule \'etape la sous-vari\'et\'e
$A_\omega$, sans avoir \`a faire de r\'ecurrence sur $g$,
r\'ecurrence tr\`es co\^uteuse pour les constantes $c$ et
$\kappa$ et qui explique leur caract\`ere exponentiel en $g$
chez Masser et W\"ustholz.\par 
Si notre borne du degr\'e de $A_{\omega}$ est proche de celle de David, la
d\'emonstration n'utilise pas les m{\^{e}}mes outils. Elle s'inscrit
encore dans le sch\'ema g\'en\'eral de la m\'ethode de
Philippon et Waldschmidt (cas p\'eriodique) mais elle utilise
largement le formalisme des pentes de Bost, comme
dans~\cite{gaudrondeux}. Toutefois il n'y a pas de m\'ethode des
pentes proprement dite. Cette derni{\`e}re est remplac\'ee par la
\emph{m\'ethode de la section auxiliaire} que le premier auteur a
introduite dans~\cite{gaudronhuit}. Il s'agit d'une variante
intrins{\`e}que de la m\'ethode classique des fonctions auxiliaires
en transcendance. Ici l'adjectif intrins{\`e}que signifie
essentiellement que nous n'aurons recours ni {\`a} une base explicite
des fonctions th\^eta de $\mathrm{H}^{0}(A,L^{\otimes n})$, ni
\`a une base de Shimura de l'espace tangent $t_{A}$. Outre la
clart\'e apport\'ee par cette approche g\'eom\'etrique, la
d\'emonstration met en \'evidence l'int\'egralit\'e des jets
de sections qui apparaissent. Cet avantage tactique autorise un
param\`etre \`a tendre vers $+\infty$ (ce qui est exceptionnel
dans une preuve de transcendance) en \'eliminant au passage
plusieurs quantit\'es parasites. Un autre atout de ce passage {\`a}
la limite est la diminution des constantes num\'eriques.\par En ce qui
concerne notre th\'eor\`eme d'isog\'enie, l'aspect
intrins\`eque du th\'eor\`eme des p\'eriodes sur lequel il
s'appuie \'evite naturellement le recours \`a des mod\`eles de
Weierstrass des courbes elliptiques (et donc \`a la notion
d'isog\'enie normalis\'ee) qui apparaissaient dans les travaux
ant\'erieurs. Dans le m\^eme ordre d'id\'ee, Pellarin devait
consid\'erer des sous-vari\'et\'es ab\'eliennes \emph{exceptionnelles}
et exclure un {\em cas d\'eg\'en\'er\'e}~\cite[hypoth\`ese 3 page
212]{pellarin}. Nous avons simplifi\'e l'analyse en montrant que ces
subtilit\'es n'ont plus lieu d'\^etre et que la seule consid\'eration
de $A_{\omega}$ suffit \`a extraire l'information sur le degr\'e
d'isog\'enie (voir th\'eor\`eme~\ref{lien}).

\section{Pr\'eliminaires}

\subsection{Polarisation} Lorsque $A$ est une vari\'et\'e ab\'elienne,
nous rappelons qu'une \emph{polarisation} sur $A$ est l'image d'un
faisceau inversible ample dans le groupe de N\'eron-Severi
$\NS(A)=\Pic(A)/\Pic^{0}(A)$. C'est
cette notion qui intervient la plupart du temps dans cet article : par
exemple le degr\'e $\deg_{L}A$ ou la forme de Riemann d'un faisceau
inversible ample $L$ ne d\'ependent que de la polarisation d\'efinie
par $L$. Lorsque nous souhaitons parler d'un faisceau repr\'esentant
la polarisation nous en choisissons toujours un sym\'etrique. Ceci
n'induit qu'une ind\'etermination finie car un \'el\'ement
sym\'etrique de $\Pic^{0}(A)$ est un \'el\'ement de $2$-torsion. En
particulier si $L$ sym\'etrique repr\'esente une polarisation alors
$L^{\otimes 2}$ est uniquement d\'efini. Rappelons aussi que sur une courbe
elliptique il existe une unique polarisation principale et toute
polarisation en est une puissance (car $\NS(A)=\mathbf{Z}$).

\subsection{Vari{\'e}t{\'e} ab{\'e}lienne orthogonale}
\label{subsecvaortho} Soit $A$ une vari{\'e}t{\'e}
ab{\'e}lienne sur un corps quelconque, munie d'une polarisation $L$. Soit
$B$ une sous-vari{\'e}t{\'e} ab{\'e}lienne de $A$. La
sous-vari{\'e}t{\'e} ab{\'e}lienne orthogonale $B^{\perp}$ de $B$ dans
$A$ est d{\'e}finie de la mani{\`e}re suivante~: soit
$\varphi_{L}:A\to\widehat{A}$ l'isog{\'e}nie dans la vari{\'e}t{\'e}
duale $\widehat{A}$ induite par $L$. Soit
${}^{\mathrm{t}}i:\widehat{A}\to\widehat{B}$ le morphisme dual {\`a}
l'inclusion $i:B\hookrightarrow A$. Alors $B^{\perp}$ est la
composante neutre du noyau de la compos{\'e}e
${}^{\mathrm{t}}i\circ\varphi_{L}$. On montre alors que le morphisme
d'addition $B\times B^{\perp}\to A$ est une isog{\'e}nie de degr{\'e}
$\mathsf{b}$ au plus $$\frac{\mathrm{h}^{0}(B,L)\mathrm{h}^{0}
(B^{\perp},L)}{\mathrm{h}^{0}(A,L)}\le\mathrm{h}^{0}(B,L)^{2}$$(voir
par exemple~\cite[th{\'e}or{\`e}me~$3$]{Bertrand}). De plus, si le
corps de base est $\mathbf{C}$, pour toute p{\'e}riode
$\omega\in\Omega_{A}$, il existe $\omega_{1}\in\Omega_{B}$ et
$\omega_{2}\in\Omega_{B^{\perp}}$ tels
que $\mathsf{b}\omega=\omega_{1}+\omega_{2}$ (voir lemme~$1.4$
de~\cite{masserwustholz1993}).

\subsection{Hauteur de Faltings}
Lorsque $A$ est une vari\'et\'e ab\'elienne d{\'e}finie sur un corps
de nombres $k$, nous d{\'e}finissons sa hauteur $h(A)$ de la
mani{\`e}re suivante~: soit $K$ une extension finie de $k$ sur
laquelle $A$ est d{\'e}finie et admet r{\'e}duction
semi-ab{\'e}lienne. Soient $\pi:\mathcal{A}\to\spec\mathcal{O}_{K}$ un
mod{\`e}le semi-ab{\'e}lien de $A$ et
$\epsilon:\spec\mathcal{O}_{K}\to\mathcal{A}$ sa section
nulle. Notons $\omega_{\mathcal{A}/\mathcal{O}_{K}}$ le faisceau
inversible
$\epsilon^{*}\Omega_{\mathcal{A}/\spec\mathcal{O}_{K}}^{g}$ sur
$\spec\mathcal{O}_{K}$. Ce fibr{\'e} devient un fibr{\'e} en droites
hermitien $\overline{\omega}_{\mathcal{A}/\mathcal{O}_{K}}$ sur
$\spec\mathcal{O}_{K}$ lorsqu'on le munit pour chaque plongement
complexe $\sigma\colon K\hookrightarrow\mathbf{C}$ de la
norme \begin{equation*}\forall
s\in\omega_{\mathcal{A}/\mathcal{O}_{K}}\otimes_{\sigma}\mathbf{C}\simeq
\mathrm{H}^0(A_\sigma,\Omega_{A_\sigma}^g),\quad
\Vert s\Vert_{\overline{\omega}_{\mathcal{A}/\mathcal{O}_{K}},\sigma}^2:=
\frac1{(2\pi)^g}\int_{A_\sigma}{\vert
s\wedge\overline{s}\vert}.\end{equation*}\begin{defi} La hauteur
$h(A)$ de $A$ est le degr{\'e} d'Arakelov normalis{\'e} de
$\overline{\omega}_{\mathcal{A}/\mathcal{O}_{K}}$.\end{defi} Cette
d{\'e}finition est ind{\'e}pendante des choix de $K$ et de
$\mathcal{A}$. Cette quantit\'e $h(A)$ est celle d\'enomm\'ee
\emph{hauteur de Faltings}
dans~\cite{bost2,bost4,gaudrondeux,graftieaux} mais ce
n'est pas la convention adopt\'ee par tous les auteurs. En particulier
la
d\'efinition originale $h_{F}(A)$ de Faltings~\cite{faltings} ne fait
pas appara{\^\i}tre le $\pi$
dans la d\'efinition de la norme ci-dessus et donc on
a $$h_F(A)=h(A)-\frac{g}{2}\log\pi.$$
La hauteur $h(A)$ est donc plus grande que la hauteur de Faltings
originale. Dans la suite, nous employons $h(A)$ mais nous avons
pr\'ef\'er\'e utiliser $h_F(A)$ dans l'introduction pour faciliter
l'emploi de nos \'enonc\'es. Ce choix a l'avantage que les \'enonc\'es
(majorations) sont vrais pour les deux hauteurs. D'autres auteurs
utilisent encore une notion diff\'erente. Par exemple
Colmez~\cite{colmez} travaille avec la hauteur
$h(A)-(g/2)\log2\pi$. Quelle que soit la normalisation, rappelons que
cette hauteur satisfait au th{\'e}or{\`e}me de
Faltings~\cite{faltings}: si $\varphi:A\to A'$ est une isog{\'e}nie
alors \begin{equation*}\label{raynaud2006}h(A')\le
h(A)+\frac{1}{2}\log\deg\varphi\end{equation*}ainsi qu'aux
propri{\'e}t{\'e}s $h(A_{1}\times A_{2})=h(A_{1})+h(A_{2})$ et
$h(\widehat{A})=h(A)$ (corollaire~$2.1.3$ de~\cite{raynaud}). De plus
la hauteur d'une sous-vari\'et\'e ab\'elienne $B$ de $A$ est
contr{\^{o}}l{\'e}e par celle de $A$~:
\begin{equation*}\label{sva} h(B)\le
h(A)+g\log(\sqrt{2\pi}\mathrm
h^0(B,L)^{2})\end{equation*}(voir~\cite[proposition~$4.9$]{gaudrondeux}).

\subsection{Forme de Riemann}\label{formeR}

Soit $\mathsf A$ une vari\'et\'e ab\'elienne complexe. D'apr\`es le
th\'eor\`eme d'Appell-Humbert (voir~\cite[p. 20]{mumford}
ou~\cite[p. 32]{lange}), le groupe de Picard $\mathrm{Pic}(\mathsf A)$
s'identifie au groupe des couples $(H,\chi)$ o\`u $H$ est une forme
hermitienne (lin{\'e}aire {\`a} droite) sur $t_{\sf A}$ telle que $\IM H(\Omega_{\sf
A},\Omega_{\sf A})\subset\mathbf{Z}$ et $\chi$ une application $\Omega_{\sf
A}\to\{z\in\mathbf{C}\mid|z|=1\}$ telle que
$\chi(\omega_1+\omega_2)\chi(\omega_1)^{-1}\chi(\omega_2)^{-1}=\exp(i\pi\IM
H(\omega_1,\omega_2))$ pour tous $\omega_1,\omega_2\in\Omega_{\sf
A}$. Lorsqu'un tel couple correspond \`a
$\mathsf{L}\in\mathrm{Pic}(\sf A)$, nous dirons que $(H,\chi)$ est la
donn\'ee d'Appell-Humbert de $\sf L$ et la premi\`ere composante $H$
s'appelle la {\em forme de Riemann} de $\sf L$. Celle-ci ne d\'epend
que de l'image de $\sf L$ dans $\rm{NS}(\sf A)$.

La forme de Riemann de $\sf L$ est d\'efinie positive si et seulement
si $\sf L$ est ample (autrement dit si $\sf L$ d\'efinit une polarisation ;
certains auteurs r\'eservent l'emploi du terme forme de Riemann \`a ce
cas). Ainsi une polarisation $\sf L$ permet de munir l'espace tangent
$t_{\sf A}$ d'une norme hermitienne not\'ee $\|.\|_{\mathsf{L}}$ : on
pose simplement $\|z\|^2_{\sf L}=H(z,z)$ pour $z\in t_{\sf A}$ lorsque
$H$ est la forme de Riemann de $\sf L$. C'est la norme utilis\'ee dans
l'introduction et dans toute la suite de ce texte. Elle permet par
exemple de d\'efinir le minimum du r\'eseau des p\'eriodes d\'ej\`a
rencontr\'e et qui fera l'objet des lemmes matriciels de la partie
suivante : $$\rho{(\sf A,L)}=\min\{\|\omega\|_{\sf L}\,;\, \omega\in\Omega_{\sf
A}\setminus\{0\}\}.$$Ce nombre r{\'e}el se rencontre aussi dans la
litt\'erature sous l'appellation \emph{diam\`etre d'injectivit\'e} car
c'est le diam\`etre de la plus grande boule sur laquelle l'exponentielle
$\exp_{\sf A}\colon t_{\mathsf{A}}\to\mathsf{A}$ est injective.

Lorsque $(A,L)$ est une vari\'et\'e ab\'elienne polaris\'ee sur un
corps de nombres $k$ et $\sigma\colon k\hookrightarrow\mathbf{C}$ un
plongement, nous noterons $\|\cdot\|_{L,\sigma}$ la norme induite par
$L_\sigma$ (au lieu de $\|\cdot\|_{L_\sigma}$).

\subsection{Fonctions th\^eta}\label{fnth} Soit $\sf L$ un faisceau
inversible sur une vari\'et\'e ab\'elienne complexe $\sf A$. On
d\'efinit son facteur d'automorphie canonique $a_{\sf
L}\colon\Omega_{\sf A}\times t_{\sf A}\to\mathbf{C}$ \`a l'aide de sa donn\'ee
d'Appell-Humbert $(H,\chi)$ : si $\omega\in\Omega_{\sf A}$ et $z\in
t_{\sf A}$ on pose $$a_{\sf L}(\omega,z)=\chi(\omega)\exp{\left(\pi
H(\omega,z)+\frac\pi2H(\omega,\omega)\right)}.$$
Ce facteur permet de d\'efinir les {\em fonctions th\^eta} associ\'ees
\`a $\sf L$ : ce sont les fonctions holomorphes
$\vartheta\colon t_{\sf A}\to\mathbf{C}$ qui v\'erifient
$\vartheta(z+\omega)=a_{\sf L}(\omega,z)\vartheta(z)$ pour tous
$\omega\in\Omega_{\sf A}$ et $z\in t_{\sf A}$. Elles trouvent leur
raison d'\^etre dans l'isomorphisme naturel entre $\mathrm{H}^0(\sf
A,L)$ et l'espace vectoriel des fonctions th\^eta associ\'ees \`a $\sf
L$ (voir~\cite[p. 25]{mumford}).

En particulier, lorsque $\sf L$ est tr\`es ample, elles fournissent
une \'ecriture explicite d'un plongement projectif associ\'e \`a $\sf
L$ : si $\vartheta_0,\ldots,\vartheta_m$ est une base des fonctions
th\^eta alors l'application
$z\mapsto(\vartheta_0(z):\cdots:\vartheta_m(z))$ d\'efinit un
morphisme $t_{\sf A}\to\mathbf{P}^m_\mathbf{C}$ qui se factorise \`a travers
$\exp_{\sf A}$ pour donner une immersion $\sf A\hookrightarrow\mathbf{P}^m_\mathbf{C}$.

\subsection{Changement de base par la conjugaison complexe}
\label{subsecconju}
Soit \`a nouveau une vari\'et\'e ab\'elienne complexe $\sf A$.
Notons $\tau$ la conjugaison complexe.
On d\'efinit $\overline{\sf A}$ par le carr\'e cart\'esien :
$$\begin{array}{ccc}\overline{\sf A}&
\mathop{\longrightarrow}\limits^{\sf f}&\sf A\\
\Big\downarrow&\Box&\Big\downarrow\\\spec\mathbf{C}&
\mathop{\longrightarrow}\limits^{\spec\tau}&\spec\mathbf{C}\rlap{\,.}\end{array}$$
Nous obtenons une vari\'et\'e ab\'elienne complexe mais il faut
prendre garde au fait que le morphisme de sch\'emas $\sf f$ n'est pas
un morphisme de $\mathbf{C}$-sch\'emas. C'est en revanche un morphisme de
$\mathbf{R}$-sch\'emas (entre $\mathbf{R}$-sch\'emas de dimension $2\dim\sf A$) que
l'on peut \'egalement voir comme un morphisme de vari\'et\'es
analytiques r\'eelles entre les tores $t_{\overline{\sf
A}}/\Omega_{\overline{\sf A}}$ et $t_{\sf A}/\Omega_{\sf A}$.

\begin{prop}\label{conjugue} L'isomorphisme $\sf f$ se rel\`eve en un
isomorphisme antilin\'eaire ${\rm d\sf f}\colon t_{\overline{\sf
A}}\to t_{\sf A}$ tel que $\rm d\sf f(\Omega_{\overline{\sf A}})
=\Omega_{\sf A}$. De plus si l'on m\'etrise les espaces tangents par
les formes des Riemann de $\sf L$ et $\sf f^*L$ alors $\rm d\sf f$ est
une isom\'etrie.\end{prop}

\begin{proof} Il n'y a pas de restriction \`a supposer que $\sf L$ est
tr\`es ample. Notons alors comme plus haut
$\vartheta_0,\ldots,\vartheta_m$ une base des fonctions th\^eta
associ\'ees. Nous d\'esignons par $V$ l'espace vectoriel complexe conjugu\'e
de $t_{\sf A}$ : le groupe ab\'elien sous-jacent est $V=t_{\sf A}$ mais la loi
$\bullet$ de $V$ de multiplication par un scalaire est donn\'ee par
$z\bullet v=\overline zv$ pour
$z\in\mathbf{C}$ et $v\in V$ (o\`u, \`a droite, on utilise la loi usuelle de
$t_\mathsf{A}$). Pour clarifier nous notons aussi $U$ le r\'eseau $\Omega_{\sf
A}$ lorsque nous le voyons comme r\'eseau de $V$. Ainsi $V/U$ est un
tore complexe. En outre les fonctions $\overline{\vartheta_j}$ pour
$0\le j\le m$ sont {\em holomorphes} sur $V$. La forme $\overline H$
est quant \`a elle une forme hermitienne d\'efinie positive sur
$V\times V$ et elle v\'erifie $\IM\overline H(U,U)\subset\mathbf{Z}$. De
m\^eme l'application $\overline\chi\colon U\to\{z\in\mathbf{C}\mid|z|=1\}$ satisfait
$\overline\chi(u_1+u_2)\overline\chi(u_1)^{-1}\overline\chi(u_2)^{-1}
=\exp(i\pi\IM\overline H(u_1,u_2))$ pour tous $u_1,u_2\in U$. Tout
ceci nous montre que $V/U$ est une vari\'et\'e ab\'elienne, que
$(\overline H,\overline\chi)$ est une donn\'ee d'Appell-Humbert sur
celle-ci et que $\overline\vartheta_0,\ldots,\overline\vartheta_m$
forment une base des fonctions th\^eta associ\'ees \`a cette
donn\'ee. En particulier elles d\'efinissent un plongement projectif
$p\colon V/U\hookrightarrow\mathbf{P}^m_\mathbf{C}$. Enfin appelons $q$
l'application $V/U\to\sf A$ induite par l'identit\'e $V\to t_{\sf A}$
(qui est antilin\'eaire). En suivant les constructions, nous avons
alors un diagramme commutatif : $$\begin{array}{*5c}V/U&
\mathop{\longrightarrow}\limits^{p}&\mathbf{P}^m_\mathbf{C}&\longrightarrow&\spec\mathbf{C}\\
\Big\downarrow\rlap{$\vcenter{\hbox{$\scriptstyle q$}}$}&&\Big\downarrow&\Box
&\Big\downarrow\rlap{$\vcenter{\hbox{$\scriptstyle\tau$}}$}\\
\sf A&\longrightarrow&\mathbf{P}^m_\mathbf{C}&\longrightarrow&\spec\mathbf{C}\rlap{\,.}\end{array}$$
Comme les fl\`eches verticales sont des isomorphismes, le carr\'e de
gauche est automatiquement cart\'esien et nous pouvons donc identifier
$V/U$ \`a $\overline{\mathsf{A}}$ et $q$ \`a $\mathsf{f}$. Dans cette
identification $V=t_{\overline{\mathsf{A}}}$ et $\mathrm{d}\mathsf{f}$
correspond \`a l'identit\'e $V\to t_{\sf A}$. Le diagramme montre
encore que $\sf f^*L$ co\"\i ncide avec $p^*\mathcal O(1)$ et a donc
pour forme de Riemann $\overline H$. Toutes les assertions de
l'\'enonc\'e d\'ecoulent imm\'ediatement de ces faits.\end{proof}

\`A titre d'exemple nous avons donc $\rho(\overline{\sf A},\sf f^*\sf
L)=\rho(\sf A,\sf L)$.

Dans le cas o\`u $(A,L)$ est une vari\'et\'e ab\'elienne polaris\'ee
sur un corps de nombres $k$ et $\sigma\colon k\hookrightarrow\mathbf{C}$ un
plongement, nous notons $\overline\sigma=\tau\circ\sigma$. Alors on a
$\overline{A_\sigma}=A_{\overline\sigma}$ et
$L_{\overline\sigma}=\mathsf{f}^{*}L_\sigma$. Avec la proposition ceci nous
permet de v\'erifier que les minima associ\'es aux couples
$(A_\sigma,L_\sigma)$ et $(A_{\overline\sigma},L_{\overline\sigma})$
co{\"\i}ncident.

Bien entendu, tous les faits de ce paragraphe sont {\em faux} pour un
automorphisme $\mathbf{C}\to\mathbf{C}$ autre que $\tau$ ou $\rm
id_\mathbf{C}$ (et donc non continu) et il n'y a
aucune relation en g\'en\'eral entre les minima de
$(A_\sigma,L_\sigma)$ et $(A_{\sigma'},L_{\sigma'})$ pour deux
plongements $\sigma$ et $\sigma'$ distincts et non conjugu\'es.

\section{Autour du lemme matriciel}

Dans cette partie nous donnons plusieurs versions du lemme matriciel
au sens donn\'e plus haut. Elles d{\'e}coulent toutes d'un nouvel
{\'e}nonc{\'e} d{\^{u}} {\`a} Autissier~\cite{autissier}. Notre
motivation est multiple : d'une part elles am\'eliorent les constantes
donn\'ees dans \cite{sinnoupph,gaudrondeux,graftieaux}; d'autre part
nous \'ecrivons le r\'esultat sans hypoth\`ese de polarisation
principale contrairement \`a ces textes. Ensuite nous nous
int{\'e}ressons plus particuli{\`e}rement au cas de la dimension
$g=1$ : il s'agit ici v\'eritablement d'un th\'eor\`eme des p\'eriodes
donc c'est une partie de la d\'emonstration de notre th\'eor\`eme
principal. Par ailleurs, l'obtention de bonnes constantes dans ce cas
nous permettra aussi d'\^etre plus efficace dans l'application aux
th\'eor\`emes d'isog\'enies de courbes elliptiques. Dans un second
temps, nous {\'e}non{\c{c}}ons des majorations de pentes maximales
dues {\`a} Graftieaux, qui reposent elles-m{\^{e}}mes sur des lemmes
matriciels. 

\subsection{Th{\'e}or{\`e}me d'Autissier et cons{\'e}quences}
Commen{\c{c}}ons par {\'e}noncer le lemme matriciel
d'Autissier~\cite{autissier} (voir le \S~\ref{formeR} pour la
d{\'e}finition de $\rho(A_{\sigma},L_{\sigma})$). 

\begin{theo}\label{thmpascal} Soit $(A,L)$ une vari{\'e}t{\'e}
ab{\'e}lienne principalement polaris{\'e}e, d{\'e}finie sur un corps
de nombres $k$. Pour tout plongement complexe $\sigma\colon
k\hookrightarrow\mathbf{C}$, notons
$\rho_{\sigma}:=\min{(\rho(A_{\sigma},L_{\sigma}),\sqrt{\pi/3g})}$.
Alors on
a $$\frac{1}{[k:\mathbf{Q}]}\sum_{\sigma\colon
k\hookrightarrow\mathbf{C}}{\left(\frac{\pi}{6\rho_{\sigma}^{2}}
+g\log\rho_{\sigma}\right)}\le 
h(A)+\frac{g}{2}\log\frac{2\pi^{2}e}{3g}.$$\end{theo}

\subsubsection{} Donnons une premi{\`e}re cons{\'e}quence de ce
th{\'e}or{\`e}me pour les courbes elliptiques, qui nous servira plus
loin dans les estimations de degr\'es d'isog\'enies.

\begin{prop}\label{ell} Soit $A$ une courbe elliptique, munie de
sa polarisation principale $L$. Soit
$$T={1\over[k:\mathbf{Q}]}\sum_{\sigma\colon
k\hookrightarrow\mathbf{C}}{\rho(A_{\sigma},L_{\sigma})^{-2}}.$$ Alors
pour tout nombre r{\'e}el $\delta$ dans l'intervalle
$[3/\pi,\max{(T,3/\pi)}]$, on a $$\pi\delta\le
3\log\delta+6h(A)+8,66.$$ En particulier on a $T\le 6,45\max{(h(A),1)}$
et $T\le 1,92\max(h(A),1000)$. 
\end{prop}
\begin{proof}Appliquons le th{\'e}or{\`e}me~\ref{thmpascal} {\`a}
$(A,L)$. En {\'e}crivant
$\log\rho_{\sigma}=-(1/2)\log(1/\rho_{\sigma}^{2})$ et en utilisant la
concavit{\'e} du logarithme, on a $$\frac{\pi}{6}T'-\frac{1}{2}\log
T'\le h(A)+\frac{1}{2}\log\frac{2\pi^{2}e}{3}\quad\text{avec}\quad
T':=\frac{1}{[k:\mathbf{Q}]}\sum_{\sigma\colon
k\hookrightarrow\mathbf{C}}{\rho_{\sigma}^{-2}}.$$La premi{\`e}re
in{\'e}galit{\'e} de la proposition~\ref{ell} d{\'e}coule alors de
la croissance de la fonction $x\mapsto (\pi/6)x-(1/2)\log x$ pour
$x\ge3/\pi$, de l'encadrement $T'\ge\max{(T,3/\pi)}\ge\delta\ge
3/\pi$ et du calcul $3\log(2\pi^{2}e/3)\le 8,66$. En ce qui concerne
la premi\`ere majoration de $T$, on proc{\`e}de de la mani{\`e}re
suivante. Posons $Y=6,45$ et $Z=1$. Si $T\le Y$, l'in{\'e}galit{\'e}
est d{\'e}montr{\'e}e. Sinon, comme $Y\ge e$, on a $\log T\le T(\log
Y)/Y$ et donc, par la premi{\`e}re partie de la proposition avec
$\delta=T$,  $$T\le \frac{Y}{Z}\frac{6Z+8,66}{\pi Y-3\log
  Y}\max{(h(A),Z)}.$$On v{\'e}rifie num{\'e}riquement que
$6Z+8,66\le\pi Y-3\log Y$ et ceci donne le r{\'e}sultat. Pour la
derni\`ere majoration, on utilise le couple $(Y,Z)=(1920,1000)$. 
\end{proof}

\begin{rema}\label{yrho} Soit $\tau_{\sigma}$ l'{\'e}l{\'e}ment du
domaine fondamental de Siegel pour lequel la courbe elliptique
$A_{\sigma}$ est isomorphe {\`a}
$\mathbf{C}/(\mathbf{Z}+\tau_{\sigma}\mathbf{Z})$. La m{\'e}trique
sur $t_{A_{\sigma}}$ d{\'e}finie par la polarisation $L_{\sigma}$
correspond {\`a} la norme $\Vert z\Vert^2=\vert
z\vert^2/\IM\tau_{\sigma}$ pour $z\in\mathbf{C}$. Pour
$(a,b)\in\mathbf{Z}^{2}\setminus\{(0,0)\}$ on a 
$\|a+b\tau\|^2=|a+b\tau|^2/\IM\tau_{\sigma}\ge1/\IM\tau_{\sigma}$ avec \'egalit\'e
si $(a,b)=(1,0)$. On trouve ainsi
$\rho(A_\sigma,L_\sigma)^{-2}=\mathrm{Im}\tau_\sigma$. Par suite la
proposition~\ref{ell} peut \^etre utilis\'ee pour donner des 
estimations de $T=[k:\mathbf
Q]^{-1}\sum_\sigma\mathrm{Im}\tau_\sigma$.  
\end{rema}

\subsubsection{} Nous allons maintenant nous affranchir de l'hypoth\`ese de
polarisation principale du th{\'e}or{\`e}me~\ref{thmpascal} et en
donner une forme plus maniable. Nous \'etudions dans un premier temps
la variation de $\rho$ par isog\'enie.

\begin{lemma}\label{varrho} Soient $f\colon\mathsf{A}\to\mathsf{B}$
une isog\'enie entre vari\'et\'es ab\'eliennes complexes et
$\mathsf{L}$ une polarisation sur $\mathsf{B}$. Alors
$\rho(\mathsf{B},\mathsf{L})\le\rho(\mathsf{A},f^*\mathsf{L})\le(\deg
f)\rho(\mathsf{B},\mathsf{L})$.\end{lemma}

\begin{proof} L'application $f$ se rel\`eve en un isomorphisme
$\mathrm{d}f\colon t_{\mathsf{A}}\to
t_{\mathsf{B}}$ tel que
$\mathrm{d}f(\Omega_\mathsf{A})\subset\Omega_\mathsf{B}$. Comme le
conoyau de cette inclusion est de cardinal $\deg f$ on a aussi
$\Omega_\mathsf{B}\subset(\deg f)\mathrm{d}f(\Omega_\mathsf{A})$. Par
ailleurs la forme de Riemann de $f^*\mathsf{L}$ s'obtient en composant
la forme de Riemann de $\mathsf{L}$ avec $\mathrm{d}f$ donc pour tout
$x\in t_\mathsf{A}$ nous avons
$\|x\|_{f^*\mathsf{L}}=\Vert\mathrm{d}f(x)\Vert_{\mathsf{L}}$. Nous en
d\'eduisons $\rho(\mathsf{A},f^*\mathsf{L})=\min\{\Vert
x\Vert_{\mathsf{L}}\,;\,x\in\mathrm{d}f(\Omega_\mathsf{A})\setminus\{0\}\}$.
L'\'enonc\'e d\'ecoule alors imm\'ediatement des deux inclusions de
r\'eseaux ci-dessus.\end{proof}

Rappelons un lemme classique.

\begin{lemma}\label{quotpp} Soit $(A,L)$ une vari\'et\'e ab\'elienne
polaris\'ee sur un corps alg\'ebrique\-ment clos. Il existe une
isog\'enie $f\colon A\to B$ et une polarisation principale $M$ sur $B$
telles que $L=f^*M$ et $\deg f=\mathrm{h}^{0}(A,L)$.\end{lemma}

\begin{proof} On fait le quotient par un sous-groupe lagrangien de $K(L)$,
voir~\cite{mumford} pages 233--234.\end{proof}

Ces lemmes permettent de donner la forme suivante du th\'eor\`eme d'Autissier.

\begin{prop}\label{coromatnet} Soit $(A,L)$ une vari\'et\'e
ab\'elienne polaris{\'e}e de dimension $g$ sur un corps de nombres
$k$. On a 
$${1\over[k:\mathbf{Q}]}\sum_{\sigma\colon k\hookrightarrow\mathbf{C}}
\rho(A_\sigma,L_\sigma)^{-2}\le11\max(1,h(A),\log\deg_{L}A).$$\end{prop}

\begin{proof} Notons $T$ le membre de gauche. Vu
l'assertion {\`a} d{\'e}montrer, nous pouvons supposer $T\ge
11\max{(1,\log g!)}$. Si $(A,L)$ est principalement polaris{\'e}e,
nous raisonnons comme dans le cas elliptique avec $T'\ge T\ge3g/\pi$ pour obtenir
\begin{equation*}\frac{\pi T}{6}-\frac{g}{2}\log T\le
h(A)+\frac{g}{2}\log\left(\frac{2\pi^{2}e}{3g}\right).\end{equation*}
Posons $$c_{1}(g):=\frac{\pi}{6}-\frac{g}2\frac{\log(11\max{(1,\log
g!)})}{11\max{(1,\log g!)}}.$$ Par d{\'e}croissance de la
fonction $x\mapsto (\log x)/x$ pour $x\ge e$, on en
d{\'e}duit \begin{equation*}c_{1}(g)T\le
h(A)+\frac{g}{2}\log\left(\frac{2\pi^{2}e}{3g}\right).\end{equation*} Le
passage {\`a} une polarisation quelconque s'effectue \textit{via}
les deux lemmes pr{\'e}c{\'e}dents. En effet, la moyenne qui
d{\'e}finit $T$ est 
invariante par extension finie du corps de base $k$. Ceci nous permet
de supposer que l'isog\'enie donn\'ee par le lemme~\ref{quotpp} est
d\'efinie sur $k$. On a alors $h(B)\le h(A)+(1/2)\log\mathrm{h}^{0}(A,L)$ et
$\rho(A_\sigma,L_\sigma)^{-2}\le\rho(B_\sigma,M_\sigma)^{-2}$ par le
lemme~\ref{varrho}, pour tout plongement $\sigma$ de $k$ dans
$\mathbf{C}$. Ainsi de la majoration ci-dessus appliqu\'ee
{\`a} $(B,M)$ d{\'e}coulent les
estimations \begin{equation*}\begin{split}c_{1}(g)T&\le
h(A)+\frac{1}{2}\log\mathrm{h}^{0}(A,L)+\frac{g}{2}\log
\left(\frac{2\pi^{2}e}{3g}\right)\\
& \le h(A)+\frac{1}{2}\log\deg_{L}A-\frac{1}{2}\log
g!+\frac{g}{2}\log\left(\frac{2\pi^{2}e}{3g}\right)\\ &\le
c_{2}(g)\max(1,h(A),\log\deg_{L}A)\end{split}\end{equation*}
o{\`u} $$c_{2}(g):=\frac{3}{2}+\frac{\max{\left(0,(g/2)\log
(2\pi^{2}e/(3g))-(1/2)\log g!\right)}}{\max{(1,\log g!)}}.$$
Pour conclure, on v{\'e}rifie que $c_{2}(g)\le 11c_{1}(g)$ pour
tout $g\ge 1$. Pour cela, on peut proc{\'e}der de la mani{\`e}re
suivante~: on montre l'in\'egalit\'e par calcul direct si
$g\le5$. Pour $g\ge6$, on a $c_2(g)=3/2$ (car
$3gg!^{1/g}\ge18\times6!^{1/6}\ge2\pi^2e$) et 
$c_{1}(g)=\pi/6-g\log(11\log g!)/(22\log g!)\ge\pi/6-(\log 11+2\log
g)/(22\log g-22)$ (en utilisant $\log\log g!\le2\log g$ et $\log g!\ge
g\log (g/e)$). Cette derni\`ere fonction est croissante pour $g\ge6$
et sup\'erieure \`a $3/22$ si $g=6$ d'o\`u le r\'esultat.\end{proof} 
Le th\'eor\`eme~\ref{matint} se d\'emontre exactement comme ci-dessus
en utilisant $h(A)=h_F(A)+g\log\sqrt\pi$ (le $11$ dans $c_1(g)$ est
remplac{\'e} par $14$, le $\pi^2$ dans $c_2(g)$ devient $\pi^3$ et
$c_2(g)=3/2$ pour $g\ge12$).

\subsection{Pente maximale}\label{subsecpentemax}
Nous introduisons ici la pente maximale qui jouera un r\^ole essentiel
dans la preuve du th\'eor\`eme-clef ci-dessous
(th\'eor\`eme~\ref{thmimportant}) et dont l'estimation repose sur une
version du lemme matriciel.

Lorsque $(A,L)$ est une vari\'et\'e ab\'elienne polaris\'ee sur un
corps de nombres $k$, on munit l'espace tangent $t_A$ d'une structure
de $k$-fibr\'e vectoriel hermitien not\'ee $\overline{t_{(A,L)}}$ ou,
la plupart du temps, $\overline{t_A}$, lorsque la polarisation
sous-entendue est claire d'apr\`es le contexte. Pour ce faire, nous
utilisons la structure enti\`ere donn\'ee par l'espace tangent du
mod\`ele de N\'eron de $A$ (qui donne donc, dans le langage des
fibr\'es ad\'eliques de \cite{gaudronsept}, des normes en toutes les
places finies de $k$). En une place infinie $v$, nous utilisons la
m\'etrique induite par la forme de Riemann, d\'ecrite au
paragraphe~\ref{formeR}. Il n'y a pas d'ambigu\"\i t\'e car si
$\sigma$ et $\overline{\sigma}$ sont deux plongements complexes correspondant
tous deux \`a $v$ alors les normes $\|\cdot\|_{L,\sigma}$ et
$\|\cdot\|_{L,\overline{\sigma}}$ co{\"\i}ncident comme nous l'avons rappel\'e
au \S~\ref{subsecconju}.

Lorsque l'on dispose d'un fibr\'e ad\'elique $\overline E$ sur $k$
nous pouvons lui associer \`a la suite de Bost deux {\em pentes} :
d'une part sa pente (toujours normalis\'ee) $\widehat\mu(\overline E)$
d\'efinie page 62 de \cite{gaudronsept} et d'autre part sa
pente maximale $\widehat{\mu}_{\mathrm{max}}(\overline{E})$
qui est le maximum des pentes des sous-fibr\'es non nuls de $\overline{E}$.

Nous souhaitons donc \'evaluer la pente et la pente maximale de
$\overline{t_A}$ mais il faut faire attention que ces quantit\'es ne
sont pas pr\'eserv\'ees {\em a priori} par extension des scalaires. Au
contraire, nous souhaitons ne manipuler que des quantit\'es
invariantes par une telle extension. Nous r\'esolvons ce probl\`eme
comme dans le paragraphe 5.1.1 de \cite{bost2} : pour calculer ces
quantit\'es, nous ferons toujours d'abord une extension de corps de
sorte que $A$ ait r\'eduction semi-stable. Sous cette condition, nos
pentes ne d\'ependent plus du corps choisi (c'est la m{\^e}me
convention que pour la hauteur de Faltings {\em stable}).

Chaque fois que nous parlerons de la pente ou de la pente maximale de
$\overline{t_A}$ (ou de son dual, de leurs puissances
sym\'etriques,\ldots) nous ferons donc r\'ef\'erence aux pentes de
$\overline{t_{A_K}}$ pour une extension finie $K$ de $k$ telle que
$A_K$ admette un mod\`ele semi-stable. En pratique, ce l\'eger abus
d'\'ecriture n'engendrera pas de confusions car, lorsque nous ferons
appel explicitement \`a la structure hermitienne sur $t_A$, nous aurons
toujours au pr\'ealable fait une extension des scalaires assurant la
condition de semi-stabilit\'e. Surtout, cette convention nous
permettra de donner des \'enonc\'es invariants sans modifier notre
corps de base (et donc nous n'aurons pas \`a estimer le degr\'e d'une
extension sur laquelle $A$ acqui\`ere r\'eduction semi-stable).

Rappelons le fait suivant.

\begin{lemma}\label{pente} Soit $(A,L)$ une vari\'et\'e
ab\'elienne polaris\'ee sur un corps de nombres. On
a $$g\widehat\mu(\overline{t_{(A,L)}})=-h(A)-
\frac{1}{2}\log\mathrm{h}^{0}(A,L)+\frac{g}{2}\log\pi.$$\end{lemma}

\begin{proof} Ceci suit facilement des d\'efinitions : voir
l'\'enonc\'e (D.1) de \cite{bost2} et une preuve page 715 de
\cite{gaudrondeux}.\end{proof}

On peut tirer d'un lemme matriciel une estimation de la pente maximale
du dual de $\overline{t_A}$. Un \'enonc\'e explicite est donn\'e par
Graftieaux comme suit :

\begin{lemma} Si $L$ est une polarisation principale on a
$$\widehat{\mu}_{\mathrm{max}}(\overline{t_A^{\mathsf{v}}})
\le(g+1)h(A)+2g^5\log2.$$\end{lemma}

\begin{proof} Voir la proposition 2.14 de \cite{graftieaux}.\end{proof}

Nous passons au cas g\'en\'eral par isog\'enie comme ci-dessus.

\begin{lemma}\label{varmu} Soient $f\colon A\to B$ une isog\'enie entre
vari\'et\'es ab\'eliennes sur $k$ et $L$ une polarisation sur
$B$. Alors $$\widehat{\mu}_{\mathrm{max}}(\overline{t_{(A,f^*L)}})
\le\widehat{\mu}_{\mathrm{max}}(\overline{t_{(B,L)}}).$$\end{lemma}

\begin{proof} Apr\`es nous \^etre plac\'es sur une extension de corps
convenable o\`u $A$ et $B$ ont des mod\`eles semi-stables, nous
consid\'erons l'isomorphisme d'espaces vectoriels $\mathrm{d}f\colon t_A\to
t_B$. Pour chaque plongement $\sigma$, l'application $(\mathrm{d}f)_\sigma$ est
une isom\'etrie avec les normes relatives \`a $(f^*L)_\sigma$ et
$L_\sigma$ (voir la d\'emonstration du lemme~\ref{varrho}). D'autre
part, $\mathrm{d}f$ pr\'eserve les structures enti\`eres puisque $f$ s'\'etend
en un morphisme entre les mod\`eles de N\'eron de $A$ et $B$. Par
suite la norme de $\mathrm{d}f$ en une place ultram\'etrique quelconque est
inf\'erieure \`a 1. Le r\'esultat suit alors par in\'egalit\'e de
pentes (voir par exemple le lemme 6.4 de~\cite{gaudronsept}).\end{proof}
Nous en d\'eduisons l'\'enonc\'e suivant.

\begin{prop}\label{mumax} Si $(A,L)$ est une vari\'et\'e ab\'elienne
polaris\'ee on a $$\widehat{\mu}_{\mathrm{max}}(\overline{t_{(A,L)}^{\mathsf{v}}})
\le(g+1)(h(A)+\frac{1}{2}\log\mathrm{h}^{0}(A,L))+2g^5\log2.$$\end{prop}

\begin{proof} Il suffit de combiner les deux lemmes pr\'ec\'edents
avec le lemme~\ref{quotpp}.\end{proof}

\section{Minimum essentiel}

Nous {\'e}non{\c{c}}ons le th{\'e}or{\`e}me principal qui
entra{\^{\i}}ne les th{\'e}or{\`e}mes cit\'es dans l'introduction.
\subsection{Minimum d'\'evitement}
Soient $\mathsf{A}$ une vari{\'e}t{\'e} ab{\'e}lienne
complexe et $\mathsf{L}$ une polarisation sur
$\mathsf{A}$. Soient $\Vert.\Vert_{\mathsf{L}}$ la norme sur l'espace
tangent $t_{\mathsf{A}}$ induite par $\mathsf{L}$ (voir
\S~\ref{formeR}) et $\mathrm{d}_{\mathsf{L}}$ la distance associ\'ee.
\begin{defi}
Soit $\mathsf{B}$ une sous-vari{\'e}t{\'e} ab{\'e}lienne de
$\mathsf{A}$. Le
\emph{minimum d'{\'e}vitement} de $\mathsf{B}$, relatif {\`a}
$(\mathsf{A},\mathsf{L})$, est le nombre
r{\'e}el $$\updelta(\mathsf{A},\mathsf{L},\mathsf{B}):=
\min{\{\mathrm{d}_{\mathsf{L}}(\omega,t_{\mathsf{B}})\,;\
\omega\in\Omega_{\mathsf{A}}\setminus\Omega_{\mathsf{B}}\}}.$$Le
\emph{minimum essentiel} de $(\mathsf{A},\mathsf{L})$ est
$\updelta(\mathsf{A},\mathsf{L}):=\sup_{\mathsf{B}}{
\updelta(\mathsf{A},\mathsf{L},\mathsf{B})}$
(la borne sup{\'e}rieure est prise sur toutes les
sous-vari{\'e}t{\'e}s ab{\'e}liennes $\mathsf{B}$ de $\mathsf{A}$,
diff{\'e}rentes de $\mathsf{A}$). 
\end{defi}
Si $\mathsf{B}=\{0\}$, on retrouve le minimum absolu
$\updelta(\mathsf{A},\mathsf{L},\{0\})=\rho(\mathsf{A},\mathsf{L})$. Voici
quelques relations {\'e}l{\'e}mentaires auxquelles satisfont ces minima.
\begin{proprietes}Soient $\mathsf{B},\mathsf{C}$ des
sous-vari{\'e}t{\'e}s ab{\'e}liennes de $\mathsf{A}$.
\begin{enumerate}
\item Si $\mathsf{C}\ne\{0\}$ et si $\mathsf{B}\cap\mathsf{C}$ est
fini alors on a
$\updelta(\mathsf{A},\mathsf{L},\mathsf{B})\le\rho(\mathsf{C},\mathsf{L})$.
\item On a
$\updelta(\mathsf{A},\mathsf{L}_{1}\otimes\mathsf{L}_{2},\mathsf{B})^{2}\ge
\updelta(\mathsf{A},\mathsf{L}_{1},\mathsf{B})^{2}+\updelta(\mathsf{A},
\mathsf{L}_{2},\mathsf{B})^{2}$ et, pour tout entier $N\ge 1$, on
a $$\updelta(\mathsf{A},\mathsf{L}^{\otimes
N},\mathsf{B})=\sqrt{N}\updelta(\mathsf{A},\mathsf{L},\mathsf{B}).$$
\item Pour $i\in\{1,2\}$, soit
$(\mathsf{A}_{i},\mathsf{L}_{i},\mathsf{B}_{i})$ comme ci-dessus. On
a $$\updelta(\mathsf{A}_{1}\times\mathsf{A}_{2},\mathsf{L}_{1}\boxtimes
\mathsf{L}_{2},\mathsf{B}_{1}\times\mathsf{B}_{2})=\min_{i\in\{1,2\}}
{\updelta(\mathsf{A}_{i},\mathsf{L}_{i},\mathsf{B}_{i})}.$$
\end{enumerate}
\end{proprietes}
Nous avons toujours
$\rho(\mathsf{A},\mathsf{L})\le\updelta(\mathsf{A},\mathsf{L})$ mais il
n'est pas vrai en g{\'e}n{\'e}ral que
$\rho(\mathsf{A},\mathsf{L})\le\updelta(\mathsf{A},\mathsf{L},\mathsf{B})$
(prendre $\sf A=E\times E$ pour une courbe elliptique $\sf E$ avec une
polarisation produit $\sf L=L_0\boxtimes L_0$ puis $\sf B$ la diagonale ; on
a alors $\updelta(\mathsf{A,L,B})=\rho(\mathsf{A,L})/\sqrt2$).

\begin{prop}\label{propminima}Soit $\mathsf{B}$ une
sous-vari{\'e}t{\'e} ab{\'e}lienne stricte de $\mathsf{A}$. Soit
$\mathsf{b}$ le degr{\'e} de l'isog{\'e}nie d'addition
$\mathsf{B}\times\mathsf{B}^{\perp}\to\mathsf{A}$. Alors on a
$\rho(\mathsf{B}^{\perp},\mathsf{L})/\mathsf{b}\le\updelta
(\mathsf{A},\mathsf{L},\mathsf{B})$. En particulier on a
$$\frac{\rho(\mathsf{A},\mathsf{L})}{(\deg_{\mathsf{L}}
\mathsf{B})^{2}}\le\updelta(\mathsf{A},\mathsf{L},\mathsf{B}).$$\end{prop}

\begin{proof}
Soit $\omega$ une p\'eriode de $\mathsf{A}$ qui n'appartient pas \`a
$t_{\mathsf{B}}$. Consid{\'e}rons $\omega_{1},\omega_{2}$
comme au \S~ \ref{subsecvaortho} attach{\'e}s {\`a} $\omega$ et {\`a}
$\mathsf{B}$. On a $\updelta(\mathsf{A},\mathsf{L},\mathsf{B})
=\mathrm{d}_{\mathsf{L}}(\omega,t_{\mathsf{B}})=\Vert\omega_{2}
\Vert_{\mathsf{L}}/\mathsf{b}$ car les espaces $t_{\mathsf{B}}$ et
$t_{\mathsf{B}^{\perp}}$ sont orthogonaux. Par hypoth{\`e}se, on a
$\omega_{2}\ne 0$ et donc $\Vert\omega_{2}\Vert_{\mathsf{L}}
\ge\rho(\mathsf{B}^{\perp},\mathsf{L})\ge\rho(\mathsf{A},\mathsf{L})$. La
deuxi{\`e}me in{\'e}galit{\'e} de la proposition d{\'e}coule alors de
la majoration $\mathsf{b}\le(\deg_{\mathsf{L}}\mathsf{B})^{2}$.\end{proof}
{\'E}tant donn{\'e} une sous-vari{\'e}t{\'e} ab{\'e}lienne
$\mathsf{B}$ de $\mathsf{A}$, de codimension $\mathsf{t}\ge 1$, on
pose $$x(\mathsf{B}):=\left(\frac{\deg_{\mathsf{L}}\mathsf{B}}
{\deg_{\mathsf{L}}\mathsf{A}}\right)^{1/\mathsf{t}}.$$ 

\begin{prop}\label{majmu} Pour toute sous-vari{\'e}t{\'e}
ab{\'e}lienne stricte $\mathsf{B}$ de $\mathsf{A}$, on a
$$x(\mathsf{B})\updelta(\mathsf{A},\mathsf{L},\mathsf{B})^{2}\le2/\sqrt3$$
(si $\dim\sf A\ge2$ on peut remplacer $2/\sqrt3$ par 1).\end{prop}
\begin{proof}Notons $\mathsf{t}$ la codimension de $\mathsf{B}$ dans
$\mathsf{A}$. La quantit{\'e}
$\updelta(\mathsf{A},\mathsf{L},\mathsf{B})$ est la plus petite 
norme d'un {\'e}l{\'e}ment non nul du r{\'e}seau
$\Omega_{\mathsf{A}}/\Omega_{\mathsf{B}}$ de
$t_{\mathsf{A}}/t_{\mathsf{B}}$ (vu comme $\mathbf{R}$-espace
vectoriel), muni de la norme quotient. Par cons{\'e}quent, le
premier th{\'e}or{\`e}me de Minkowski donne
l'estimation $$\updelta(\mathsf{A},\mathsf{L},\mathsf{B})^{2}\le
\gamma_{2\mathsf{t}}\covol(\Omega_{\mathsf{A}}/
\Omega_{\mathsf{B}})^{1/\mathsf{t}}$$($\gamma_{2\mathsf{t}}$
est la constante d'Hermite). Le covolume du r{\'e}seau quotient est
le quotient des covolumes et, d'apr{\`e}s~\cite{berpph}, l'on sait
que $\covol(\Omega_{\mathsf{A}})$ est {\'e}gal {\`a}
$\mathrm{h}^{0}(\mathsf{A},\mathsf{L})$ (idem pour $\mathsf{B}$). En
revenant aux degr{\'e}s, la borne de Minkowski donne
donc $$x(\mathsf{B})\updelta(\mathsf{A},\mathsf{L},\mathsf{B})^{2}
\le\gamma_{2\mathsf{t}}\left(\frac{(g-\mathsf{t})!}{g!}\right)^{1/\mathsf{t}}.$$
Si $\mathsf{t}=1$, la valeur $\gamma_{2}=\frac{2}{\sqrt{3}}$ donne la
majoration voulue. Si $\mathsf{t}\ge 2$, on 
majore $(g-\mathsf{t})!/g!$ par $1/\mathsf{t}!$. Si
$\mathsf{t}\in\{2,3\}$, on conna{\^{\i}}t la valeur explicite de
$\gamma_{2t}$~: $$\gamma_{4}=\sqrt{2},\quad\gamma_{6}=\frac{2}{3^{1/6}},$$avec
laquelle on v{\'e}rifie ais{\'e}ment que
$\gamma_{2\mathsf{t}}\mathsf{t}!^{-1/\mathsf{t}}\le 1$. En
g{\'e}n{\'e}ral, on dispose de la borne de Blichfeldt
\cite[th\'eor\`eme 2, p. 387]{gruker}~:
$$\gamma_{2\mathsf{t}}\mathsf{t!}^{-1/\mathsf{t}}\le
\frac{2}{\pi}(\mathsf{t}+1)^{1/\mathsf{t}}.$$ L'on
peut alors conclure en observant que, si $\mathsf{t}\ge4$, on a
$(1+\mathsf{t})^{1/\mathsf{t}}\le\pi/2$.
\end{proof}

\subsection{Th\'eor\`eme-clef}

Soit $(A,L)$ une vari{\'e}t{\'e} ab{\'e}lienne polaris\'ee sur un
corps de nombres $k$. 
Pour un plongement complexe $\sigma\colon k\hookrightarrow\mathbf{C}$ et une
sous-vari\'et\'e ab\'elienne stricte $\sf B$ de $A_\sigma$, nous avons
d\'efini ci-dessus une quantit\'e $x(\sf B)$. Nous posons \`a
pr\'esent $$x:=\min{\{x({\sf B})\,;\ {\sf B}\subsetneq A_\sigma\}}$$
qui ne d\'epend pas du choix du plongement $\sigma$ mais seulement du
couple $(A,L)$.
Notons d'ores et d\'ej\`a $$(\deg_LA)^{-1}\le x\le(\deg_LA)^{-1/g}$$
comme on le voit avec $\deg_{L_\sigma}\sf B\ge1$ et $x\le x(\{0\})$.

\begin{theo}\label{thmimportant}Consid{\'e}rons pour chaque plongement
complexe $\sigma\colon k\hookrightarrow\mathbf{C}$ une sous-vari{\'e}t{\'e}
ab{\'e}lienne $B[\sigma]$ de $A_{\sigma}$, diff{\'e}rente de
$A_{\sigma}$. On suppose que $B[\sigma]$ et $B[\overline{\sigma}]$ se
correspondent \textit{via} l'isomorphisme ${\sf f}\colon 
A_{\sigma}\simeq A_{\overline{\sigma}}$ du
\S~\ref{subsecconju}. Soient \begin{equation*}\updelta_{\sigma}:=
\updelta(A_{\sigma},L_{\sigma},B[\sigma])\quad\text{et}\quad\xi_{\sigma}:= 
\left(\frac{x}{x(B[\sigma])}\right)^{\codim B[\sigma]}.\end{equation*}
Alors on a\begin{equation*}\frac{1}{[k:\mathbf{Q}]}
\sum_{\sigma\colon k\hookrightarrow\mathbf{C}}\left(\xi_\sigma\over
\updelta_\sigma\right)^2\le
131g^{2g+6}x\max{\left(1,h(A),\log\deg_{L}A,\frac{1}{[k:\mathbf{Q}]}
\sum_{\sigma\colon k\hookrightarrow\mathbf{C}}{\log\deg_{L_{\sigma}}B[\sigma]}\right)}.
\end{equation*}De plus, si $g=1$ ou si $x\le1/2141$ alors on peut
remplacer la constante num{\'e}rique $131$ par $23$.\end{theo}
Notons que la condition sur les $B[\sigma]$ entra{\^\i}ne $\updelta_{\sigma}
=\updelta_{\overline{\sigma}}$ et $\xi_{\sigma}=\xi_{\overline{\sigma}}$.
La constante $1/2141$ qui appara{\^{\i}}t provient de la
d{\'e}monstration de la cons{\'e}quence suivante.
\begin{coro}\label{corollaireclef}{\'E}tant donn{\'e} une
vari{\'e}t{\'e} ab{\'e}lienne polaris{\'e}e $(A,L)$ sur un corps de
nombres $k$, on a\begin{equation*}{\sqrt3\over2}x\le \frac{1}{[k:\mathbf{Q}]} 
\sum_{\sigma\colon k\hookrightarrow\mathbf{C}}{\updelta(A_{\sigma},L_{\sigma})^{-2}}\le
23g^{2g+6}x\max{\left(1,h(A),\log\deg_{L}A\right)}.
\end{equation*}
\end{coro}
\begin{proof} La minoration de la moyenne des
$\updelta(A_{\sigma},L_{\sigma})^{-2}$ 
est une simple application de la proposition~\ref{majmu} (en minorant
$x(B[\sigma])$ par $x$). Pour la majoration, observons que l'on a
toujours $$\frac{1}{[k:\mathbf{Q}]}\sum_{\sigma\colon k\hookrightarrow\mathbf{C}}
{\updelta(A_{\sigma},L_{\sigma})^{-2}}\le\frac{1}{[k:\mathbf{Q}]}
\sum_{\sigma\colon k\hookrightarrow\mathbf{C}}{\rho(A_{\sigma},L_{\sigma})^{-2}}$$ car
$\rho(A_{\sigma},L_{\sigma})\le\updelta(A_{\sigma},L_{\sigma})$. Par
cons{\'e}quent, si $x\ge11/(23g^{2g+6})$, le lemme matriciel de la
proposition~\ref{coromatnet} permet de conclure
imm{\'e}diatement. Dans 
le cas contraire, on a ou bien $g=1$ ou $x\le11/(23g^{2g+6})\le1/2141$.
Pour tout plongement $\sigma$ nous
choisissons une sous-vari{\'e}t{\'e} ab{\'e}lienne $B[\sigma]$ de
$A_{\sigma}$ telle que $x(B[\sigma])=x$ (et donc $\xi_\sigma=1$). Dans
ce cas on a $\deg_{L_{\sigma}}B[\sigma]\le\deg_LA$ et, par
d{\'e}finition, $\updelta(A_\sigma,L_\sigma)^{-2}\le
\updelta(A_{\sigma},L_{\sigma},B[\sigma])^{-2}$. Le
th\'eor\`eme~\ref{thmimportant} donne alors le r{\'e}sultat
voulu. 
\end{proof}
\subsection{Premi\`eres r\'eductions} Nous montrons ici que, pour
\'etablir le th\'eor\`eme~\ref{thmimportant}, nous pouvons supposer
$g\ge2$ et faire une extension finie du corps $k$.
\subsubsection{Courbes elliptiques}\label{redell} Lorsque $g=1$, nous avons
automatiquement $B[\sigma]=0$ et $x=x(0)=(\deg_LA)^{-1}$. Par suite
$\xi_\sigma=1$ tandis que $\updelta_\sigma=\rho(A_\sigma,L_\sigma)$. En
outre, la polarisation $L$ est une puissance de l'unique polarisation
principale de $A$, disons $L_0$. Ainsi $L=L_0^{\otimes\deg_LA}$ et donc
$\rho(A_\sigma,L_\sigma)^2=(\deg_LA)\rho(A_\sigma,(L_0)_\sigma)^2$. Finalement
la formule \`a d\'emontrer se simplifie donc en $$\frac{1}{[k:\mathbf{Q}]}
\sum_{\sigma\colon k\hookrightarrow\mathbf{C}}\rho(A_\sigma,(L_0)_\sigma)^{-2}\le
23\max(1,h(A),\log\deg_LA)$$ et elle d\'ecoule alors facilement du lemme
matriciel pour $(A,L_0)$ (par exemple la proposition~\ref{coromatnet}
suffit). Nous supposons d\'esormais $g\ge2$.
\subsubsection{Variation du corps}\label{ssredext}
Les donn\'ees $(A,L,(B[\sigma])_{\sigma\colon k\hookrightarrow\mathbf{C}})$
admettent une notion naturelle d'extension des scalaires : si $K$ est
une extension finie de $k$ alors on d\'efinit
$(A_K,L_K,(B[\sigma'])_{\sigma'\colon K\hookrightarrow\mathbf{C}})$ en posant
simplement $B[\sigma']:=B[\sigma'|_k]$. Alors le th\'eor\`eme est
invariant par extension des scalaires. Ainsi dans la suite nous
pourrons faire librement une extension finie du corps $k$.

\subsection{Strat{\'e}gie}
La d{\'e}monstration du th{\'e}or{\`e}me~\ref{thmimportant} repose sur
une construction de transcendance, qui s'inspire du \emph{cas
p{\'e}riodique} de la th{\'e}orie des formes lin{\'e}aires de
logarithmes. Plus pr{\'e}cis{\'e}ment, nous utilisons la variante de
la m{\'e}thode de Gel'fond-Baker propos{\'e}e par Philippon et
Waldschmidt~\cite{pph-miw}, variante qui permet d'extrapoler sur les
d{\'e}rivations (dans une direction bien choisie) plut{\^{o}}t que sur
les points. \par Sch{\'e}matiquement, cette m{\'e}thode consiste {\`a}
construire une fonction auxiliaire qui est petite en l'origine de
$t_{A_{\sigma}}$ dans toutes les directions (jusqu'{\`a} un certain
ordre $gT$) \emph{sauf une} (en substance celle donn{\'e}e par un
\'el\'ement qui r{\'e}alise le minimum $\updelta_{\sigma}$) pour laquelle
l'ordre est bloqu{\'e} {\`a} $T_{\sigma}\ll T\xi_{\sigma}$. Par le
biais d'un lemme d'interpolation analytique (en \emph{une} variable),
on montre alors que l'on peut s'affranchir de cette derni{\`e}re
restriction, quitte {\`a} remplacer $gT$ par $T$, ce qui fournit des
bornes (dites \emph{fines}) de la \og{}premi{\`e}re\fg{}
d{\'e}riv{\'e}e non nulle de la fonction auxiliaire en
l'origine. Apr{\`e}s renormalisation {\'e}ventuelle, cette
d{\'e}riv{\'e}e est un nombre alg{\'e}brique et un lemme de
multiplicit{\'e}s assure qu'il est non nul. Ce nombre satisfait alors
{\`a} la formule du produit. La majoration de ses valeurs absolues en
les places $p$-adiques du corps de nombres ambiant $k$ conduit par
comparaison avec les estimations archim{\'e}diennes fines {\`a} une
in{\'e}galit{\'e} brute de laquelle est extraite l'information voulue
(ici la majoration de la moyenne des
$(\xi_{\sigma}/\updelta_{\sigma})^{2}$). Consid{\'e}rer toutes les places
de $k$ au lieu d'une seule avec, en outre, des $\xi_{\sigma}$ non
n{\'e}cessairement {\'e}gaux {\`a} $1$ est une des
caract{\'e}ristiques originales de notre d{\'e}monstration.
\par Nous avons perfectionn{\'e} ce canevas sous trois angles~: (i) nous
avons introduit la \emph{m{\'e}thode de la section auxiliaire},
{\'e}labor{\'e}e dans~\cite{gaudronhuit}, qui remplace celle des
fonctions auxiliaires, avec les avantages d{\'e}j{\`a}
{\'e}voqu{\'e}s {\`a} la fin de l'introduction, (ii) nous apportons
un nouveau lemme d'interpolation analytique, d'int{\'e}r\^{e}t
ind{\'e}pendant, qui fera l'objet de la partie suivante, (iii) nous
{\'e}valuons de mani{\`e}re quasi-optimale les rangs asymptotiques
des syst{\`e}mes lin{\'e}aires avec lesquels est b{\^{a}}tie la
section auxiliaire. Ces {\'e}volutions permettent de travailler dans
un cadre plus agr{\'e}able qui {\'e}limine naturellement certaines
difficult{\'e}s techniques (par exemple, il n'y a plus \og{}d'astuce
d'Anderson-Baker-Coates\fg{}), tout en conduisant {\`a} de bien
meilleures constantes num{\'e}riques qu'auparavant.
\section{Pr\'elude \`a l'extrapolation analytique}
\label{secextrapolation}
Dans cette partie, nous \'etablissons le r\'esultat crucial pour
extrapoler sur les d\'eriv\'ees dans la d\'emonstration du
th\'eor\`eme~\ref{thmimportant}. Il s'agit d'un lemme de Schwarz approch\'e, de
facture assez classique. On en trouvera par exemple une formulation
plus g\'en\'erale dans l'article de Cijsouw et Waldschmidt
\cite{cijsouwmiw}. Nous avons cependant besoin d'une version significativement
plus fine en vue des calculs explicites de constantes. Pour cela, nous
modifions la trame de la preuve de \cite{cijsouwmiw} de trois fa{\c c}ons~: en
premier lieu, puisque nous ne souhaitons extrapoler qu'en 0, nous ne
majorons le module de notre fonction analytique que sur un petit
disque (de rayon 1 au lieu de $2S$, dans les notations ci-dessous) ;
ensuite, nous rempla\c cons en fait ce disque par un domaine plus
compliqu\'e le contenant (voir figure), pour \'eviter au mieux les contours
d'int\'egration, l\'eg\`erement contract\'es, qui apparaissent dans
nos calculs de r\'esidus (formule d'interpolation d'Hermite) ; enfin
nous estimons de mani\`ere tr\`es pr\'ecise les extrema du polyn\^ome
auxiliaire de la dite formule (voir lemme~\ref{lemmecinqdeux}).\par Voici notre
r\'esultat, d\'eclin\'e en une forme brute et une forme l\'eg\`erement
plus faible que nous utiliserons plus bas. Si $R$ est un
nombre r\'eel positif et si $D(0,R)$ d\'esigne le disque
ferm\'e $\{z\in\mathbf{C};\ \vert z\vert\le R\}$, on note
$\vert f\vert_R$ la borne sup\'erieure des $\vert
f(z)\vert$ pour $z\in D(0,R)$.
\begin{prop}\label{lana} Soient $S$ et $T$ deux entiers naturels non
nuls, $\varepsilon$ un nombre r\'eel tel que $0<\varepsilon<1/2$ et
$f\colon\mathbf{C}\to\mathbf{C}$ une fonction holomorphe. Alors on a :
$$\vert
f\vert_1\le4\left(\frac{(S-1)!^2\sh\pi}{\pi(2S-1)!}\right)^T\vert
f\vert_{S}+\frac{ST}{\varepsilon}\left(\frac{\sh\pi}{\cos\pi
\varepsilon}\right)^{T}\max_{\genfrac{}{}{0cm}{}{j\in\mathbf{Z},\
\vert j\vert<S}{\ell\in\mathbf{N},\ \ell<T}}\left\vert\frac{1}
{2^\ell\ell!}f^{(\ell)}(j)\right\vert.$$ En particulier, on a aussi
$$\vert f\vert_1\le4\left(\frac{10}{4^S}\right)^T\vert f\vert_S
+12ST(12)^T\max_{\genfrac{}{}{0cm}{}{j\in\mathbf{Z},\ \vert j\vert<S}
{\ell\in\mathbf{N},\ \ell<T}}{\left\vert\frac{1}{2^\ell\ell!}
f^{(\ell)}(j)\right\vert}.$$\end{prop}

On comparera avec \cite[p.~179--180]{cijsouwmiw} en prenant
$\delta=1$, $k=2S-1$,
$E=\{1-S,\ldots,S-1\}$, $r=S$, $R=8S$ qui donne la m\^eme puissance
$2^{-2ST}$ dans le premier terme mais au prix de remplacer $|f|_S$ par
$|f|_{8S}$ ; dans le second terme la puissance de l'ordre de $81^{ST}$
devient $12^T$ ; bien s\^ur, rappelons que nous majorons seulement
$|f|_1$ et non $|f|_{2S}$ mais cela ne fait que peu de diff{\'e}rence
lorsqu'il s'agit d'estimer les d{\'e}riv{\'e}es en 0.

Commen\c cons par un lemme pr{\'e}liminaire.

\begin{lemma}\label{lemmecinqdeux} Soient $S$ un entier naturel non nul et
$P=\prod_{j=1-S}^{S-1}(X-j)\in\mathbf{Z}[X]$.\begin{itemize}\item[$(1)$]
$P(S)=(2S-1)!=-P(-S)$.\item[$(2)$] Si $t\in\mathbf{R}$ et $|t|\le S$ alors
$|P(t)|\ge(S-1)!^2\pi^{-1}|\sin(\pi t)|$.\item[$(3)$] Si $z\in\mathbf{C}$ et
$\min(|z|,2|z-1|,2|z+1|)\le1$ alors
$|P(z)|\le(S-1)!^2\pi^{-1}\sh(\pi)$.\item[$(4)$] Si $k\in\mathbf{Z}$ et
$\rho\in\mathbf{R}^{+}$ alors $$\min{\{\vert P(z)\vert\,;\
z\in\mathbf{C}\ \textrm{et}\ \vert
z-k\vert=\rho\}}=\min(|P(k+\rho)|,|P(k-\rho)|).$$\end{itemize}\end{lemma}

\begin{proof} L'assertion (1) se passe de commentaires. Pour (2) et
(3) \'ecrivons
$P(X)=X\prod_{j=1}^{S-1}{(X^2-j^2)}=(S-1)!^2X\prod_{j=1}^{S-1}{(X^2/j^2-1)}$.
On rappelle aussi que $$\sin\pi t=\pi
t\prod_{j=1}^{\infty}{\left(1-\frac{t^2}{j^2}\right)}\quad\textrm{et}
\quad\sh\pi=\pi\prod_{j=1}^{\infty}{\left(1+\frac{1}{j^2}\right)}.$$
La relation (2) se r{\'e}duit
donc \`a $\prod_{j=S}^{\infty}{|1-t^2/j^2|}\le1$ qui d{\'e}coule bien de
$|t|\le S$. Pour (3) nous devons montrer
$|z|\prod_{j=1}^{S-1}{\vert1-z^2/j^2\vert}\le\prod_{j=1}^{\infty}{(1+1/j^2)}$.
Cette formule \'etant claire pour $\vert z\vert\le1$ et invariante
sous $z\mapsto-z$, nous pouvons supposer $\vert z-1\vert\le1/2$. Nous
avons alors $|z||1-z^2|\le(3/4)\vert 1+z\vert\le15/8\le2=1+1/1^2$ et
il suffit donc de v{\'e}rifier $|1-z^2/j^2|\le1+1/j^2$ pour
$j\ge2$. En {\'e}levant au carr{\'e} et en simplifiant, ceci
{\'e}quivaut \`a $|z|^4-1\le2j^2(1+\mathrm{Re}(z^2))$. Enfin nous avons
$|z|^4\le(3/2)^4\le8\le2j^2$ et $\mathrm{Re}(z^2)\ge0$ car, par
exemple, $|\mathrm{Arg}(z)|\le\pi/4$. Passons \`a (4). Si $|z-k|=\rho$
et $x=\mathrm{Re}(z-k)$
alors $$|P(z)|^2=\prod_{j=k-S+1}^{k+S-1}{(j^2+\rho^2+2jx)}.$$ Scindons
$E=\{k-S+1,\ldots,k+S-1\}$ en $F=\{j\in E\,;\ -j\in E\}$ et $G=E\setminus
F$ (chacun pouvant \^etre vide). On remarque que tous les {\'e}l{\'e}ments
de $G$ ont le m\^eme signe donc la fonction $x\mapsto\prod_{j\in
G}{(j^2+\rho^2+2jx)}$ est monotone (tous les facteurs sont positifs car
$-\rho\le x\le\rho$) et elle atteint son minimum en $\rho$ ou
$-\rho$. D'autre part, si $F\ne\emptyset$, on a $0\in F$
et $$\prod_{j\in F}{(j^2+\rho^2+2jx)}=\rho^2\prod_{j\in F,\
j\ge1}{((j^2+\rho^2)^2-4j^2x^2)}.$$ Nous obtenons donc une fonction paire
minimale en $x=\rho$ et en $x=-\rho$. En faisant le produit, nous
voyons que $\vert P(z)\vert$ est minimal en l'un des deux points donn{\'e}s par
$\vert x\vert=\rho$. C'est le r{\'e}sultat.\end{proof}

\begin{proof}[D\'emonstration de la proposition~\ref{lana}] Si $S=1$,
l'\'enonc\'e est tautologique donc nous supposons $S\ge2$. Notons
$\Gamma=\{\zeta\in\mathbf{C}\,;\ \vert\zeta\vert=S\}$ et
$\Gamma_j=\{\zeta\in\mathbf{C}\,;\
\vert\zeta-j\vert=(1/2)-\varepsilon\}$ pour $|j|<S$
ainsi que $Q=P^T$ avec la notation $P$ du lemme~\ref{lemmecinqdeux}. Comme dans
\cite{cijsouwmiw}, nous partons de la formule d'interpolation d'Hermite
$$\frac{f(z)}{Q(z)}=\frac{1}{2i\pi}\int_\Gamma{\frac{f(\zeta)}{Q(\zeta)}
\frac{\mathrm{d}\zeta}{\zeta-z}}-\frac{1}{2i\pi}\sum_{j=1-S}^{S-1}{\sum_{
\ell=0}^{T-1}{\frac{f^{(\ell)}(j)}{\ell!}\int_{\Gamma_{j}}
{\frac{(\zeta-j)^\ell}{Q(\zeta)}\frac{\mathrm{d}\zeta}{\zeta-z}}}}$$
valable pour $z\in\mathbf{C}$ v{\'e}rifiant
$\vert z\vert<S$ et $|z-j|>(1/2)-\varepsilon$ pour $1-S\le j\le S-1$. Nous
l'appliquons pour $z$ tel que $\min(|z|,2|1-z|,2|1+z|)=1$ :

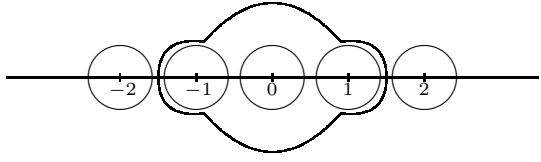
\setlength{\unitlength}{0.5mm}
\begin{picture}(60,90)(-55,-30)
\put(20,20){\circle{16}}
\put(40,20){\circle{16}}
\put(60,20){\circle{16}}
\put(80,20){\circle{16}}
\put(100,20){\circle{16}}
\put(-10,20){\line(1,0){140}}

\qbezier(30,20)(30,31)(42,29.5)
\qbezier(30,20)(30,9)(42,10.5)
\qbezier(90,20)(90,31)(78,29.5)
\qbezier(90,20)(90,9)(78,10.5)
\qbezier(42,29.5)(60,50)(78,29.5)
\qbezier(42,10.5)(60,-10)(78,10.5)

\put(20,19){\line(0,1){2}}
\put(17,16.5){${}_{-2}$}
\put(40,19){\line(0,1){2}}
\put(37,16.5){${}_{-1}$}
\put(60,19){\line(0,1){2}}
\put(58.5,16.5){${}_0$}
\put(80,19){\line(0,1){2}}
\put(78.5,16.5){${}_1$}
\put(100,19){\line(0,1){2}}
\put(98.5,16.5){${}_2$}

\put(-30,-10){Trac{\'e} approximatif de la courbe $\min(|z|,2|1-z|,2|1+z|)=1$.}
\put(-15,-20){\small Elle est toujours distante d'au moins $\varepsilon$ des
petits cercles.}\end{picture}

Dans l'int{\'e}grale le long de $\Gamma$, nous avons $|\zeta-z|\ge
S-(3/2)\ge S/4$. De plus, les assertions (1) et (4) du
lemme~\ref{lemmecinqdeux} avec $k=0$ et $\rho=S$ donnent
$|Q(\zeta)|\ge(2S-1)!^T$ ; comme $\Gamma$ est de longueur $2\pi S$, il
vient $$\left\vert\frac{1}{2i\pi}\int_\Gamma{\frac{f(\zeta)}{Q(\zeta)}
\frac{\mathrm{d}\zeta}{\zeta-z}}\right\vert\le\frac{4}{(2S-1)!^T}\vert
f\vert_S.$$ Dans l'int{\'e}grale le long de $\Gamma_j$, nous estimons
$\vert\zeta-j\vert=(1/2)-\varepsilon\le1/2$ et
$\vert\zeta-z\vert\ge\varepsilon$ tandis que
$|Q(\zeta)|\ge\min(|Q(j-(1/2)+\varepsilon)|,|Q(j+(1/2)-\varepsilon)|)
\ge(S-1)!^{2T}\pi^{-T}(\cos\pi\varepsilon)^T$ par le
lemme~\ref{lemmecinqdeux}, (2) et (4). 
Comme $\Gamma_j$ est de longueur $\pi-2\pi\varepsilon\le\pi$, nous trouvons
\begin{equation*}\begin{split}&\left\vert\frac{1}{2i\pi}
\sum_{j=1-S}^{S-1}{\sum_{\ell=0}^{T-1}
{\frac{f^{(\ell)}(j)}{\ell!}\int_{\Gamma_j}{\frac{(\zeta-j)^\ell}
{Q(\zeta)}\frac{\mathrm{d}\zeta}{\zeta-z}}}}\right\vert\\
&\quad\le\frac{ST}{\varepsilon}\left(\frac{\pi}{(S-1)!^2
\cos\pi\varepsilon}\right)^{T}\max_{j\in\mathbf{Z},\  |j|<S,\ \ell
\in\mathbf{N},\ \ell<T}\left\vert\frac{1}{2^\ell\ell!}f^{(\ell)}(j)
\right\vert.\end{split}\end{equation*} Pour obtenir la premi\`ere
majoration de l'{\'e}nonc{\'e} il reste \`a utiliser
$|Q(z)|\le(S-1)!^{2T}\pi^{-T}\sh(\pi)^T$ d'apr\`es le
lemme~\ref{lemmecinqdeux}, (3), et {\`a} rappeler que le principe du
maximum donne $$|f|_1\le\sup\{\vert f(z)\vert\,;\ z\in\mathbf{C}\
\textrm{et}\ \min(|z|,2|1-z|,2|1+z|)=1\}.$$ Afin de passer \`a la
seconde formulation, nous {\'e}crivons
$u_S=4^S(S-1)!^2(2S-1)!^{-1}$. Un calcul imm{\'e}diat fournit
$u_S/u_{S+1}=1+1/(2S)$ donc $u_S$ d{\'e}cro{\^\i}t puis $u_S\le
u_2=8/3$. Nous avons donc $$\frac{(S-1)!^2\sh\pi}{\pi(2S-1)!}\le
\frac{8\sh\pi}{3\pi}4^{-S}.$$
Parall\`element nous utilisons $\varepsilon=1/12$ et nous terminons par les
estimations num\'e\-riques $$\frac{8\sh\pi}{3\pi}\le10\qquad\textrm{et}
\qquad\frac{\sh\pi}{\cos(\pi/12)}\le12.$$\end{proof}

\begin{remas} Nous pourrions, comme dans \cite{cijsouwmiw}, supprimer
\`a la fois le $T$ de $ST/\varepsilon$ et le $2^\ell$ en utilisant
$\sum_{\ell=0}^{T-1}{2^{-\ell}}<2$. Ceci n'a aucune influence pour notre
application car $T$ tendra vers l'infini et seule importera la limite
de $(1/T)\log\vert f\vert_1$. Pour cette m\^eme raison, nous pourrions garder
$\varepsilon$ dans la formule et le faire tendre vers 0 \textit{in fine}.

En supposant $S\ge175$ le premier terme pourrait \^etre remplac{\'e} par
$4^{-ST}\vert f\vert_S$. Alternativement nous pourrions {\'e}crire
$4(15S^{-1/2}4^{-S})^T\vert f\vert_S$ en majorant $u_S$ plus finement.

\end{remas}

\section{D{\'e}monstration du th{\'e}or{\`e}me-clef}

\subsection{Choix des m{\'e}triques}\label{subsectionchoix}
Soit $(A,L)$ la vari{\'e}t{\'e} ab{\'e}lienne polaris{\'e}e du
th{\'e}or{\`e}me-clef~\ref{thmimportant}. Pour chaque plongement
complexe $\sigma\colon k\hookrightarrow\mathbf{C}$, il existe une unique
m{\'e}trique sur $L_{\sigma}$, dite \emph{m{\'e}trique
cubiste}, de forme de courbure invariante par translation et
rigidifi{\'e}e {\`a} l'origine~:
$0_{A_{\sigma}}^{*}L_{\sigma}\simeq\mathcal{O}_{\spec\mathbf{C}}$ est une
isom{\'e}trie (avec la m{\'e}trique triviale sur
$\mathcal{O}_{\spec\mathbf{C}}$). Quitte {\`a} faire une extension finie
(voir~\ref{ssredext}), l'on peut supposer que le
couple $(A,L)$ poss{\`e}de un mod{\`e}le de Moret-Bailly
$(\mathcal{A},\overline{\mathcal{L}})$ sur $k$, au sens
suivant~:\begin{enumerate}
\item[$\bullet$] il existe un sch{\'e}ma en groupes
$\mathcal{A}\to\spec\mathcal{O}_{k}$ semi-stable (donc lisse), de
fibr{\'e} g{\'e}n{\'e}rique isomorphe {\`a} $A$ (ce
sch{\'e}ma en groupes est le mod{\`e}le de N{\'e}ron de $A$),
\item[$\bullet$] il existe un fibr{\'e} hermitien cubiste
$\overline{\mathcal{L}}:=(\mathcal{L},(\Vert.
\Vert_{\mathrm{cub},\sigma})_{\sigma\colon k\hookrightarrow\mathbf{C}})$
sur $\mathcal{A}$, de fibre g{\'e}n{\'e}rique $L$ (le terme cubiste
signifie que la m{\'e}trique $\Vert.\Vert_{\mathrm{cub},\sigma}$ sur
$\mathcal{L}\otimes_{\sigma}\mathbf{C}$ est cubiste pour tout
$\sigma\colon k\hookrightarrow\mathbf{C}$).
\end{enumerate} L'existence d'un tel mod{\`e}le est d{\'e}montr{\'e}e
au \S~$4.3$ de~\cite{bostduke}. Soulignons que la d{\'e}finition de
fibr{\'e} cubiste implique que $\overline{\mathcal{L}}$ est
rigidifi{\'e} {\`a} l'origine. Pour tout entier $n\ge 1$, le 
$k$-espace vectoriel $H_{n}:=\mathrm{H}^{0}(A,L^{\otimes n})$ des
sections globales poss\`ede une structure de fibr\'e ad\'elique
hermitien
$\overline{H_{n}}=(H_{n},(\Vert.\Vert_{\overline{H_{n}},v})_{v})$ sur
$k$\,; la structure enti\`ere est donn\'ee par
$\mathrm{H}^{0}(\mathcal{A},\mathcal{L}^{\otimes n})$~: pour toute
place ultram{\'e}trique $v$ de $k$, pour tout
$s\in\mathrm{H}^{0}(A,L^{\otimes n})\otimes_{k}k_{v}$, on
a \begin{equation}\label{formnormeultra}\Vert
s\Vert_{\overline{H_{n}},v}:=\min{\left\{\vert\lambda\vert_{v}\,;\
\lambda\in k_{v}\setminus\{0\}\ \text{et}\
s/\lambda\in\mathrm{H}^{0}(\mathcal{A},\mathcal{L}^{\otimes
n})\otimes_{\mathcal{O}_{k}}\mathcal{O}_{v}\right\}}\end{equation}($k_{v}$
est le compl{\'e}t{\'e} de $k$ en la place $v$ et $\mathcal{O}_{v}$
son anneau de valuation). La structure archim{\'e}dienne de
$\overline{H_{n}}$ est donn{\'e}e par int{\'e}gration des normes
cubistes~: pour tout $\sigma\colon k\hookrightarrow\mathbf{C}$, pour tout
$s\in\mathrm{H}^{0}(A,L^{\otimes n})\otimes_{\sigma}\mathbf{C}$, $$\Vert
s\Vert_{\overline{H_{n}},\sigma}:=\left(\int_{A_{\sigma}}{\Vert
s(x)\Vert_{\mathrm{cub},\sigma}^{2}\mathrm{d}x}\right)^{1/2}$$o{\`u}
$\mathrm{d}x$ est la mesure de Haar normalis{\'e}e sur
$A_{\sigma}$. Muni de ces normes, $\overline{H_{n}}$ a une
structure de fibr{\'e} ad{\'e}lique hermitien et sa pente d'Arakelov
normalis{\'e}e a {\'e}t{\'e} calcul{\'e}e par Bost (voir le
th{\'e}or{\`e}me~$4.10$, (v),
de~\cite{bostduke})~:\begin{equation}\label{formulepenterestriction}
\widehat{\mu}(\overline{H_{n}})=-\frac{1}{2}h(A)
+\frac{1}{4}\log\left(\frac{n^{g}\mathrm{h}^{0}(A,L)}{(2\pi)^{g}}\right).
\end{equation}

Par ailleurs, comme nous l'avons vu au \S~\ref{subsecpentemax},
l'espace tangent $t_{A}$ de $A$ poss{\`e}de lui-m{\^{e}}me une
structure de fibr{\'e} ad{\'e}lique hermitien
$\overline{t_{A}}=(t_{A},(\Vert.\Vert_{L,\sigma})_{\sigma\colon k\hookrightarrow\mathbf{C}})$
(dont nous pourrons utiliser la pente sans risque en vertu de
l'hypoth\`ese de semi-stabilit\'e faite ci-dessus). Il
existe un lien entre la m\'etrique cubiste et la m\'etrique
$\Vert.\Vert_{L,\sigma}$.
Si $\vartheta$ est la fonction th\^eta (voir \S~\ref{fnth})
associ{\'e}e {\`a} $s\in\mathrm{H}^{0}(A,L^{\otimes
n})\otimes_{\sigma}\mathbf{C}$ alors, pour tout $x=\exp_{A_\sigma}(z)\in
A_{\sigma}$, on a \begin{equation}\label{formulesectiontheta}\Vert
s(x)\Vert_{\mathrm{cub},\sigma}=\vert\vartheta(z)\vert
\exp{\left(-\frac{\pi}{2}n\Vert
z\Vert_{L,\sigma}^{2}\right)}\cdotp\end{equation}

\subsection{Choix des param{\`e}tres}\label{secchoixpara}
Soit $\mathsf{n}$ un nombre r{\'e}el $\ge 1$ tel que $n:=x\mathsf{n}$
soit un entier. On pose $T:=[\mathsf{n}]+1$. Nous introduisons aussi
le nombre r\'eel $\theta=(\log2)/\pi$ et, pour chaque plongement complexe
$\sigma\colon k\hookrightarrow\mathbf{C}$, le r\'eel
$\varepsilon_\sigma=(6\sqrt2-8)g^{-g}\xi_\sigma$. Nous signalons toutefois que
ces valeurs exactes ne seront utilis\'ees qu'au
paragraphe~\ref{secconcl}. D'ici l\`a, nous n'aurons besoin que de
$\theta>0$ et $0<\varepsilon_\sigma<1$ pour tout $\sigma$.

Notre derni\`ere famille de param{\`e}tres (entiers) est d\'efinie par
$T_{\sigma}:=[\varepsilon_{\sigma}\mathsf{n}]$ pour tout plongement
$\sigma$. Notons $T_\sigma\le T$. Le param{\`e}tre $\mathsf{n}$ va
tendre vers $+\infty$ en fin de d{\'e}monstration. En particulier l'on
peut supposer que $T,n$ et les $T_{\sigma}$ ne sont pas nuls. Le choix
de $x$ assure le r{\'e}sultat suivant.
\begin{prop}\label{propmultiplicites}
Il n'existe aucune section non nulle de $\mathrm{H}^{0}(A,L^{\otimes
n})$ qui s'annule {\`a} l'ordre $gT$ le long de $t_{A}$ en $0_{A}$.
\end{prop}
\begin{proof}
Dans le cas contraire, le lemme de multiplicit{\'e}s de
Nakamaye~\cite{nakamaye} assure l'existence d'une sous-vari{\'e}t{\'e}
ab{\'e}lienne $A'$ de $A$, avec $A'\ne A$ et $A'$ d{\'e}finie sur
$\overline{k}$, telle que $T^{g-\dim
A'}\deg_{L}A'\le(\deg_{L}A)n^{g-\dim A'}$. En {\'e}crivant cette
in{\'e}galit{\'e} au moyen de $x(A')$ on
trouve $$\frac{[\mathsf{n}]+1}{\mathsf{n}}\cdot\frac{x(A')}{x}\le 1$$
qui est impossible puisque $x(A')\ge x$.
\end{proof}
\label{subsectioncomplements}
Soient $\sigma\colon k\hookrightarrow\mathbf{C}$ un plongement complexe de $k$ et
$\omega_{\sigma}\in(\Omega_{A_{\sigma}}+t_{B[\sigma]})\setminus t_{B[\sigma]}$ de
norme {\'e}gale {\`a} $\updelta_{\sigma}$. Cette condition implique
que $\omega_{\sigma}$ appartient {\`a} l'orthogonal de $t_{B[\sigma]}$
dans $(t_{A_{\sigma}},\Vert.\Vert_{L,\sigma})$. Il est donc possible
de fixer une base orthonorm{\'e}e
$f_{\sigma}:=(f_{1,\sigma},\ldots,f_{g,\sigma})$ de $t_{A_{\sigma}}$
ayant les propri{\'e}t{\'e}s
suivantes~:\begin{itemize}\item[(i)] $(f_{1,\sigma},\ldots,f_{\dim
B[\sigma],\sigma})$ est une base de $t_{B[\sigma]}$,\item[(ii)]
$f_{g,\sigma}:=\omega_{\sigma}/\Vert\omega_{\sigma}\Vert_{L,\sigma}$.\end{itemize}

\subsection{Fibr{\'e} ad{\'e}lique des sections auxiliaires}\label{secaux}
Au \S~\ref{subsectionchoix}, nous avons muni le $k$-espace vectoriel
$H_{n}=\mathrm{H}^{0}(A,L^{\otimes n})$ d'une structure ad{\'e}lique
hermitienne
$\overline{H}_{n}=(H_{n},(\Vert.\Vert_{\overline{H}_{n},v})_{v})$. L'objectif
de ce paragraphe est de munir $H_{n}$ d'une structure hermitienne
diff{\'e}rente en certaines places archim{\'e}diennes de $k$,
structure obtenue par d{\'e}formation de
$\Vert.\Vert_{\overline{H}_{n},v}$. Une fois ce fibr{\'e} ad{\'e}lique
tordu d{\'e}fini, nous estimerons sa pente d'Arakelov.\par
On pose $\nu:=\dim_{k}H_{n}=n^{g}\mathrm{h}^{0}(A,L)$. Soit $\EuScript{V}$
l'ensemble des plongements complexes $\sigma$ de $k$ tels que
$\updelta_{\sigma}^{2}/\varepsilon_{\sigma}\le\theta/x$. On notera que
$\EuScript{V}$ est stable par conjugaison complexe. {\`A} chaque
plongement complexe $\sigma\colon k\hookrightarrow\mathbf{C}$ qui appartient {\`a}
$\EuScript{V}$, l'on associe l'entier $S_{\sigma}\ge 1$, qui ne
d{\'e}pend pas de $\mathsf{n}$, d{\'e}fini
par \begin{equation*}S_{\sigma}:=\left[\frac{\theta
\varepsilon_{\sigma}}{x\updelta_{\sigma}^{2}}\right]\end{equation*}et
le nombre r{\'e}el $\alpha_{\sigma}:=4^{T_{\sigma}S_{\sigma}}$. Ce
nombre $\alpha_{\sigma}$ ne d{\'e}pend que de la place $v$ de $k$
sous-jacente {\`a} $\sigma$. Soit $(s_{1},\ldots,s_{\nu})$ une base
orthonorm{\'e}e de
$(H_{n}\otimes_{\sigma}\mathbf{C},\Vert.\Vert_{\overline{H}_{n},\sigma})$. Pour
tout $i\in\{1,\ldots,\nu\}$, soit $\vartheta_{i}:t_{A_{\sigma}}\to\mathbf{C}$
la fonction th{\^{e}}ta associ{\'e}e {\`a} $s_{i}$ (voir
\S~\ref{fnth}). {\'E}tant donn{\'e} une base
$\mathsf{e}=(e_{1},\ldots,e_{g})$ de $t_{A_{\sigma}}$, un multiplet
$\tau=(\tau_{1},\ldots,\tau_{g})\in\mathbf{N}^{g}$ et un vecteur
$z=z_{1}e_{1}+\cdots+z_{g}e_{g}\in t_{A_{\sigma}}$, on note
$\frac{1}{\tau !}\mathrm{D}_{\mathsf{e}}^{\tau}\vartheta(z)$ la
d{\'e}riv{\'e}e divis{\'e}e
$\frac{1}{\tau_{1}!\cdots\tau_{g}!}\left(\frac{\partial}{\partial
z_{1}}\right)^{\tau_{1}}\cdots\left(\frac{\partial}{\partial
z_{g}}\right)^{\tau_{g}}\vartheta(z_{1}e_{1}+\cdots+z_{g}e_{g})$.\par
Soit $\Upsilon_{\sigma}$ l'ensemble des couples
$(m,\tau)\in\mathbf{Z}\times\mathbf{N}^{g}$ v{\'e}rifiant les propri{\'e}t{\'e}s
suivantes~:\begin{enumerate}\item[(i)]
$m\in\{1-S_{\sigma},\ldots,S_{\sigma}-1\}$, \item[(ii)] si $\tau$
s'{\'e}crit $(\tau_{1},\ldots,\tau_{g})$ alors
$\vert\tau\vert:=\tau_{1}+\cdots+\tau_{g}\le
gT+T_{\sigma}-1$, \item[(iii)] $\tau_{g}\le
T_{\sigma}-1$.\end{enumerate}Soit $\upsilon_{\sigma}$ le cardinal de
$\Upsilon_{\sigma}$. On a l'estimation triviale
$\upsilon_{\sigma}\le(4gS_{\sigma}T)^{g}$. Rappelons que
$f_{\sigma}=(f_{1,\sigma},\ldots,f_{g,\sigma})$ d{\'e}signe la base
orthonorm{\'e}e de $t_{A_{\sigma}}$ introduite au
\S~\ref{subsectioncomplements} et consid{\'e}rons la matrice complexe
$\mathsf{a}_{\sigma}$ de taille $\upsilon_{\sigma}\times\nu$, de
coefficients~:\begin{equation*}\mathsf{a}_{\sigma}[(m,\tau),i]:=
\left(\frac{1}{\tau!}\mathrm{D}_{f_{\sigma}}^{\tau}
\vartheta_{i}(m\omega_{\sigma})\right)\exp{\left\{-\frac{\pi}{2}n\Vert
m\omega_{\sigma}\Vert_{L,\sigma}^{2}\right\}}
\end{equation*}pour tous $(m,\tau)\in\Upsilon_{\sigma}$ et
$i\in\{1,\ldots,\nu\}$. Dans la suite, on notera $\varrho_{\sigma}$ le
rang de la matrice $\mathsf{a}_{\sigma}$.
\begin{defi}\label{definormesmodif}Posons
$\alpha:=(\alpha_{\sigma})_{\sigma\in\EuScript{V}}$. Le fibr{\'e}
ad{\'e}lique hermitien $\overline{H}_{n,\alpha}$ sur $k$ est le
fibr{\'e} vectoriel ad{\'e}lique d'espace vectoriel sous-jacent
$H_{n}$ et dont les normes sont les suivantes~: en une place $v$ de
$k$ qui n'induit aucun plongement $k\hookrightarrow\mathbf{C}$ appartenant
{\`a} $\EuScript{V}$, on pose
$\Vert.\Vert_{\overline{H}_{n,\alpha},v}:=\Vert.\Vert_{\overline{H}_{n},v}$;
si $v$ est une place archim{\'e}dienne de $k$ tel qu'un plongement
complexe $\sigma\colon k\hookrightarrow\mathbf{C}$ associ{\'e} appartienne {\`a}
$\EuScript{V}$, la norme $\Vert.\Vert_{\overline{H}_{n,\alpha},v}$
est d{\'e}finie par $$\Vert
\mathtt{x}_{1}s_{1}+\cdots+\mathtt{x}_{\nu}s_{\nu}\Vert_{\overline{H}_{n,\alpha},v}
:=\left(\vert\mathtt{x}\vert_{2}^{2}+\vert\alpha_{\sigma}
\mathsf{a}_{\sigma}(\mathtt{x})\vert_{2}^{2}\right)^{1/2}$$pour tout
$\mathtt{x}={}^{\mathrm{t}}(\mathtt{x}_{1},\ldots,\mathtt{x}_{\nu})\in\mathbf{C}^{\nu}$
(la norme $\vert.\vert_{2}$ est la norme hermitienne usuelle sur
$\mathbf{C}^{\nu}$ ou $\mathbf{C}^{\upsilon_{\sigma}}$).
\end{defi}
Aux places archim{\'e}diennes, la norme ainsi d{\'e}finie ne
d{\'e}pend pas du choix de $\sigma$ associ{\'e} {\`a}
$v$. L'estimation de la pente de $\overline{H}_{n,\alpha}$ requiert le
lemme suivant, variante de l'in{\'e}galit{\'e} de Cauchy pour les
fonctions holomorphes.
\begin{lemma}\label{lemmecauchy}Soit $\sigma\colon k\hookrightarrow\mathbf{C}$ un
plongement complexe  de $k$. Soient $s\in\mathrm{H}^{0}(A,L^{\otimes
n})\otimes_{\sigma}\mathbf{C}$ et $\vartheta$ la fonction th{\^{e}}ta
associ{\'e}e. Soit $\mathsf{e}=(e_{1},\ldots,e_{g})$ une base
orthonorm{\'e}e de $(t_{A_{\sigma}},\Vert.\Vert_{L,\sigma})$. Alors,
pour tout $z\in t_{A_{\sigma}}$, pour tout
$\tau=(\tau_{1},\ldots,\tau_{g})\in\mathbf{N}^{g}$, on
a\begin{equation*}\left\vert\frac{1}{\tau!}
\mathrm{D}_{\mathsf{e}}^{\tau}\vartheta(z)\right\vert
\exp{\left\{-\frac{\pi}{2}n\Vert
z\Vert_{L,\sigma}^{2}\right\}}\le\Vert
s\Vert_{\infty,\sigma}\exp{\left\{\frac{\pi}{2}n(1+2\Vert
z\Vert_{L,\sigma})\right\}}\end{equation*}o{\`u} $\Vert
s\Vert_{\infty,\sigma}:=\sup{\{\Vert
s(x)\Vert_{\mathrm{cub},\sigma};\ x\in A_{\sigma}\}}$.
\end{lemma}

\begin{proof}
L'in{\'e}galit{\'e} de Cauchy pour la fonction holomorphe $\vartheta$
se traduit par la majoration $$\left\vert\frac{1}{\tau!}
\mathrm{D}_{\mathsf{e}}^{\tau}\vartheta(z)\right\vert\le
\frac{1}{r^{\vert\tau\vert}}\sup{\{\vert\vartheta(z+y)\vert\,;\
y\in t_{A_{\sigma}}\ \text{et}\ \Vert y\Vert_{L,\sigma}\le
r\}}$$valide pour tout nombre r{\'e}el $r>0$. La
relation~\eqref{formulesectiontheta} entre $s$ et $\vartheta$ fournit
l'estimation $$\vert\vartheta(z+y)\vert\exp{\left\{-\frac{\pi}{2}n\Vert
z\Vert_{L,\sigma}^{2}\right\}}\le\Vert
s\Vert_{\infty,\sigma}\exp{\left\{\frac{\pi}{2}n(r^{2}+2r\Vert
z\Vert_{L,\sigma})\right\}},$$ce qui d{\'e}montre le lemme, en
choisissant $r=1$. \end{proof}
Un \emph{lemme de Gromov} assure l'existence d'une constante $c>0$,
qui ne d{\'e}pend que de $(A,L)$, telle que, pour tout $s\in
H_{n}\otimes_{\sigma}\mathbf{C}$, on a $\Vert s\Vert_{\infty,\sigma}\le
n^{c}\Vert s\Vert_{\overline{H}_{n},\sigma}$
(voir~\cite[lemme~$30$]{gilletsoule}). De ces r{\'e}sultats d{\'e}coule la
proposition suivante (rappelons que $\varrho_{\sigma}$ d\'esigne le
rang de la matrice $\mathsf{a}_{\sigma}$).
\begin{prop}\label{propositionconstruction}Il existe une constante
$c>0$, qui ne d{\'e}pend pas de $\mathsf{n}$, telle que la pente
d'Arakelov normalis{\'e}e
$\widehat{\mu}(\overline{H}_{n,\alpha})$ de
$\overline{H}_{n,\alpha}$ est minor{\'e}e
par\begin{equation*}-\sum_{\sigma\in\EuScript{V}}{\frac{\varrho_{\sigma}}
{[k:\mathbf{Q}]\nu}\left(\log\alpha_{\sigma}+\frac{\pi}{2}n(1+2S_{\sigma}\updelta_{\sigma})
\right)}-c\log\mathsf{n}. \end{equation*}\end{prop}
\begin{proof}
En vertu de la proposition~$4.0.7$ de~\cite{gaudronhuit}, la
diff{\'e}rence des pentes
$\widehat{\mu}(\overline{H}_{n,\alpha})-\widehat{\mu}(\overline{H}_{n})$
est minor{\'e}e
par $$-\sum_{\sigma\in\EuScript{V}}{\frac{\varrho_{\sigma}}{[k:\mathbf{Q}]\nu}
\left(\log(1+\alpha_{\sigma}^{2})^{1/2}+\log\max{\{1,\Vert
\mathsf{a}_{\sigma}\Vert_{\mathrm{op}}\}}\right)}$$o{\`u}
$\Vert\mathsf{a}_{\sigma}\Vert_{\mathrm{op}}$ d{\'e}signe la norme
d'op{\'e}rateur de $\mathsf{a}_{\sigma}:(\mathbf{C}^{\nu},
\vert.\vert_{2})\to(\mathbf{C}^{\upsilon_{\sigma}},\vert.\vert_{2})$. La
comparaison de cette norme avec celle de Hilbert-Schmidt conduit {\`a}
la majoration $$\Vert\mathsf{a}_{\sigma}\Vert_{\mathrm{op}}
\le(\nu\upsilon_{\sigma})^{1/2}\max{\left\{\vert\mathsf{a}_{\sigma}
[(m,\tau),i]\vert\,;\ (m,\tau)\in\Upsilon_{\sigma},\ 1\le i\le\nu\right\}}\
.$$D'apr{\`e}s le lemme~\ref{lemmecauchy} et \textit{via} la majoration
de Gromov, le maximum qui appara{\^{\i}}t ci-dessus est plus petit que
$\exp{\left\{\frac{\pi}{2}n(1+2S_{\sigma}\updelta_{\sigma})\right\}}n^{c'}$
($c'$ constante qui ne d{\'e}pend pas de $\mathsf{n}$). La partie
$(\nu\upsilon_{\sigma})^{1/2}n^{c'}$ entre dans le $c\log\mathsf{n}$
de la proposition, ainsi que la diff{\'e}rence entre
$\log(1+\alpha_{\sigma}^{2})^{1/2}$ et $\log\alpha_{\sigma}$. Quant
{\`a} la pente $\widehat{\mu}(\overline{H}_{n})$, la
formule~\eqref{formulepenterestriction} montre qu'elle fait partie
elle aussi de $c\log\mathsf{n}$. \end{proof}
\subsection{Estimation de rangs}
Dans la proposition~\ref{propositionconstruction} du paragraphe
pr{\'e}c{\'e}dent est apparu le rang $\varrho_{\sigma}$ de la matrice
$\mathsf{a}_{\sigma}$. Pour que cette proposition soit utilisable dans
la suite, il est important d'avoir une estimation soigneuse de
$\varrho_{\sigma}$, plus pr{\'e}cise que
$\varrho_{\sigma}\le\min{\{\nu,\upsilon_{\sigma}\}}$. Comme l'ont
montr{\'e} Philippon et Waldschmidt~\cite{pph-miw}, le choix de $x$ et
son incorporation dans le param{\`e}tre $n=x\mathsf{n}$ vont permettre
de faire en sorte que $\varrho_{\sigma}/\nu<1$. {\'E}tant donn{\'e} un
nombre r{\'e}el $\varepsilon$, on
note $$r(g,\varepsilon):=(g+\varepsilon)^g-g^g.$$
Si $\varepsilon\le1$, nous avons facilement
$r(g,\varepsilon)\le g^g\varepsilon(1-\varepsilon)^{-1}$.
\begin{prop}\label{propositionrang} Pour tout plongement
$\sigma\colon k\hookrightarrow\mathbf{C}$ appartenant {\`a} $\EuScript{V}$, le
quotient $\varrho_{\sigma}/\nu$ du rang $\varrho_{\sigma}$ de la
matrice $\mathsf{a}_{\sigma}$ par la dimension $\nu$ de
$\mathrm{H}^{0}(A,L^{\otimes n})$ est plus petit que
$r(g,\varepsilon_{\sigma})/\xi_{\sigma}+\mathrm{o}(1)$ o{\`u}
$\mathrm{o}(1)$ d{\'e}signe une fonction qui tend vers $0$ lorsque
$\mathsf{n}$ tend vers $\infty$.
\end{prop}

\begin{proof} Soit $g_{\sigma}:=\dim B[\sigma]$. L'id\'ee de
Philippon et Waldschmidt est de majorer $\varrho_{\sigma}$ par
$\dim E-\dim F$ o\`u $E$ est l'espace des fonctions th\^eta associ\'ees \`a
$L_\sigma^{\otimes n}$ et $F$ le sous-espace form\'e des fonctions
dont toutes les d\'eriv\'ees $D_{f_{\sigma}}^{\tau}\vartheta$ sont
identiquement nulles sur $t_{B[\sigma]}$ pour
$\tau=(0,\ldots,0,\tau_{g_{\sigma}+1},\ldots,\tau_{g})$ de longueur
$\le gT+T_{\sigma}-1$ avec, de plus, $\tau_{g}\le T_{\sigma}-1$. Pour
$-1\le\ell\le gT+T_\sigma-1$ on note aussi $F_\ell$ le sous-espace de
$E$ d\'efini de m\^eme en limitant la condition aux indices de
longueur au plus $\ell$. Nous avons donc
$F=F_{gT+T_\sigma-1}\subset F_{gT+T_\sigma-2}\subset\cdots\subset
F_0\subset F_{-1}=E$. D'un autre c\^ot\'e, $\varrho_\sigma=\dim E-\dim
G$ o\`u $G$ est le sous-espace form\'e des fonctions telles que 
$D_{f_{\sigma}}^{\tau}\vartheta(m\omega_\sigma)=0$ pour tous
$(m,\tau)\in\Upsilon_\sigma$. L'in\'egalit\'e $\varrho_\sigma\le\dim
E-\dim F$ d\'ecoule donc de $F\subset G$ : si $\vartheta\in F$ et
$(m,\tau)\in\Upsilon_\sigma$, on \'ecrit
$m\omega_\sigma\in\omega+t_{B[\sigma]}$ pour
$\omega\in\Omega_{A_\sigma}$ et l'on applique la d\'erivation
$D^\tau_{f_\sigma}$ \`a la formule
$\vartheta(z+\omega)=a_{L_\sigma}(z,\omega)^n\vartheta(z)$.
Maintenant si $\tau$ est un indice intervenant
dans la d\'efinition de $F_\ell$ et si $\vartheta\in F_{\ell-1}$ alors
la d\'eriv\'ee $D_{f_{\sigma}}^{\tau}\vartheta$ d\'efinit une fonction
th\^eta sur $t_{B[\sigma]}$ : en effet, comme pr\'ec\'edemment,
lorsque l'on d\'erive par la formule de Leibniz l'\'egalit\'e
$\vartheta(z+\omega)=a_{L_\sigma}(z,\omega)^n\vartheta(z)$ pour
$\omega\in\Omega_{B[\sigma]}$ et $z\in t_{B[\sigma]}$ alors toutes les
autres d\'eriv\'ees apparaissant sont nulles par d\'efinition de
$F_{\ell-1}$. Par suite on d\'efinit une injection de
$F_{\ell-1}/F_\ell$ dans une somme de copies de
$\mathrm H^0(B[\sigma],L_\sigma^{\otimes n})$. En sommant sur $\ell$ et en
calculant le nombre total $\mathtt X$ de copies, nous trouvons
$\varrho_{\sigma}\le\dim E-\dim F\le\texttt{X}\mathrm
h^0(B[\sigma],L_\sigma^{\otimes n})=\texttt{X}
(\deg_{L_{\sigma}}B[\sigma])n^{g_\sigma}/g_\sigma!$ o\`u
\begin{equation*}\begin{split}\mathtt{X}&:=\card\left\{\tau'=
(\tau_{g_{\sigma}+1},\ldots,\tau_{g})\in\mathbf{N}^{g-g_{\sigma}}\,;\
\vert\tau'\vert\le gT+T_{\sigma}-1\ \mathrm{et}\ \tau_{g}\le
T_{\sigma}-1\right\}\\&=\binom{gT+T_{\sigma}-1+g-g_{\sigma}}
{g-g_{\sigma}}-\binom{gT+g-g_{\sigma}}{g-g_{\sigma}}.\end{split}
\end{equation*} Si $\mathsf{y}\in\mathbf{N}$ alors le coefficient binomial
$\binom{\mathsf{x}+\mathsf{y}}{\mathsf{y}}$ est {\'e}quivalent {\`a}
$\mathsf{x}^{\mathsf{y}}/\mathsf{y}!$ lorsque $\mathsf{x}$ tend vers
$\infty$. En divisant $\varrho_{\sigma}$ par
$\nu=n^{g}(\deg_{L}A)/g!$ et gr{\^{a}}ce au choix des param{\`e}tres
$T=[\mathsf{n}]+1$, $T_{\sigma}=[\varepsilon_{\sigma}\mathsf{n}]$ et
$n=x\mathsf{n}$, nous obtenons
alors \begin{equation*}\label{borneinterquo}\frac{\varrho_{\sigma}}{\nu}
\le\binom{g}{g_{\sigma}}\times\left
((g+\varepsilon_{\sigma})^{g-g_{\sigma}}-g^{g-g_{\sigma}}\right)\times
\frac{\deg_{L_{\sigma}}B[\sigma]}{x^{g-g_{\sigma}}\deg_{L}A}+\mathrm{o}(1)
\end{equation*}lorsque
$\mathsf{n}\to+\infty$. Dans ce majorant, le dernier quotient vaut
exactement $\xi_{\sigma}^{-1}$ tandis que le facteur qui le
pr\'ec\`ede est major\'e par $r(g,\varepsilon_\sigma)$ (en utilisant
${g\choose g_\sigma}\le g^{g_\sigma}$).\end{proof}
\subsection{Construction d'une section auxiliaire}
Si $\overline{E}=(E,(\Vert.\Vert_{\overline{E},v})_{\text{$v$ place de
$k$}})$ est un fibr{\'e} vectoriel ad{\'e}lique sur $k$, la
hauteur $h_{\overline{E}}(x)$ d'un {\'e}l{\'e}ment $x\in
E\setminus\{0\}$ est le nombre
r{\'e}el~: $$h_{\overline{E}}(x):=\frac{1}{[k:\mathbf{Q}]}\sum_{v}{[k_{v}:\mathbf{Q}_{v}]\log\Vert
x\Vert_{\overline{E},v}}.$$En notant $\Delta_{k/\mathbf{Q}}$ le
discriminant absolu de $k$, le lemme de Siegel de
Bombieri-Vaaler~\cite{BombieriVaaler} affirme qu'il existe $x\in
E\setminus\{0\}$ tel que $$h_{\overline{E}}(x)\le
-\widehat{\mu}(\overline{E})+\frac{1}{2}\log\dim
E+\frac{1}{2[k:\mathbf{Q}]}\log\vert\Delta_{k/\mathbf{Q}}\vert.$$En appliquant ce
lemme {\`a} $\overline{E}=\overline{H}_{n,\alpha}$ et en utilisant les
propositions~\ref{propositionconstruction} et~\ref{propositionrang},
on a le r{\'e}sultat suivant.
\begin{prop}\label{propsiegel}
Il existe une section $s\in\mathrm{H}^{0}(A,L^{\otimes n})$ non nulle telle
que \begin{equation*}h_{\overline{H}_{n,\alpha}}(s)\le\frac{1}{[k:\mathbf{Q}]}
\sum_{\sigma\in\EuScript{V}}\frac{r(g,\varepsilon_{\sigma})}{\xi_{\sigma}}
\left(\log\alpha_{\sigma}+\frac{\pi}{2}n(1+2S_{\sigma}\updelta_{\sigma})\right)
+\mathrm{o}(\mathsf{n}).\end{equation*}
\end{prop}
\subsection{Extrapolation analytique}
{\`A} partir de maintenant, la section $s$ qui appara{\^{\i}}t est
celle construite dans la proposition~\ref{propsiegel} du paragraphe
pr{\'e}c{\'e}dent. Soit $\ell$ l'ordre d'annulation de $s$ en $0$ (le
long de $t_{A}$). La proposition~\ref{propmultiplicites} 
fournit l'estimation $\ell\le gT$.\par Soit $v$ une place
archim{\'e}dienne de $k$ telle qu'un plongement
$\sigma\colon k\hookrightarrow\mathbf{C}$ induit par cette place appartienne {\`a}
$\EuScript{V}$. Soit $\vartheta:t_{A_{\sigma}}\to\mathbf{C}$ la fonction
th{\^{e}}ta associ{\'e}e {\`a} $s$ dans
$H_{n}\otimes_{\sigma}\mathbf{C}$. Dans ce paragraphe, nous {\'e}tablissons
une majoration fine de la $v$-norme du jet de $s$ d'ordre $\ell$ en
$0$. Ceci est rendu possible par la construction de $s$ et de la norme
tordue sur $\overline{H}_{n,\alpha}$ qui implique que les d{\'e}riv{\'e}es
$\frac{1}{\tau!}\mathrm{D}_{f_{\sigma}}^{\tau}\vartheta(m\omega_{\sigma})$
sont \og{}petites\fg{} pour $(m,\tau)\in\Upsilon_{\sigma}$. {\`A}
cette fin, nous allons utiliser le lemme d'interpolation analytique du
\S~\ref{secextrapolation}.

Soit $\tau=(\tau_{1},\ldots,\tau_{g})\in\mathbf{N}^{g}$ de longueur
$\vert\tau\vert=\ell$ et posons
$\tau':=(\tau_{1},\ldots,\tau_{g-1},0)$. Pour $z\in\mathbf{C}$,
consid{\'e}rons la fonction
enti{\`e}re $$\mathtt{f}(z):=\frac{1}{\tau'!}
\mathrm{D}_{f_{\sigma}}^{\tau'}\vartheta(z\omega_{\sigma})$$
o{\`u} $f_{\sigma}$ est la base orthonorm{\'e}e de $t_{A_{\sigma}}$
introduite au \S~\ref{subsectioncomplements} (avec laquelle a
{\'e}t{\'e} construite la matrice $\mathsf{a}_{\sigma}$). Notons
$\mathrm{D}_{\omega_{\sigma}}=\updelta_{\sigma}\mathrm{D}_{f_{g,\sigma}}$
la d{\'e}riv{\'e}e dans la direction de $\omega_{\sigma}$. Pour
$h\in\mathbf{N}$, la d{\'e}riv{\'e}e divis{\'e}e $h^{\text{{\`e}me}}$ de
$\mathtt{f}$ s'{\'e}crit
$$\frac{1}{h!}\mathtt{f}^{(h)}(z)=\frac{1}{\tau'!h!}\mathrm{D}_{f_{
\sigma}}^{\tau'}\mathrm{D}_{\omega_{\sigma}}^{h}\vartheta(z\omega_{\sigma})
=\updelta_{\sigma}^{h}\cdot\frac{1}{\tau^{(h)}!}\mathrm{D}_{f_{  
\sigma}}^{\tau^{(h)}}\vartheta(z\omega_{\sigma})$$ o\`u
$\tau^{(h)}:=\tau'+(0,\ldots,0,h)$. Lorsque $h<T_{\sigma}$, la
longueur de $\tau^{(h)}$ est plus petite que
$\vert\tau\vert+T_{\sigma}-1\le gT+T_{\sigma}-1$ et la derni{\`e}re
coordonn{\'e}e de ce multiplet est plus petite
que $T_{\sigma}-1$. Par cons{\'e}quent, si $h<T_{\sigma}$, les
nombres \begin{equation}\label{nombre}\frac1{\tau^{(h)}!}
\mathrm{D}_{f_{\sigma}}^{\tau^{(h)}}\vartheta(m\omega_{\sigma})
\exp{\left\{-\frac{\pi}{2}n\Vert m\omega_{\sigma}\Vert_{L,\sigma}^{2}
\right\}}\end{equation}avec $m\in\{1-S_{\sigma},\ldots,S_{\sigma}-1\}$
sont des coordonn{\'e}es du vecteur 
$\mathsf{a}_{\sigma}(\mathtt{x})$ o{\`u} $\mathtt{x}$ est le vecteur
des coordonn{\'e}es de $s$ (voir
d{\'e}finition~\ref{definormesmodif}). La norme hermitienne du vecteur
form{\'e} par les nombres~\eqref{nombre} est donc plus petite que
$\alpha_{\sigma}^{-1}\Vert s\Vert_{\overline{H}_{n,\alpha},v}$. En utilisant la
d{\'e}finition de $\alpha_{\sigma}=4^{T_{\sigma}S_{\sigma}}$, on
trouve ainsi \begin{equation}\label{taupetit}\left\vert\frac{1}
{\updelta_{\sigma}^{h}h!}\mathtt{f}^{(h)}(m)\right\vert=\left\vert
\frac{1}{\tau^{(h)}!}\mathrm{D}_{f_{\sigma}}^{\tau^{(h)}}\vartheta
(m\omega_{\sigma})\right\vert\le4^{-T_{\sigma}S_{\sigma}}\Vert
s\Vert_{\overline{H}_{n,\alpha},v}\exp{\left\{\frac{\pi}2n
(S_{\sigma}\updelta_{\sigma})^{2}\right\}}\end{equation}valide
pour tous $m\in\{1-S_{\sigma},\ldots,S_{\sigma}-1\}$ et
$h\in\{0,\ldots, T_{\sigma}-1\}$.
Par ailleurs, le lemme~\ref{lemmecauchy} donne la
majoration \begin{equation*}\vert\mathtt{f}\vert_{S_{\sigma}}\le\Vert
s\Vert_{\infty,\sigma}\exp{\left\{\frac{\pi}{2}n(S_{\sigma}
\updelta_{\sigma}+1)^{2}\right\}}.\end{equation*}Comme nous l'avons vu avant la
proposition~\ref{propositionconstruction}, la norme $\Vert
s\Vert_{\infty,\sigma}$ peut {\^{e}}tre remplac{\'e}e par $n^{c}\Vert
s\Vert_{\overline{H}_{n},v}\le n^{c}\Vert
s\Vert_{\overline{H}_{n,\alpha},v}$ o{\`u} $c>0$ est une constante qui
ne d{\'e}pend que de $(A,L)$.\par La proposition~\ref{lana}
appliqu{\'e}e {\`a} $\mathtt{f}$ et aux param{\`e}tres $S_{\sigma}$ et
$T_{\sigma}$ et l'in{\'e}galit{\'e} de Cauchy $$\left\vert\frac{1}
{\tau_{g}!}\mathtt{f}^{(\tau_{g})}(0)\right\vert\le\vert\mathtt{f}\vert_{1}$$
donnent
alors \begin{equation*}\updelta_{\sigma}^{\tau_{g}}\left\vert\frac{1}
{\tau!}\mathrm{D}_{f_{\sigma}}^{\tau}\vartheta(0)\right\vert\le\left(
\frac{12}{4^{S_{\sigma}}}\right)^{T_{\sigma}}\mathsf{n}^{2c}\Vert
s\Vert_{\overline{H}_{n,\alpha},v}\max{\{1,\updelta_{\sigma}\}}^{T_{\sigma}}
\exp{\left\{\frac{\pi}{2}n(S_{\sigma}\updelta_{\sigma}+1)^{2}\right\}}.
\end{equation*}Cette estimation est valide pour tout
$\tau=(\tau',\tau_{g})\in\mathbf{N}^{g-1}\times\mathbf{N}$ de longueur
$\ell$. En l'utilisant si $\tau_{g}\ge T_{\sigma}$, mais en prenant
plut\^{o}t~\eqref{taupetit} (avec $m=0$) si $\tau_{g}=h<T_{\sigma}$,
on obtient dans tous les cas la borne \begin{equation}
\label{ineqfondamentale}\left\vert\frac{1}{\tau!}\mathrm{D
}_{f_{\sigma}}^{\tau}\vartheta(0)\right\vert\le\left(\frac{12}{4^{S_{\sigma}}} 
\right)^{T_{\sigma}}\mathsf{n}^{2c}\Vert s\Vert_{\overline{H}_{n,
\alpha},v}\max{\left(1,\frac{1}{\updelta_{\sigma}}\right)}^{\ell}
\exp{\left\{\frac{\pi}{2}n(S_{\sigma}\updelta_{\sigma}+1)^{2}\right\}}.
\end{equation}Nous allons la traduire en termes d'une majoration de la norme
du jet de $s$ d'ordre $\ell$ en $0$. Commen{\c{c}}ons par rappeler la
d{\'e}finition alg{\'e}brique d'un jet de section dans un cadre
g{\'e}n{\'e}ral. Soient $m$ un entier naturel, $\EuScript{A}$ un
sch{\'e}ma sur $\EuScript{S}$ et $\EuScript{L}$ un faisceau inversible
sur $\EuScript{A}$. On suppose qu'il existe une immersion ferm{\'e}e
$\epsilon:\EuScript{S}\hookrightarrow\EuScript{A}$. Notons
$\EuScript{I}$ le faisceau d'id{\'e}aux sur $\EuScript{A}$ d{\'e}fini
par $\epsilon$ et $\Omega_{\EuScript{A}/\EuScript{S}}$ le
$\mathcal{O}_{\EuScript{A}}$-module des diff{\'e}rentielles
relatives. Lorsque $\EuScript{A}\to\EuScript{S}$ est lisse le long de
$\epsilon$ ($\EuScript{I}$ r{\'e}gulier), le quotient
$\epsilon^*(\EuScript{I}^{m}/\EuScript{I}^{m+1})$ est isomorphe {\`a}
la puissance sym{\'e}trique
$S^{m}(\epsilon^*\Omega_{\EuScript{A}/\EuScript{S}})$. Une 
section $s\in\mathrm{H}^{0}(\EuScript{A},\EuScript{L})$ s'annule {\`a}
l'ordre $m$ le long de $\epsilon$ si
$s\in\mathrm{H}^{0}(X,\EuScript{I}^{m}\otimes\EuScript{L})$. Dans ce
cas, le \emph{jet d'ordre $m$ de $s$ en $\epsilon$}, not{\'e}
$\jet^{m}s(\epsilon)$, est l'image de $s$ par l'application
compos{\'e}e $$\mathrm{H}^{0}\left(\EuScript{A},\EuScript{I}^{m}
\otimes\EuScript{L}\right)\to\mathrm{H}^{0}\left(\EuScript{S},\epsilon^*(
\EuScript{I}^{m}/\EuScript{I}^{m+1})\otimes\epsilon^*\EuScript{L}\right)
\longrightarrow\mathrm{H}^{0}\left(\EuScript{S},
S^{m}\left(\epsilon^{*}\Omega_{\EuScript{A}/\EuScript{S}}\right)
\otimes\epsilon^{*}\EuScript{L}\right).$$
Notons $\exp_{\sigma}:t_{A_{\sigma}}\to A_{\sigma}$ l'application
exponentielle de $A_{\sigma}$. L'{\'e}l{\'e}ment $\vartheta $ de
$\mathrm{H}^{0}(t_{A_{\sigma}},\exp_{\sigma}^{*}L_{\sigma}^{\otimes
n})$ a un jet d'ordre $\ell$ en $0$. En consid{\'e}rant la base
duale (orthonorm{\'e}e)
$(f_{1,\sigma}^{\mathsf{v}},\ldots,f_{g,\sigma}^{\mathsf{v}})$ de
$(f_{1,\sigma},\ldots,f_{g,\sigma})$, on a 
\begin{equation*}\mathrm{jet}^{\ell}\vartheta(0)=\sum_{\vert\tau
\vert=\ell}\left(\frac{1}{\tau!}\mathrm{D}_{f_{\sigma}}^{\tau}\vartheta(0)\right)
(f_{1,\sigma}^{\mathsf{v}})^{\tau_{1}}\cdots(f_{g,\sigma}^{\mathsf{v}})^{\tau_{g}}\in
S^{\ell}(t_{A}^{\mathsf{v}})\otimes_{\sigma}\mathbf{C}\end{equation*} 
(dans cette somme, $\tau=(\tau_{1},\ldots,\tau_{g})$). La norme de ce jet est
{\'e}gale {\`a} celle du jet de $s$ car $L_{\sigma}$ est rigidifi{\'e}
isom{\'e}triquement en l'origine. Les normes sur la puissance
sym{\'e}trique $S^{\ell}(t_{A}^{\mathsf{v}})$ sont les normes quotient
de $\overline{t_{A}^{\mathsf{v}}}^{\otimes\ell}$
(voir~\cite[p. $45$]{gaudronsept}). En notant
$S^{\ell}(\overline{t_{A}^{\mathsf{v}}})$ le fibr{\'e} ad{\'e}lique
hermitien obtenu, on a
alors \begin{equation}\label{formulejety}\Vert\mathrm{jet}^{\ell}s(0)
\Vert_{S^{\ell}(\overline{t_{A}^{\mathsf{v}}}),v}\le\binom{g-1+\ell}{g-1}
\max_{\vert\tau\vert=\ell}{\left
\vert\frac{1}{\tau!}\mathrm{D}_{f_{\sigma}}^{\tau}\vartheta(0)\right\vert}
\cdotp\end{equation}De plus, gr{\^{a}}ce {\`a} la
proposition~\ref{propminima}, on a 
$$\max{\left(1,\frac{1}{\updelta_{\sigma}}\right)}\le(\deg_{L_{\sigma}}
B[\sigma])^{2}\max{\left(1,\frac{1}{\rho(A_{\sigma},L_{\sigma})}\right)}.$$
De cette observation et des majorations~\eqref{ineqfondamentale}
et~\eqref{formulejety} d\'ecoule l'{\'e}nonc{\'e} suivant.
\begin{prop}Il existe une constante $c>0$, qui ne d\'epend pas de
$\mathsf{n}$, ayant la propri{\'e}t{\'e} suivante. Soit $v$ une
place archim{\'e}dienne de $k$ telle qu'un plongement
$\sigma\colon k\hookrightarrow\mathbf{C}$ induit par $v$ soit dans
$\EuScript{V}$. Le premier jet non nul $\mathrm{jet}^{\ell}s(0)$ de
$s$ en $0$ est de $v$-norme inf{\'e}rieure {\`a}
\begin{equation}
\label{majorationfine}
12^{T_{\sigma}}
\mathsf{n}^{c}\left((\deg_{L_{\sigma}}B[\sigma])^{2}\max{\left\{1,
\frac{1}{\rho(A_{\sigma},L_{\sigma})}\right\}}\right)^{\ell}\Vert
s\Vert_{\overline{H}_{n,\alpha},v}
\exp{\left\{
\frac{\pi}{2}n(S_{\sigma}\updelta_{\sigma}+1)^{2}-T_{\sigma}S_{\sigma}\log
4\right\}}.
\end{equation}
\end{prop}
\subsection{Estimation de la hauteur du premier jet non nul}\label{sixhuit}
Le paragraphe pr{\'e}c{\'e}dent a {\'e}t{\'e} consacr{\'e} {\`a}
majorer la norme du premier jet non nul $\mathrm{jet}^{\ell}s(0)$ de
$s$ en $0$ en certaines places archim{\'e}diennes de $k$. Ces normes
ne sont qu'une partie de la hauteur du jet~: $$h_{S^{\ell}
(\overline{t_{A}^{\mathsf{v}}})}(\mathrm{jet}^{\ell}s(0))=
\frac{1}{[k:\mathbf{Q}]}\sum_{v}{[k_{v}:\mathbf{Q}_{v}]\log\Vert\mathrm{jet}^{\ell}
s(0)\Vert_{S^{\ell}(\overline{t_{A}^{\mathsf{v}}}),v}}.$$Ici nous
estimons les normes restantes en distinguant selon leur caract{\`e}re
archim{\'e}dien ou ultram{\'e}trique.
\subsubsection{Majoration de la norme en une place ultram{\'e}trique}
Soit $v$ une place ultram{\'e}trique de $k$. D'apr{\`e}s la
formule~\eqref{formnormeultra} donnant $\Vert
s\Vert_{\overline{H_{n}},v}$,  la section
$s\in\mathrm{H}^{0}(A,L^{\otimes
n})\hookrightarrow\mathrm{H}^{0}(A,L^{\otimes n})\otimes_{k}k_{v}$
s'{\'e}crit $\lambda s'$ avec
$s'\in\mathrm{H}^{0}(\mathcal{A},\mathcal{L}^{\otimes
n})\otimes_{\mathcal{O}_{k}}\mathcal{O}_{v}$ de norme $1$
($\mathcal{O}_{v}$ est l'anneau de valuation de $k_{v}$). Notons
$\mathcal{A}_{v}=\mathcal{A}\times\spec\mathcal{O}_{v}$,
$\epsilon_{v}$ sa section nulle et $\mathcal{L}_{v}$ le faisceau
inversible sur $\mathcal{A}_{v}$ induit par
$\mathcal{L}$. Vu la d{\'e}finition du jet de $s'$
appliqu{\'e}e au quadruplet
$(\EuScript{A},\EuScript{S},\epsilon,\EuScript{L})
=(\mathcal{A}_{v},\spec\mathcal{O}_{v},\epsilon_{v},\mathcal{L}_{v})$,
l'\'el\'ement $\jet^{\ell}s(0)$ vu dans $S^{\ell}(t_{A}^{\mathsf{v}})$ est
{\'e}gal {\`a} $\lambda.\jet^{\ell}s'(\epsilon_{v})$ avec
$\jet^{\ell}s'(\epsilon_{v})\in
S^{\ell}(t_{\mathcal{A}_{v}}^{\mathsf{v}})$ (par lissit{\'e} de
$\mathcal{A}_{v}\to\spec\mathcal{O}_{v}$\,; nous avons omis
$\epsilon_{v}^{*}\mathcal{L}_{v}$ car $\mathcal{L}$ est rigidifi{\'e}
en l'origine). La norme $v$-adique de
$\mathrm{jet}^{\ell}s(0)$ est calcul{\'e}e relativement au mod{\`e}le
entier $S^{\ell}(t_{\mathcal{A}}^{\mathsf{v}})$ de
$S^{\ell}(t_{A}^{\mathsf{v}})$ et on a
l'estimation \begin{equation}\label{majorationultra}
\Vert\mathrm{jet}^{\ell}s(0)\Vert_{S^{\ell}(
\overline{t_{A}^{\mathsf{v}}}),v}\le\vert\lambda\vert_{v}=\Vert 
s\Vert_{\overline{H_{n}},v}=\Vert
s\Vert_{\overline{H}_{n,\alpha},v}.\end{equation}
On peut reformuler cette majoration en disant que la taille du
sous-sch{\'e}ma formel induit par $\mathcal{A}_{v}$ vaut $1$ (car
$\mathcal{A}_{v}$ est lisse le long de l'origine), ce qui
entra{\^{\i}}ne l'int{\'e}gralit{\'e} des jets
(voir~\cite[lemme~$3.3$]{bost6}).
\subsubsection{Majoration de la norme en une place archim{\'e}dienne}
Soient $v$ une place archim{\'e}dienne de $k$ et
$\sigma\colon k\hookrightarrow\mathbf{C}$ un plongement complexe associ{\'e}. Au
moyen de l'estimation~\eqref{formulejety} du jet de $s$ en $0$ {\`a}
l'ordre $\ell$ et \textit{via} le lemme~\ref{lemmecauchy}, on
a $$\Vert\mathrm{jet}^{\ell}s(0)\Vert_{S^{\ell}(
\overline{t_{A}^{\mathsf{v}}}),v}\le\binom{g-1+gT}{g-1}e^{\pi
n/2}\Vert s\Vert_{\infty,\sigma}.$$Le lemme de Gromov fournit
alors l'existence d'une constante $c>0$, ne d{\'e}pendant que de
$(A,L)$, telle que \begin{equation}\label{majorationarchi}
\Vert\mathrm{jet}^{\ell}s(0)\Vert_{S^{\ell}(\overline{t_{A}^{\mathsf{v}}}),v}\le
e^{\pi n/2}\mathsf{n}^{c}\Vert s\Vert_{\overline{H}_{n,\alpha},v}\
.\end{equation}
\subsubsection{Hauteur du jet}
En regroupant les estimations~\eqref{majorationfine},
\eqref{majorationultra} et~\eqref{majorationarchi}, la hauteur
du premier jet non nul
v{\'e}rifie\begin{equation}\begin{split}\label{majorationduhjet}
h_{S^{\ell}(\overline{t_{A}^{\mathsf{v}}})}(\mathrm{jet}^{\ell}s(0))\le
&\,h_{\overline{H}_{n,\alpha}}(s)+c\log\mathsf{n}+\frac{2gT}{[k:\mathbf{Q}]}
\sum_{\sigma\in\EuScript{V}}{\log\deg_{L_{\sigma}}B[\sigma]}\\
&\quad+\frac{gT}{[k:\mathbf{Q}]}\sum_{\sigma\in\EuScript{V}}{\log\max{
\left\{1,\frac{1}{\rho(A_{\sigma},L_{\sigma})}\right\}}}-\frac{\log
4}{[k:\mathbf{Q}]}\sum_{\sigma\in\EuScript{V}}{T_{\sigma}S_{\sigma}}\\
&\quad+\frac{\pi n}{2}\left(1+\frac{1}{[k:\mathbf{Q}]}
\sum_{\sigma\in\EuScript{V}}{S_{\sigma}\updelta_{\sigma}(S_{\sigma}
\updelta_{\sigma}+2)}\right)+\frac{\log12}{[k:\mathbf{Q}]}\sum_{\sigma\in\EuScript{V}}
{T_{\sigma}}\end{split}\end{equation}pour une certaine constante $c>0$
qui ne d{\'e}pend pas de $\mathsf{n}$.
\subsection{Conclusion}\label{secconcl}
Rappelons que la pente maximale
$\widehat{\mu}_{\mathrm{max}}(\overline{E})$ d'un $k$-fibr{\'e}
ad{\'e}lique $\overline{E}$ est le maximum des pentes des
sous-fibr{\'e}s non nuls de $\overline{E}$. En consid{\'e}rant, pour
$\mathsf{e}\in E\setminus\{0\}$, la droite $k.\mathsf{e}$ munie des
m{\'e}triques de $\overline{E}$, on a $-h_{\overline{E}}(\mathsf{e})=
\widehat{\mu}(k.\mathsf{e},(\Vert.\Vert_{\overline{E},v})_{v})
\le\widehat{\mu}_{\mathrm{max}}(\overline{E})$. En appliquant ce
principe {\`a} $\overline{E}=S^{\ell}(\overline{t_{A}^{\mathsf{v}}})$ et
$\mathsf{e}=\mathrm{jet}^{\ell}s(0)$ on trouve
\begin{equation}\label{ineqliou}h_{S^{\ell}(\overline{t_{A}^{\mathsf{v}}})}
(\mathrm{jet}^{\ell}s(0))\ge-\widehat{\mu}_{\mathrm{max}}\left(S^{\ell}
(\overline{t_{A}^{\mathsf{v}}})\right).\end{equation}
En outre, un th{\'e}or{\`e}me de Bost {\'e}value la pente maximale
de la puissance sym{\'e}trique $\ell^{\text{{\`e}me}}$ d'un fibr{\'e}
ad{\'e}lique hermitien, qui ici
s'{\'e}crit $$\widehat{\mu}_{\mathrm{max}}\left(S^{\ell}(
\overline{t_{A}^{\mathsf{v}}})\right)\le\ell\left(
\widehat{\mu}_{\mathrm{max}}(\overline{t_{A}^{\mathsf{v}}})+(g+1/2)\log
g\right)$$(une d{\'e}monstration est donn{\'e}e au \S~$7$
de~\cite{gaudronsept}). Dans la suite on utilisera le majorant
$gT(\max{\{0,\widehat{\mu}_{\mathrm{max}}(\overline{t_{A}^{\mathsf{v}}})\}}
+(g+1/2)\log g)$. On compare cette majoration {\`a}
l'estimation~\eqref{majorationduhjet} de la hauteur de
$\mathrm{jet}^{\ell}s(0)$. On fait intervenir la majoration de
$h_{\overline{H}_{n,\alpha}}(s)$ de la proposition~\ref{propsiegel}
dans laquelle est int{\'e}gr{\'e}e la valeur de
$\alpha_{\sigma}=4^{T_{\sigma}S_{\sigma}}$. On obtient ainsi la version
\og{}d{\'e}pli{\'e}e\fg{} de l'in{\'e}galit{\'e}~\eqref{ineqliou}, que
l'on divise par $\mathsf{n}$. Puis on fait tendre $\mathsf{n}$ vers
$+\infty$. De plus nous utilisons \`a pr\'esent la valeur de
$r(g,\varepsilon)$ d\'efinie avant la
proposition~\ref{propositionrang}. Le choix explicite de
$\varepsilon_\sigma$ (\S~\ref{secchoixpara}) est fait pour donner
la borne $r(g,\varepsilon_{\sigma})
\xi_{\sigma}^{-1}\le1/2$ : on v\'erifie en effet
$(g+(6\sqrt2-8)g^{-g}\xi_\sigma)^g\le g^g+\xi_\sigma/2$ par calcul
direct si $g=2$ et en utilisant $r(g,\varepsilon)\le
g^g\varepsilon(1-\varepsilon)^{-1}$ si $g\ge3$.
\par Nous pouvons alors {\'e}crire le r{\'e}sultat brut de la
mani{\`e}re suivante. Posons \begin{equation}\begin{split}
\label{defidealephun}\aleph_{1}:=&g\max{\{0,\widehat{\mu}_{\mathrm{max}}
(\overline{t_{A}^{\mathsf{v}}})\}}+g(g+1/2)\log g+\frac{g}{[k:\mathbf{Q}]}
\sum_{\sigma\in\EuScript{V}}{\log\max{\left\{1,\frac{1}{
\rho(A_{\sigma},L_{\sigma})}\right\}}}\\&+\frac{2g}{[k:\mathbf{Q}]}
\sum_{\sigma\in\EuScript{V}}{\log\deg_{L_{\sigma}}B[\sigma]}+\frac{\log
12}{[k:\mathbf{Q}]}\sum_{\sigma\in\EuScript{V}}{\varepsilon_{\sigma}}
\end{split}\end{equation} et \begin{equation*}\label{defialephdeux}
\aleph_{2}:=1+\frac{1}{[k:\mathbf{Q}]}\sum_{\sigma\in\EuScript{V}}{S_{\sigma}
\updelta_{\sigma}(S_{\sigma}\updelta_{\sigma}+2)}+
\frac{1}{[k:\mathbf{Q}]}\sum_{\sigma\in\EuScript{V}}{\frac{1}{2}
(1+2S_{\sigma}\updelta_{\sigma})}.\end{equation*}Alors
on a\begin{equation}\label{ineqbrute}\frac{\log 2}{[k:\mathbf{Q}]}
\sum_{\sigma\in\EuScript{V}}{\varepsilon_{\sigma}S_{\sigma}}\le
\aleph_{1}+\frac{\pi}{2}x\aleph_{2}.\end{equation}
Posons $$M:=\frac{1}{[k:\mathbf{Q}]}\sum_{\sigma\in
\EuScript{V}}{\left(\frac{\varepsilon_{\sigma}}{\updelta_{\sigma}}\right)^{2}}.$$
En utilisant $[a]>a-1$ pour $a\in\mathbf{R}$ et au moyen de la d{\'e}finition de
$S_{\sigma}=[\theta\varepsilon_{\sigma}/(x\updelta_{\sigma}^{2})]$ et de la
borne $2\varepsilon_{\sigma}\le g^{-g}$, on a
\begin{equation}\label{minovingt}\frac{\log
2}{[k:\mathbf{Q}]}\sum_{\sigma\in\EuScript{V}}{\varepsilon_{\sigma}S_{\sigma}}>(\log
2)\left(\frac{\theta M}{x}-\frac1{2g^{g}}\right).\end{equation} En
utilisant l'in{\'e}galit\'e de Cauchy-Schwarz, on obtient une
majoration simple de
$\aleph_{2}$\begin{equation*}\label{majalephdeux}\aleph_{2}\le\frac32
+\frac{3\theta\sqrt{M}}{x}+\left(\frac{\theta}{x}\right)^{2}M.\end{equation*}
En reportant ces estimations dans~\eqref{ineqbrute}, on
trouve  $$\frac{\theta M}{x}\left(\log
2-\frac{\pi\theta}{2}\right)\le\left(\aleph_{1}+\frac{3\pi}{4}x+\frac{\log
2}{2g^{g}}\right)+\frac{3\pi\theta}{2}\sqrt{M},$$puis, avec
le choix de $\theta=(\log 2)/\pi$, on
a \begin{equation}\label{ineqsix}M-\left(\frac{3\pi x}{\log
2}\right)\sqrt{M}\le\frac{2\pi x}{(\log
2)^{2}}\left(\aleph_{1}+\frac{3\pi}{4}x+\frac{\log
2}{2g^{g}}\right).\end{equation}\par\noindent\textbf{Fait}~:
\emph{Soient $\alpha,\beta$ des nombres r{\'e}els positifs. Si
$M-\alpha\sqrt{M}\le\beta$ alors on a $$M\le\beta\left(
\frac{\alpha}{2\sqrt{\beta}}+\sqrt{1+\frac{\alpha^{2}}{4\beta}}
\right)^{2}.$$}\par
Le majorant de $M$ est le carr\'e de la racine positive du
trin\^ome $X^{2}-\alpha X-\beta$, ce qui justifie le fait. Ici,
ce r\'esultat fournit la majoration \begin{equation*}M\le\frac{2\pi x}{(\log
2)^{2}}\left(\max{\left\{105,\aleph_{1}+\frac{3\pi}{4}x+\frac{\log
2}{2g^{g}}\right\}}\right)\left(\sqrt{3\pi x\over280}+\sqrt{1+{3\pi
x\over280}}\right)^{2}.\end{equation*} Nous nous pla\c cons d'abord
dans le cas o\`u $x\le1/2141$ (le cas restant sera \'etudi\'e \`a la
fin de ce paragraphe). La borne ci-dessus devient 
$M\le\mathsf13,2x\max{\left\{105,\aleph_1+\frac{3\pi}{4}x+\frac{\log
2}{2g^{g}}\right\}}$. Par ailleurs, en revenant {\`a} la
d{\'e}finition de $\EuScript{V}$ (\S~\ref{secaux}), si 
$\sigma\not\in\EuScript{V}$, on
a $$\left(\frac{\varepsilon_{\sigma}}{\updelta_{\sigma}}\right)^{2}
\le\frac{x\varepsilon_\sigma}{\theta}\le\frac{x\pi}{2(\log
2)g^{g}}\cdotp$$ Cette majoration coupl{\'e}e avec la borne
obtenue pr{\'e}c{\'e}demment pour $M$
donne\begin{equation*}\begin{split}\frac{1}{[k:\mathbf{Q}]}
\sum_{\sigma\colon k\hookrightarrow\mathbf{C}}{\left(\frac{\varepsilon_{\sigma}}
{\updelta_{\sigma}}\right)^{2}}&\le13,2x\left(
\max{\left\{105,\aleph_{1}+\frac{3\pi 
x}{4}+\frac{\log 2}{2g^{g}}\right\}}+\frac{\pi}{26(\log
2)g^{g}}\right)\\ &\le13,2x\max{\left(106,\aleph_{1}+1/7\right)},
\end{split}\end{equation*} avec \`a nouveau $x\le1/2141$.
On en d{\'e}duit \begin{equation}\label{majsommedeux}
\frac{1}{[k:\mathbf{Q}]}\sum_{\sigma\colon k\hookrightarrow\mathbf{C}}{\left(\frac
{\xi_{\sigma}}{\updelta_{\sigma}}\right)^{2}}\le56,06g^{2g}x
\max\{106,\aleph_{1}+1/7\}.\end{equation}Il ne reste
plus qu'{\`a} estimer $\aleph_1$ pour conclure.
\begin{prop}
On a \begin{equation*}\aleph_{1}\le0,342g^{6}\max{\{1,h(A),
\log\deg_{L}A\}}+\frac{2g}{[k:\mathbf{Q}]}
\sum_{\sigma\colon k\hookrightarrow\mathbf{C}}{\log\deg_{L_{\sigma}}B[\sigma]}.
\end{equation*}\end{prop}
\begin{proof} La majoration de $\aleph_{1}$ repose sur la
d{\'e}finition~\eqref{defidealephun}. Le premier terme avec la pente
maximale du cotangent est estim\'e par la proposition~\ref{mumax}
et l'on utilise la majoration $\mathrm{h}^{0}(A,L)\le\deg_{L}A$.
En posant $h=\max(1,h(A),\log\deg_LA)$, on a 
$${1\over[k:\mathbf{Q}]}\sum_{\sigma\colon k\hookrightarrow\mathbf{C}}
\log\max(1,\rho(A_\sigma,L_\sigma)^{-2})\le\max\left(1,
\log{1\over[k:\mathbf{Q}]}\sum_{\sigma\colon
k\hookrightarrow\mathbf{C}}\rho(A_\sigma,L_\sigma)^{-2}\right)\le\log(11h)$$
gr\^ace \`a la proposition~\ref{coromatnet}. On aboutit alors {\`a} la
majoration \begin{equation*}\begin{split}\aleph_{1}\le&\frac{3g(g+1)h+
g\log h}{2}+\frac{2g}{[k:\mathbf{Q}]}
\sum_{\sigma\colon k\hookrightarrow\mathbf{C}}{\log\deg_{L_{\sigma}}B[\sigma]}\\
&\quad+2g^{6}\log 2+g(g+1/2)\log g+{g\over2}\log11+\frac{\log
12}{2g^{g}}.\end{split}\end{equation*} La derni{\`e}re constante est 
major{\'e}e par $1,49g^6$. De plus $(3g(g+1)h+g\log h)/2$ est plus
petit que $0,147g^{6}h$ (les coefficients num{\'e}riques sont obtenus
avec $g=2$ et $\log h\le h/e$). Lorsque $x\le 1/2141$, on
note que $h\ge\log\deg_LA\ge-\log x\ge\log2141$ puis
$0,147+1,49/\log2141\le0,342$ qui donne la proposition.\end{proof}
Le majorant dans cette proposition est toujours sup\'erieur \`a
$0,342\times 2^6\log2141>106-1/7$. \textit{Via}
l'in\'egalit\'e~\eqref{majsommedeux} et la borne
$0,342g^{6}+2g+1/7\le0,41g^{6}$ pour $g\ge
2$, on en d\'eduit le th\'eor\`eme~\ref{thmimportant} (sous
l'hypoth\`ese $x\le 1/2141$) en observant que $0,41\times56,06<23$.\par
Dans le cas o\`u $x>1/2141$, nous utilisons
$x\le(\deg_LA)^{-1/g}\le1/\sqrt2$. Nous menons les calculs exactement
de la m\^eme mani\`ere, seules les constantes num\'eriques \'evoluent
selon le tableau suivant :
$$\begin{array}{|l|c|c|c|c|c|c|}\hline x\le1/2141&13,2&1/7&56,06&0,342&0,41&23
\\\hline x\le1/\sqrt2&17,8&1,8&75,59&1,637&1,73&131\\\hline\end{array}$$
\subsection{Cas de la dimension deux} Dans ce paragraphe, nous
\'etablissons une variante du th{\'e}or{\`e}me~\ref{thmimportant},
dans le cas particulier utile pour l'application aux th{\'e}or{\`e}mes
d'isog{\'e}nies elliptiques.
\begin{prop}\label{propgdeux}Supposons que $g=2$ et que $\deg_{L}A\ge
10^{10}$. Pour chaque plongement $\sigma\colon k\hookrightarrow\mathbf{C}$, soit
$B[\sigma]$ une sous-vari{\'e}t{\'e} ab{\'e}lienne de $A_{\sigma}$
telle que $x(B[\sigma])=x$ et rappelons que $\updelta_{\sigma}$
d{\'e}signe $\min{\{\mathrm{d}_{\sigma}(\omega,t_{B[\sigma]})\,;\
\omega\in\Omega_{A_{\sigma}}\setminus t_{B[\sigma]}\}}$. Alors on 
a \begin{equation*}\begin{split}
\frac{1}{[k:\mathbf{Q}]}\sum_{\sigma\colon k\hookrightarrow\mathbf{C}}{\frac{1}{\updelta_{\sigma}^{2}}}\le
1778x&\left(\max{\{1000,\widehat{\mu}_{\mathrm{max}}
(\overline{t_{A}^{\mathsf{v}}})\}}+1,95+\log\deg_{L}A\right.\\
&\qquad\left.{}+\frac{1}{[k:\mathbf{Q}]}\sum_{\sigma\colon k\hookrightarrow\mathbf{C}}
{\log\max{\left(1,\frac{1}{\rho(A_{\sigma},L_{\sigma})}\right)}}
\right).\end{split}\end{equation*}\end{prop}
\begin{proof} Nous reprenons la d\'emonstration du
th\'eor\`eme~\ref{thmimportant}. Nous n'apportons aucun changement aux
paragraphes~\ref{subsectionchoix} \`a~\ref{sixhuit}. L'invariance de
l'\'enonc\'e par extension de corps utilis\'ee au
paragraphe~\ref{subsectionchoix} reste valable pour la
pr\'esente proposition en raison de notre convention sur la pente
maximale.\par Nous pouvons alors modifier les estimations du
paragraphe~\ref{secconcl} \`a partir de l'in\'egalit\'e
principale~\eqref{ineqbrute}. Comme $x(B[\sigma])=x$,
tous les $\xi_{\sigma}$ valent $1$ et en particulier
$\varepsilon_\sigma$ est ind\'ependant de $\sigma$. Nous notons ici
simplement $\varepsilon$ cette valeur commune~:
$\varepsilon=(3\sqrt{2}-4)/2$. L'\'egalit\'e
$x(B[\sigma])=x$ fournit aussi 
$\deg_{L_{\sigma}}B[\sigma]\le\sqrt{\deg_{L}A}$. Ces consid{\'e}rations
permettent de voir que \begin{equation*}\begin{split}\aleph_{1}\le
\widetilde{\aleph}_{1}:={}&2\max{\{1000,\widehat{\mu}_{\mathrm{max}}(
\overline{t_{A}^{\mathsf{v}}})\}}+\frac{2}{[k:\mathbf{Q}]}\sum_{\sigma\colon k
\hookrightarrow\mathbf{C}}{\log\max{\left\{1,\frac{1}
{\rho(A_{\sigma},L_{\sigma})}\right\}}}\\&+2\log\deg_{L}A
+3,77\end{split}\end{equation*} (o\`u nous employons $5\log
2+\varepsilon\log12\le3,77$). Notons en outre 
$m:=\left(\sum_{\sigma\in\EuScript{V}}{\updelta_{\sigma}^{-2}}
\right)/[k:\mathbf{Q}]$. Nous avons alors la
majoration \begin{equation*}\label{majalephdeuxbis}\aleph_{2}\le
\frac{3}{2}+\frac{3\theta\varepsilon\sqrt{m}}{x}+
\left(\frac{\theta\varepsilon}{x}\right)^{2}m.\end{equation*}
Le pendant de~\eqref{minovingt} est
$(\sum_{\sigma\in\EuScript{V}}{S_{\sigma}})/[k:\mathbf{Q}]\ge(\theta\varepsilon
m)/x-1$ et l'in{\'e}galit{\'e}~\eqref{ineqsix}
devient \begin{equation*}m-\frac{3x}{\theta\varepsilon}
\sqrt{m}\le\frac{2x}{\pi(\theta\varepsilon)^{2}}
\left(\widetilde{\aleph}_{1}+\theta\varepsilon\pi+\frac{3\pi
x}{4}\right).\end{equation*}Le fait qui suit
l'in{\'e}galit{\'e}~\eqref{ineqsix} et la minoration
$\widetilde{\aleph}_{1}\ge 2000$ conduisent {\`a} la
borne \begin{equation*}m\le\frac{2x}{\pi
(\theta\varepsilon)^{2}}\left(\widetilde{\aleph}_{1}+
\theta\varepsilon\pi+\frac{3\pi x}{4}\right)\times
\left(\frac{3}{2}\sqrt{\frac{\pi x}{4000}}+\sqrt{1+
\frac{9x\pi}{16000}}\right)^{2}.\end{equation*}
Pour passer de $m$ {\`a} la somme sur tous les plongements
$\sigma\colon k\hookrightarrow\mathbf{C}$, nous devons ajouter les termes
$\updelta_{\sigma}^{-2}$, $\sigma\not\in\EuScript{V}$, qui sont plus petits
que $x/(\theta\varepsilon)$. On obtient donc
finalement \begin{equation*}\frac{1}{[k:\mathbf{Q}]}\sum_{\sigma\colon k\hookrightarrow\mathbf{C}}
{\frac{1}{\updelta_{\sigma}^{2}}}\le\frac{2x}{\pi(\theta\varepsilon)^{2}}\left(
\widetilde{\aleph}_{1}+\frac{3\theta\varepsilon\pi}{2}+\frac{3\pi x}{4}
\right)\times\left(\frac{3}{2}\sqrt{\frac{\pi x}{4000}}+\sqrt{1+
\frac{9x\pi}{16000}}\right)^{2}.\end{equation*} Pour conclure, il
reste alors \`a observer que $x$ est plus petit que
$(\deg_{L}A)^{-1/2}\le10^{-5}$ et \`a remplacer les param{\`e}tres
$\theta$ et $\varepsilon$ par leurs valeurs. Apr\`es estimations
num\'eriques, nous aboutissons \`a la formule de l'\'enonc\'e.\end{proof}

\subsection{D{\'e}monstrations des th\'eor\`emes~\ref{perint}
et~\ref{thmintro} de l'introduction} 
\label{dthint}
\subsubsection{}
Le th\'eor\`eme~\ref{perint} d{\'e}coule du
corollaire~\ref{corollaireclef} et de l'estimation 
$g\le2\max(1,\log g!)$ qui montre que
$h(A)=h_F(A)+g\log\sqrt\pi\le\log(\pi e)\max(1,h_F(A),\log\deg_LA)$. Nous
pouvons ainsi remplacer la constante num{\'e}rique $23$ par $50$ et $h(A)$ par
$h_F(A)$ et il ne reste plus qu'\`a majorer $x$ par $(\deg_LA)^{-1/g}$
pour conclure.
\subsubsection{}
D\'eduisons maintenant le th\'eor\`eme~\ref{thmintro} du
th\'eor\`eme~\ref{perint} que nous venons d'\'etablir.

Dans le cas o\`u $g=1$, le lemme matriciel donne (comme au
\S~\ref{redell}) bien mieux. Par exemple avec le
th\'eor\`eme~\ref{matint} nous avons
$\deg_LA_\omega=\deg_LA\le14[k:\mathbf{Q}]\|\omega\|^2_{L,\sigma_0}
\max(1,h_F(A))$. Pour la suite, supposons $g\ge2$, notons
$\mathsf{d}=\deg_LA_\omega$ et montrons dans un premier temps :
$$\mathsf{d}^{1/\dim A_{\omega}}\le50
g^{2g+6}[k':\mathbf{Q}]\|\omega\|_{L,\sigma_0}^2\max(1,h_F(A),\log\mathsf{d}).$$
Il s'agit d'appliquer le th\'eor\`eme~\ref{perint} \`a
$(A_\omega,L_{|A_\omega})$ en remarquant
$\updelta((A_\omega)_{\sigma_0'},L_{\sigma_0'})\le\|\omega\|_{L,\sigma_0}$. Lorsque
$A=A_\omega$ cela nous donne exactement la formule ci-dessus. Sinon
nous utilisons la derni\`ere in\'egalit\'e du \S~\ref{sva} pour
\'ecrire $h_F(A_\omega)\le
h_F(A)+g\log(\sqrt{2\pi}\mathrm h^0(A_\omega,L)^{2})\le
(3g+1)\max(1,h_F(A),\log\mathsf{d})$ et l'on conclut avec $(\dim
A_\omega)^{2\dim A_\omega+6}(3g+1)\le g^{2g+6}$.

Pour passer \`a l'\'enonc\'e de notre th\'eor\`eme, \'ecrivons pour
all\'eger $C=50g^{2g+6}[k':\mathbf{Q}]\|\omega\|_{L,\sigma_0}^2$. Si
$\mathsf{d}^{1/\dim A_\omega}\le3,9g^3C$ alors le th\'eor\`eme est
acquis par $3,9\times50=195$. Sinon $$\log\mathsf{d}=(\dim
A_\omega)\log{\mathsf{d}^{1/\dim
A_\omega}\over C}+(\dim A_\omega)\log C\le0,221{\mathsf{d}^{1/\dim
A_\omega}\over C}+g\log C$$ (o\`u l'on utilise
$\log(3,9g^3)\le0,221\times3,9g^2$).  Nous en d\'eduisons donc $$0,779
\mathsf{d}^{1/\dim A_\omega}\le C\max(1,h_F(A),g\log
C)\le3g^3C\max(1,h_F(A),\log([k':\mathbf{Q}]\|\omega\|_{L,\sigma_0}^2))$$
et l'on conclut par $(3/0,779)\times50\le195$.

\section{Degr\'es minimaux d'isog\'enies elliptiques}

Dans cette partie, nous \'etablissons le th\'eor\`eme~\ref{isoprin} et
son corollaire.

\subsection{Rappels sur les isog\'enies de courbes elliptiques}

Soient $E_1$ et $E_2$ deux courbes elliptiques sur un m{\^e}me corps de
nombres $k$. On note $\Hom(E_1,E_2)$ l'ensemble des morphismes de
groupes alg\'ebriques $\varphi\colon E_1\to E_2$ sur une cl{\^{o}}ture
alg{\'e}brique $\overline{k}$ de $k$. Un \'el\'ement non nul de ce groupe
est une \emph{isog\'enie}. Le \emph{degr\'e} d'une
isog\'enie est le cardinal de son noyau. On pose aussi
$\deg(0)=0$. Pour $n\in\mathbf{Z}$ on note $[n]\in\Hom(E_i,E_i)$ ($1\le
i\le 2$) le morphisme de multiplication par $n$. On a $\deg[n]=n^2$.

\begin{lemma}\label{lemmeseptun} Il existe une unique bijection
$\Hom(E_1,E_2)\to\Hom(E_2,E_1), \varphi\mapsto\widehat\varphi$ telle
que \begin{itemize}\item[$(1)$]
$\widehat{\varphi}\circ\varphi=[\deg\varphi]$
\item[$(2)$] $\deg\widehat{\varphi}=\deg\varphi$ \item[$(3)$] $\varphi$ et
$\widehat\varphi$ sont d\'efinis sur les m\^emes extensions de
$k$.\end{itemize}\end{lemma}

\begin{proof} On pose $\widehat{0}=0$. Si $\varphi$ est une
isog\'enie, le groupe fini
$\Ker\varphi$ est de cardinal $\deg\varphi$ donc son exposant
divise $\deg\varphi$ donc $\Ker\varphi\subset\Ker[\deg\varphi]$. Par
suite il existe une
unique factorisation de $[\deg\varphi]$ \`a travers $\varphi$ que l'on
\'ecrit $\widehat{\varphi}\circ\varphi=[\deg\varphi]$. Si $\varphi$ est
d\'efini sur une extension $k'$ de $k$ alors il en va de m{\^e}me de
$\widehat{\varphi}$ puisque ses conjugu\'es au-dessus de $k'$ r\'ealisent la
m\^eme factorisation. De plus on a $(\deg\varphi)^2=\deg[\deg\varphi]
=\deg\widehat\varphi\circ\varphi=(\deg\widehat\varphi)(\deg\varphi)$
qui donne (2). Par ailleurs, $\varphi\circ\widehat\varphi\circ\varphi=
\varphi\circ[\deg\varphi]=[\deg\varphi]\circ\varphi$ donc, par
surjectivit\'e de $\varphi$, l'on trouve
$\varphi\circ\widehat\varphi=[\deg\varphi]$ (dans
$\Hom(E_2,E_1)$). Ceci montre $\widehat{\widehat\varphi}=\varphi$ et donc que
l'on a une bijection. De plus $\varphi$ est \'egalement d\'efini sur
toute extension o\`u $\widehat\varphi$ est d\'efini, ce qui donne
(3). L'unicit\'e est assur\'ee par (1) si $\varphi\ne0$ et par (2)
sinon.\end{proof}

Dans la suite, nous supposerons toujours que $E_1$ et $E_2$ sont
isog\`enes c'est-\`a-dire que $\Hom(E_1,E_2)\ne\{0\}$. Dans ce cas,
l'ensemble $\{\deg\varphi\,;\ \varphi\in\Hom(E_1,E_2)\setminus\{0\}\}
\subset\mathbf{N}\setminus\{0\}$ admet un \'el\'ement minimal
$\Delta$ (qui est aussi $\min\{\deg\psi\,;\
\psi\in\Hom(E_2,E_1)\setminus\{0\}\}$). Une isog\'enie
$\varphi\in\Hom(E_1,E_2)$ de degr\'e $\Delta$ sera dite
\emph{minimale}. Une isog\'enie est dite \emph{cyclique} si son noyau
est un groupe cyclique.

\begin{lemma} Une isog\'enie minimale est cyclique.\end{lemma}

\begin{proof} Soit $\varphi$ une isog\'enie minimale. Le groupe
$\Ker\varphi$ est un sous-groupe de
$\Ker[\deg\varphi]\simeq(\mathbf{Z}/\deg\varphi\mathbf{Z})^2$ donc
isomorphe \`a $\mathbf{Z}/a\mathbf{Z}\times\mathbf{Z}/b\mathbf{Z}$
avec $a\mid b\mid\deg\varphi$. Il contient donc un sous-groupe
isomorphe \`a $(\mathbf{Z}/a\mathbf{Z})^2$ qui est n\'ecessairement
$\Ker[a]$. Ainsi $\Ker[a]\subset\Ker\varphi$ donc il existe
une factorisation $\varphi=\varphi'\circ[a]$ avec
$\varphi'\in\Hom(E_1,E_2)$. On a $\deg\varphi=a^2\deg\varphi'$
donc, par minimalit\'e, $a=1$. Ceci montre bien que $\Ker\varphi$
est cyclique d'ordre $b$.\end{proof}

Pour la suite, nous distinguons deux cas :
\begin{itemize}\item[$(1)$] $E_1$ et $E_2$ sont sans multiplications
complexes. Ici $\Hom(E_1,E_2)$ est un $\mathbf{Z}$-module libre de rang
1.\item[$(2)$] $E_1$ et $E_2$ sont \`a multiplications complexes. Ici
$\Hom(E_1,E_2)$ est un $\mathbf{Z}$-module libre de rang 2.\end{itemize}

\begin{prop}\label{degre} Dans le cas (1), toute isog\'enie cyclique est
minimale. Il n'y a que deux isog\'enies minimales $\varphi$ et
$-\varphi$. Il existe une extension $k'$ de $k$ de degr\'e 1 ou 2
telle que : toute isog\'enie $\varphi\colon E_1\to E_2$ est d\'efinie
sur $k'$ ; si $k'\ne k$ aucune isog\'enie n'est d\'efinie sur $k$. Dans
le cas (2) il existe une extension $k'$ de $k$ de degr\'e dans
$\{1,2,3,4,6,8,12\}$ telle que toute isog\'enie $\varphi\colon
E_1\to E_2$ est d\'efinie sur $k'$.\end{prop}

\begin{proof} Dans le cas (1) soit $\varphi$ minimale. On a
$\Hom(E_1,E_2)=\mathbf{Z}\varphi$ et $\deg n\varphi=n^2\deg\varphi$ pour
$n\in\mathbf{Z}$ donc $\varphi$ et $-\varphi$ sont les seules isog\'enies
minimales. Si $\vert n\vert>1$, $\Ker(n\varphi)$ contient
$\Ker[n]\simeq(\mathbf{Z}/n\mathbf{Z})^2$ qui n'est pas cyclique donc
$n\varphi$ n'est pas cyclique. En g\'en\'eral, le groupe de Galois
$\gal(\overline{k}/k)$ agit sur $\Hom(E_1,E_2)\simeq\mathbf{Z}^m$ dans
le cas $(m)$, $m\in\{1,2\}$. Ceci donne un 
morphisme $\alpha\colon\gal(\overline{k}/k)\to\GL_m(\mathbf{Z})$ dont le
noyau est de la forme $\gal(\overline{k}/k')$ avec $k'$ extension
galoisienne finie de $k$. Le groupe $\gal(k'/k)$ est isomorphe \`a
l'image de $\alpha$. On sait qu'un sous-groupe de
$\GL_1(\mathbf{Z})\simeq\mathbf{Z}^\times$ est de cardinal $1$ ou $2$
tandis qu'un sous-groupe fini de $\GL_2(\mathbf{Z})$ est de cardinal dans
$\{1,2,3,4,6,8,12\}$ (voir \cite{Serre}\footnote{Si $G\subset\GL_2(\mathbf{Z})$
est fini alors $H=G\cap\SL_2(\mathbf{Z})$ est d'indice 1 ou 2 dans $G$ ;
il suffit donc de voir que $H$ est cyclique d'ordre 1,2,3,4 ou 6 et
m\^eme, en consid\'erant les valeurs propres (racines de l'unit\'e de
degr\'e 1 ou 2 sur $\mathbf{Q}$) qu'il est cyclique ; on pose $A=\sum_{B\in
H}{}^tBB$ ; si $B\in H$ alors ${}^tBAB=A$ ; on \'ecrit $A={}^tCC$ avec
$C\in\GL_2(\mathbf{R})$ donc $CBC^{-1}\in\SO_2(\mathbf{R})$ pour $B\in H$ ;
ceci montre que $H$ est isomorphe \`a un sous-groupe fini de
$\SO_2(\mathbf{R})\simeq\mathbf{R}/2\pi\mathbf{Z}$ donc il est
cyclique. On v\'erifie
que les valeurs de $\card G$ sont toutes atteintes \`a l'aide de
sous-groupes des groupes suivants : $G_1$ de cardinal 8 est engendr\'e
par $\begin{pmatrix}0&1\\ 1&0\end{pmatrix}$ et $\begin{pmatrix}1&0\\
0&-1\end{pmatrix}$ et $G_2$ de cardinal 12 est engendr\'e par
$\begin{pmatrix}0&1\\ 1&0\end{pmatrix}$ et
$\begin{pmatrix}1&1\\-1&0\end{pmatrix}$.}). Ceci donne l'assertion sur
le degr\'e de
$k'$. Dans le cas (1), si $k'\ne k$, alors l'\'el\'ement
$\sigma\in\gal(k'/k)\setminus\{\mathrm{id}\}$ agit par
$\sigma(\varphi)=-\varphi$ sur $\Hom(E_1,E_2)$. Par suite, si
$\varphi\ne0$, on a $\sigma(\varphi)\ne\varphi$ donc $\varphi$ n'est
pas d\'efini sur $k$.\end{proof}

\begin{rema} Pellarin (voir \cite[Remarque fondamentale,
p.~211]{pellarin}) affirme que dans le cas (1) l'on a toujours
$k'=k$. C'est faux car en fait sur tout corps de nombres $k$
il existe deux courbes elliptiques sans multiplications complexes qui
ne sont pas isog\`enes sur $k$ mais le sont sur $\overline{k}$. Pour le
voir, choisir $E_1$ sans multiplications complexes donn\'ee par une
\'equation de Weierstrass $y^2=x^3+ax+b$ ; choisir $c\in k\setminus
k^2$ et d{\'e}finir $E_2$ par $cy^2=x^3+ax+b$. L'application
$(x,y)\mapsto(x,c^{-1/2}y)$ d\'ecrit un isomorphisme $E_1\to E_2$
d\'efini sur $k'=k(c^{1/2})$ mais pas sur $k$ (il diff\`ere de son
conjugu\'e $(x,y)\mapsto(x,-c^{-1/2}y)$) donc d'apr\`es la proposition
aucune isog\'enie $E_1\to E_2$ n'est d\'efinie sur $k$. Si l'on veut
des exemples o\`u le degr\'e minimal d'isog\'enie $\Delta$ soit arbitraire,
on \'etend $k$ pour que $E_1$ ait un point de torsion $P$ d'ordre $\Delta$
rationnel (et l'on choisit $c$ ensuite). Alors $E_1'=E_1/\mathbf{Z} P$ et
$E_2$ sont d\'efinies sur $k$, isog\`enes sur $\overline{k}$. Le degr\'e
minimal est $\Delta$ car $E_2\to E_1\to E_1'$ est cyclique de degr\'e $\Delta$
mais elles ne sont pas isog\`enes sur $k$ car sinon $E_1$ et $E_2$ le
seraient aussi (car $E_1$ et $E_1'$ sont isog\`enes sur $k$).\end{rema}

\subsection{Cas non CM : lien avec les p\'eriodes}
\label{parsept}
Dans la situation pr\'ec\'edente, on choisit un plongement
$\sigma_0\colon k\hookrightarrow\mathbf{C}$. On abr\`ege
$\Omega_i=\Omega_{(E_i)_{\sigma_0}}$ pour $i\in\{1,2\}$.
Nous posons $A=E_1^2\times E_2^2$. L'espace tangent de $A_{\sigma_0}$
s'\'ecrit $t_{A_{\sigma_0}}=t_{(E_1)_{\sigma_0}}\oplus t_{(E_1)_{\sigma_0}}\oplus
t_{(E_2)_{\sigma_0}}\oplus t_{(E_2)_{\sigma_0}}$ et contient le r\'eseau des
p\'eriodes $\Omega_{A_{\sigma_0}}=\Omega_1^{\oplus2}\oplus\Omega_2^{\oplus2}$. Si
$\omega=(\omega_{11},\omega_{12},\omega_{21},\omega_{22})\in\Omega_{A_{\sigma_0}}$
on note $A_\omega$ la plus petite sous-vari\'et\'e ab\'elienne de $A_{\sigma_0}$
dont l'espace tangent contient $\omega$. Cette vari\'et\'e ab\'elienne
complexe est d\'efinie sur un corps de nombres : si nous notons $k'$
le plus petit sous-corps de $\mathbf{C}$ contenant $\sigma_0(k)$ sur lequel
$A_\omega$ est d\'efinie alors nous voyons $A_\omega$ comme une
vari\'et\'e ab\'elienne sur $k'$ et la vari\'et\'e ab\'elienne
complexe de d\'epart (sous-vari\'et\'e de $A_{\sigma_0}$) s'\'ecrit
$(A_\omega)_{\sigma'_0}$ si $\sigma'_0$ d\'esigne le plongement de
$k'$ dans $\mathbf{C}$ donn\'e par la d\'efinition (il \'etend $\sigma_0$
lorsque l'on voit $k'$ comme une extension de $k$).\par Rappelons que
$\Delta$ d{\'e}signe le degr\'e minimal d'isog\'enie entre $E_1$ et
$E_2$. Si nous supposons que $E_1$ et $E_2$ sont sans multiplications
complexes (cas (1) du paragraphe pr\'ec\'edent), alors
le lien entre $\Delta$ et $A_\omega$ est donn\'e par l'\'enonc\'e suivant.

\begin{theo}\label{lien} On suppose que $(\omega_{11},\omega_{12})$ est une base
de $\Omega_1$ et $(\omega_{21},\omega_{22})$ une base de
$\Omega_2$. Alors il existe une extension $k'$ de $k$ de degr\'e 1 ou
2 telle que :\begin{itemize}\item[$(1)$] $A_\omega$ est d\'efinie sur
$k'$.\item[$(2)$] $A_\omega$ est isomorphe sur $k'$ \`a $E_1\times
E_2$.\item[$(3)$] $A_\omega\cap(\{0\}^2\times E_2^2)$ est fini de
cardinal $\Delta$.\end{itemize}\end{theo}

\begin{proof} On choisit pour $k'$ l'extension sur laquelle sont d\'efinies
toutes les isog\'enies $E_1\to E_2$ et $E_2\to E_1$. La
sous-vari\'et\'e ab\'elienne $A_\omega$ est l'image d'un endomorphisme
de $A$. Celui-ci est donn\'e par 16 morphismes $E_i\to E_j$ avec $1\le
i,j\le2$ donc est d\'efini sur $k'$ et, par suite, il en va de m\^eme
de $A_\omega$. Ceci assure (1). En ce qui concerne (2) et (3)
voyons d'abord qu'il suffit de les \'etablir pour une seule p\'eriode
$\omega$. En effet, si $\omega$ et $\omega'$ satisfont les
hypoth\`eses du th\'eor\`eme, il existe deux isomorphismes $f_i\colon
E_i^2\to E_i^2$ ($1\le i\le2$) tels que l'application tangente $\mathrm{d}f$
\`a $f=f_1\times f_2$ envoie $\omega$ sur $\omega'$ :
$\mathrm{d}f(\omega)=\omega'$ (l'isomorphisme $f_i$ r\'ealise simplement le
changement de base de $(\omega_{i1},\omega_{i2})$ \`a
$(\omega_{i1}',\omega_{i2}')$). Ainsi l'espace tangent \`a une
sous-vari\'et\'e ab\'elienne $B$ de $A$ contient $\omega$ si et
seulement si l'espace tangent de $f(B)$ contient $\omega'$. Ceci
montre $f(A_\omega)=A_{\omega'}$ et, en particulier, $A_\omega$ et
$A_{\omega'}$ sont isomorphes sur leur corps de d\'efinition commun
$k'$ ($f\in\End(A)$ est lui d\'efini sur $k$). D'autre part on a
\'evidemment $f(\{0\}^2\times E_2^2)=\{0\}^2\times E_2^2$ donc les
ensembles $A_\omega\cap\{0\}^2\times E_2^2$ et
$A_{\omega'}\cap\{0\}^2\times E_2^2$ sont en bijection. Tout ceci
montre bien que (2) et (3) sont vraies pour $\omega$ si et seulement
si elles le sont pour $\omega'$. Nous allons donc les \'etablir pour
un $\omega$ particulier de fa{\c{c}}on \`a ce que $A_\omega$ admette une
description tr{\`e}s simple. Soient pour cela $\varphi\colon E_1\to E_2$
une isog\'enie minimale (sur $k'$) donc avec $\Delta=\deg\varphi$ et
$\widehat{\varphi}\colon E_2\to E_1$ telle que
$\widehat{\varphi}\circ\varphi=[\Delta]$. Soit $\psi$ le morphisme $E_1\times
E_2\to A$ d\'ecrit par $\psi(x,y)=(\widehat{\varphi}(y),x,y,\varphi(x))$. Il
est patent que l'image $\IM\psi$ est une sous-vari\'et\'e
ab\'elienne de $A$ isomorphe \`a $E_1\times E_2$ et que l'intersection
$\IM\psi\cap(\{0\}^2\times E_2^2)$ est en bijection avec le groupe
$\Ker\widehat{\varphi}$ de cardinal $\Delta$. Il nous suffit donc seulement
pour conclure de trouver $\omega$ tel que $A_\omega=\IM\psi$. En
fait, il suffit m{\^e}me de trouver $\omega$ comme dans l'{\'e}nonc{\'e}
dans l'espace tangent de $\IM\psi$ (ce qui assure
$A_{\omega}\subset\IM\psi$ par minimalit\'e) car nous avons toujours
$\dim A_\omega\ge2$ : dans le cas contraire, la projection $B$ de
$A_\omega$ sur $E_1^2$ serait une sous-vari\'et\'e ab\'elienne de
dimension 0 ou 1 dont l'espace tangent contiendrait
$(\omega_{11},\omega_{12})$. Or, en l'absence de multiplications
complexes, un tel $B$ est contenu dans un sous-groupe de la forme
$\{(x,y)\in E_1^2\,;\ nx=my\}$ pour
$(n,m)\in\mathbf{Z}^2\setminus\{(0,0)\}$. En passant \`a l'espace tangent,
on aurait $n\omega_{11}=m\omega_{12}$ qui contredirait la libert\'e de
$(\omega_{11},\omega_{12})$. Finalement il reste \`a trouver $\omega$
dans l'espace tangent de $\IM\psi$ c'est-\`a-dire tel que
$\omega_{11}=\mathrm{d}\widehat{\varphi}(\omega_{21})$ et
$\omega_{22}=\mathrm{d}\varphi(\omega_{12})$. Puisque $\mathrm{d}
\varphi\circ\mathrm{d}\widehat{\varphi}=\Delta\mathrm{id}$, la
premi\`ere condition s'\'ecrit
$\omega_{21}=\Delta\mathrm{d}\varphi(\omega_{11})$. L'existence des
deux bases adapt\'ees $(\omega_{11},\omega_{12})$ et $(\omega_{21},\omega_{22})$
d\'ecoule donc du fait que $\Omega_2/\mathrm{d}\varphi(\Omega_1)$ est un groupe
cyclique de cardinal $\Delta$ car $\varphi$ est cyclique.\end{proof}

Cet \'enonc\'e est plus ou moins classique (voir
\cite{masserwustholz1990,pellarin}) \`a part peut-\^etre l'assertion
(2) qui ne semble pas avoir \'et\'e not\'ee explicitement.

\subsection{Cas non CM : estimations}
\label{parhuit}
Nous d\'emontrons la premi\`ere assertion du
th\'eor\`eme~\ref{isoprin}. Comme nous traitons plus loin
diff\'eremment le cas o\`u $k$ poss\`ede une place r\'eelle, nous
pouvons supposer ici que toutes les places de $k$ sont complexes. Il
en va alors bien s\^ur de m\^eme des places de l'extension $k'$.

Avec les notations ci-dessus, nous imposons maintenant que
$(\omega_{11},\omega_{12})$ forme une base minimale de
$\Omega_1$. Ceci signifie que
$\Vert\omega_{11}\Vert_{L_1,\sigma_0}=\rho((E_1)_{\sigma_0},(L_1)_{\sigma_0})$
o\`u $L_1$ est l'unique polarisation principale sur $E_1$ et que
$\omega_{12}=\tau\omega_{11}$ o\`u $\tau$ appartient au domaine
fondamental de Siegel : $|\tau|\ge1$ et $\vert\RE\tau\vert\le1/2$. On \'ecrit
$y=\IM\tau$. On sait alors que
$y=\rho((E_1)_{\sigma_0},(L_1)_{\sigma_0})^{-2}$ (voir remarque~\ref{yrho}).

Nous fixons \`a pr\'esent le choix de $\sigma_0$ jusqu'ici
arbitraire, en demandant que $y$ soit minimal pour ce choix. En vertu
de la proposition~\ref{ell}, cela nous fournit $y\le1,92H$ o\`u nous
notons, ici et dans toute la suite, $H=\max(h(E_1),1000)$. Ce petit
raffinement all\`ege quelque peu les calculs qui suivent mais ne
modifie que le terme logarithmique de l'estimation finale.\par Soient
$p_{1}$ et $p_{2}$ les deux projections $E_{1}^{2}\to E_{1}$. Posons
$n=[\vert\tau\vert^2]\in\mathbf{N}\setminus\{0\}$ et
consid\'erons la polarisation $L'=p_1^*L_1^{\otimes n}\otimes
p_2^*L_1$ sur $E_1^2$ et $p$ la compos\'ee $A_\omega\to A\to
E_1^2$. D'apr\`es l'assertion (3) du th\'eo\-r\`eme~\ref{lien}, $p$
est une isog\'enie (de degr\'e $\Delta$) donc $L=p^*L'$ est ample sur
$A_\omega$ et $$\deg_{L}A_\omega=(\deg p)\deg_{L'}E_1^2=2n\Delta.$$
Ceci nous permet d'appliquer la proposition~\ref{propgdeux} au couple
$(A_\omega,L)$ sur le corps de nombres $k'$. En effet, si
$\deg_LA_\omega<10^{10}$ alors $\Delta<10^{10}$ et la
majoration~\eqref{racinedelta} que
nous allons d\'emontrer plus bas est tr\`es largement vraie.
Nous majorons $x\le(\deg_{L}A_\omega)^{-1/2}=(2n\Delta)^{-1/2}$
tandis que nous avons $\updelta_{\sigma_0'}\le\|\omega\|_{L,\sigma_0'}$
puisque $\omega\in t_{(A_{\omega})_{\sigma_0'}}$ mais $\omega\not\in
t_{B[\sigma_0']}$ par minimalit\'e. Comme
$\updelta_{\sigma_0'}=\updelta_{\overline{\sigma_0'}}$ et
$\sigma_0'\ne\overline{\sigma_0'}$, la proposition~\ref{propgdeux}
donne donc \begin{equation*}\begin{split}{2\sqrt{2n\Delta}\over
D\|\omega\|_{L,\sigma'_0}^2}\le 1778 &
\Big(\max(1000,\widehat{\mu}_{\mathrm{max}}
(\overline{t_{A_{\omega}}^{\mathsf{v}}}))
+\log(2n\Delta)+1,95\\ &\quad+{1\over D}\sum_{\sigma'\colon
k'\hookrightarrow\mathbf{C}}\log\max(1,\rho((A_{\omega})_{\sigma'},
L_{\sigma'})^{-1})\Big)\end{split}\end{equation*}o\`u
$D=[k':\mathbf{Q}]$. Nous allons maintenant estimer les termes qui
apparaissent dans cette majoration. En premier lieu, on a
\begin{equation*}\begin{split}\Vert\omega\Vert^2_{L,\sigma_0'}&
=\|(\omega_{11},\omega_{12})\|^2_{L',\sigma_0}=n\|\omega_{11}
\|^2_{L_1,\sigma_0}+\|\omega_{12}\|^2_{L_1,\sigma_0}\\
&=(n+|\tau|^2)\Vert\omega_{11}\Vert^2_{L_1,\sigma_0}=(n+|\tau|^2)
\rho((E_1)_{\sigma_0},(L_1)_{\sigma_0})^2\\ &={n+|\tau|^2\over y}
\le{n+|\tau|^2\over\sqrt{|\tau|^2-{1\over4}}}
\le{2n\over\sqrt{n-{1\over4}}}\end{split}\end{equation*}(la
derni\`ere in\'egalit\'e vient de ce que la fonction
$t\mapsto(n+t)/\sqrt{t-1/4}$ d\'ecro\^\i t sur
$[n,n+1/2]$, cro\^\i t sur $[n+1/2,n+1]$ et d'une comparaison entre
les valeurs en $n$ et $n+1$). Par un calcul analogue, si
$\omega'=(\omega'_{11},\omega'_{12},\omega'_{21},\omega'_{22})$ est
une p\'eriode de $(A_\omega)_{\sigma'}$ pour $\sigma'\colon
k'\hookrightarrow\mathbf{C}$ quelconque, nous avons $$\|\omega'\|^2_{L,\sigma'}
=n\|\omega'_{11}\|^2_{L_1,\sigma'}+\|\omega'_{12}\|^2_{L_1,\sigma'}
\ge\max(\|\omega'_{11}\|_{L_1,\sigma'},\|\omega'_{12}\|_{L_1,\sigma'})^2.$$
Si $\omega'\ne0$ on a $\omega_{11}'\ne0$ ou $\omega'_{12}\ne0$
(toujours car $p$ est une isog\'enie) donc
$$\rho((A_\omega)_{\sigma'},L_{\sigma'})
\ge\rho((E_1)_{\sigma'},(L_1)_{\sigma'}).$$
Ainsi \begin{equation*}\begin{split}\!\frac{1}{D}\!\sum_{\sigma'\colon
k'\hookrightarrow\mathbf{C}}{\!\!\log\max(1,
\rho((A_{\omega})_{\sigma'},L_{\sigma'})^{-1})}&
\le\frac{1}{2D}\!\sum_{\sigma'\colon k'\hookrightarrow\mathbf{C}}{\!\!
\log\max(1,\rho((E_1)_{\sigma'},(L_1)_{\sigma'})^{-2})}\\
&\le\frac{1}{2}\max{\left(1,\log\frac{1}{D}\sum_{\sigma'\colon
k'\hookrightarrow\mathbf{C}}{\rho((E_1)_{\sigma'},(L_1)_{\sigma'})^{-2}}\right)}\\
&\le\frac{1}{2}\max(1,\log(1,92H))
={1\over2}\log(1,92H)\end{split}\end{equation*}{\`a} nouveau avec la
proposition~\ref{ell}
(appliqu\'ee \`a $(E_1)_{k'}$). Nous avons \`a ce stade
\begin{equation*}\sqrt{\Delta}\le1778D\left(2
-\frac{1}{2n}\right)^{-1/2}\left(\max(1000,\widehat{\mu}_{\mathrm{max}}
(\overline{t_{A_{\omega}}^{\mathsf{v}}}))+\log(2n\Delta)+1,95+\frac12
\log(1,92H)\right)\end{equation*} et il nous reste \`a estimer la
pente maximale. 

\begin{lemma}\label{pentedeux} Nous avons
$$\widehat{\mu}_{\mathrm{max}}(\overline{t_{A_{\omega}}^{\mathsf{v}}})\le
h(E_1)+\log\Delta+\frac12\log\frac n\pi.$$\end{lemma}

\begin{proof}Comme on a $L=p^*L'$ pour l'isog\'enie $p$, le lemme~\ref{varmu}
donne $\widehat{\mu}_{\mathrm{max}}(\overline{t_{(A_{\omega},L)}})
\le\widehat{\mu}_{\mathrm{max}}(\overline{t_{(E_1^2,L')}})$ et, par
propri\'et\'e des pentes maximales, ce majorant vaut
$\widehat{\mu}_{\mathrm{max}}(\overline{t_{(E_1,L_1^{\otimes n})}}\oplus
\overline{t_{(E_1,L_1)}})=\max(\widehat{\mu}(
\overline{t_{(E_1,L_1^{\otimes n})}}),\widehat{\mu}(
\overline{t_{(E_1,L_1)}}))$. Par ailleurs, comme $\dim A_\omega=2$,
nous avons aussi $\widehat{\mu}_{\mathrm{max}}(
\overline{t_{A_{\omega}}^{\mathsf{v}}})=
\widehat{\mu}_{\mathrm{max}}(\overline{t_{A_{\omega}}})
-2\widehat{\mu}(\overline{t_{A_{\omega}}})$. Nous
\'evaluons les diff\'erentes pentes par le lemme~\ref{pente}. En
particulier $\max(\widehat{\mu}(
\overline{t_{(E_1,L_1^{\otimes n})}}),\widehat{\mu}(
\overline{t_{(E_1,L_1)}}))=\widehat{\mu}(
\overline{t_{(E_1,L_1)}})=-h(E_1)+(1/2)\log\pi$ et
$2\widehat{\mu}(\overline{t_{A_{\omega}}})
=-h(A_\omega)-(1/2)\log\mathrm{h}^{0}(A_\omega,L)+\log\pi$. Nous
obtenons
donc $$\widehat{\mu}_{\mathrm{max}}(\overline{t_{A_{\omega}}^{\mathsf{v}}})\le
h(A_\omega)-h(E_1)+{1\over2}\log{\deg_LA_\omega\over2\pi}.$$ On
conclut alors avec $\deg_LA_\omega=2n\Delta$,
$h(A_\omega)=h(E_1)+h(E_2)$ (d'apr\`es l'assertion (2) du
th\'eor\`eme~\ref{lien}) et $h(E_2)\le h(E_1)+(1/2)\log\Delta$.\end{proof}

Avec $\max(1000,\widehat{\mu}_{\mathrm{max}}
(\overline{t_{A_{\omega}}^{\mathsf{v}}}))\le
H+\log\Delta+\frac12\log\frac n\pi$ et quelques calculs num\'eriques
nous aboutissons \`a
$$\sqrt\Delta\le1778D\left(2-{1\over2n}\right)^{-1/2}(H+{1\over2}\log
H+\frac32\log n+2\log\Delta+2,4).$$
Si $n=1$ ceci
s'\'ecrit \begin{equation}\label{racinedelta}
\sqrt\Delta\le1778D\sqrt{2\over3}(H+{1\over2}\log
H+2\log\Delta+2,4).\end{equation}
Voyons que cette formule vaut aussi si $n\ge2$. Dans ce cas, on majore
$n\le|\tau|^2\le y^2+1/4\le(1,92H)^2+1/4\le4H^2$ donc
$$\sqrt\Delta\le1778D\sqrt{4\over7}(H+3,5\log
H+2\log\Delta+4,5).$$Avec $H\ge1000$ on a $3,5\log H+4,5\le0,03H$ et
l'estimation $1,03\sqrt{4/7}\le\sqrt{2/3}$ montre
que~\eqref{racinedelta} est encore valable (largement). \par
Nous utiliserons~\eqref{racinedelta} plus bas. Ici nous pouvons encore
simplifier cette forme brute. Toujours avec $H\ge1000$ nous \'ecrivons
$0,5\log H+2,4\le0,006H$ et donc, en employant
$1,006\times1778\sqrt{2/3}\le1461$, nous trouvons
$$\sqrt\Delta\le1461D(H+2\log\Delta).$$
Ceci entra\^\i ne \`a son tour
$$\sqrt\Delta\le 1545D(H+4\log D).$$
En effet, c'est clair si $\sqrt\Delta\le1,545.10^6D$. Sinon $$\log\Delta=2\log
D+2\log{\sqrt\Delta\over D}\le2\log D+{2\log(
1,545.10^6)\over1,545.10^6}{\sqrt\Delta\over D}\le2\log
D+1,85.10^{-5}{\sqrt\Delta\over D}$$ et l'on conclut par
$1461/(1-1461\times 3,7.10^{-5})\le1545$. Maintenant nous pouvons
encore utiliser $H+4\log D\le1,02(H+4\log D-19)$ et \'eventuellement
$H+4\log D-19\le1000\max(h(E_1)-1,\log(D/2),1)$ et, comme
$1545\times1,02\le(2,5\times10^6)^{1/2}$, nous aboutissons
\`a $$\Delta\le2,5\times10^6D^2(H+4\log D-19)^2$$
ou $$\Delta\le2,5\times10^{12}D^2\max\left(h(E_1)-1,\log\frac D2,1\right)^2.$$
Ceci donne (dans le cas sans multiplications complexes et sans place
r\'eelle) les deux premi\`eres assertions du
th\'eor\`eme~\ref{isoprin} avec $h(E_1)-1\le h_F(E_1)$ et
$D\le2[k:\mathbf{Q}]$ (th{\'e}or{\`e}me~\ref{lien}).

\subsection{Cas CM}

Soient $E_1$ et $E_2$ deux courbes elliptiques \`a multiplications
complexes isog\`enes. On les suppose d\'efinies sur un corps de
nombres $k$, on choisit un plongement $\sigma_0$ de $k$ dans $\mathbf{C}$. On
consid\`ere les extensions \`a $\mathbf{C}$ {\em via} $\sigma_0$ de $E_1$ et
$E_2$ et $\omega_1$, $\omega_2$ des p\'eriodes minimales. On forme
$A=E_1\times E_2$ et $\omega=(\omega_1,\omega_2)$ et l'on
s'int\'eresse \`a $A_\omega$. Comme plus haut $A_\omega$ est d\'efinie
sur une extension $k'$ de $k$ munie d'un plongement $\sigma_0'$
\'etendant $\sigma_0$. Ici on a $[k':k]\le12$ par la proposition~\ref{degre}.
Soient $\Delta_1=\card A_\omega\cap\{0\}\times E_2$ et
$\Delta_2=\card A_\omega\cap E_1\times\{0\}$.

\begin{lemma}\label{lsepthuit} La vari\'et\'e ab\'elienne $A_\omega$
est de dimension 1 et $\Delta_1$ et $\Delta_2$ sont finis.\end{lemma}

\begin{proof} Il suffit de v\'erifier que $A_\omega$ n'est ni $\{0\}$, ni
$E_1\times\{0\}$, ni $\{0\}\times E_2$, ni $E_1\times E_2$. Les trois
premi\`eres ne contiennent pas $\omega$ dans leur espace tangent. Si
$\varphi\colon E_1\to E_2$ est une isog\'enie alors
$\mathrm{d}\varphi(\Omega_1)\subset\Omega_2$ et donc $\End(E_2)\cdot
\mathrm{d}\varphi(\omega_1)\subset\Omega_2$ est de conoyau fini donc il
existe $N\in\mathbf{N}\setminus\{0\}$ avec
$N\Omega_2\subset\End(E_2)\cdot\mathrm{d}\varphi(\omega_1)$ donc il existe
$\psi\in\End(E_2)$ tel que
$N\omega_2=\mathrm{d}\psi\circ\mathrm{d}\varphi(\omega_1)$ ce qui
montre $A_\omega\subset\{(P_1,P_2)\in E_1\times
E_2\mid\psi\circ\varphi(P_1)=[N]P_2\}\ne E_1\times E_2$.\end{proof}

\begin{lemma}\label{septneuf} Il existe des isog\'enies
$A_\omega\rightleftharpoons E_1$, $A_\omega\rightleftharpoons E_2$ et
$E_1\rightleftharpoons E_2$ de degr\'es respectifs $\Delta_1$,
$\Delta_2$ et $\Delta_1\Delta_2$.\end{lemma}

\begin{proof} La projection ${p_1}_{|A_\omega}\colon A_\omega\to E_1$ est de
degr\'e $\Delta_1$ et $\widehat{{p_1}_{|A_\omega}}$ aussi (voir
lemme~\ref{lemmeseptun}). Il en va de m\^eme pour
$A_\omega\rightleftharpoons E_2$ puis l'on compose.\end{proof}

Nous nous int\'eressons \`a
$\rho((A_\omega)_{\sigma_0'},(L_\omega)_{\sigma_0'})$ o\`u $L_\omega$
est la polarisation principale sur la courbe elliptique
$A_\omega$. Par d\'efinition 
$\omega\in\Omega_{(A_\omega)_{\sigma_0'}}$. \'Evaluons
$\|\omega\|_{L_\omega,\sigma'_0}$. Comme $L_\omega^{\otimes
\Delta_1}=p_1^*L_1$ (o\`u $L_1$ est la polarisation principale sur $E_1$)
on a $$\|\omega\|_{L_\omega,\sigma'_0}^2={1\over
\Delta_1}\|\omega\|_{p_1^*L_1,\sigma'_0}^2={1\over
\Delta_1}\|\mathrm{d}p_1(\omega)\|_{L_1,\sigma'_0}^2={1\over
\Delta_1}\|\omega_1\|_{L_1,\sigma_0}^2.$$ De la m\^eme fa\c con, on a aussi
$\|\omega\|_{L_\omega,\sigma'_0}^2={1\over
\Delta_2}\|\omega_2\|_{L_2,\sigma_0}^2$. Par suite, on a $$
\rho((A_\omega)_{\sigma'_0},(L_\omega)_{\sigma'_0})^2\le
{1\over \Delta_1}\rho((E_1)_{\sigma_0},(L_1)_{\sigma_0})^2={1\over
\Delta_2}\rho((E_2)_{\sigma_0},(L_2)_{\sigma_0})^2.$$
Comme $1/\rho^2\ge\sqrt3/2$ sur une courbe elliptique
(remarque~\ref{yrho}), nous trouvons 
$$\rho((A_\omega)_{\sigma'_0},(L_\omega)_{\sigma'_0})^2
\le{2\over\sqrt3\max(\Delta_1,\Delta_2)}$$
et donc, en posant $D=[k':\mathbf{Q}]$ et $\delta=\max(\Delta_1,\Delta_2)/D$, $$T_\omega:={1\over[k':\mathbf{Q}]}\sum_{\sigma'\colon
k'\hookrightarrow\mathbf{C}}\rho((A_\omega)_{\sigma'},(L_\omega)_{\sigma'})^{-2}
\ge\frac{\sqrt{3}\delta}{2}.$$Notons $H=\max(1,h(E_1)+(1/2)\log(D/12\pi))$ et montrons que $\delta\le\sqrt{233}H$. Pour cela, on peut supposer $\delta\ge\sqrt{233}$ et l'on sait alors par la proposition~\ref{ell} appliqu{\'e}e avec $\sqrt{3}\delta/2$ que $$\pi\sqrt{3}\delta/2\le3\log(\sqrt{3}\delta/2)+6h(A_\omega)+8,66.$$Ici on a $h(A_\omega)\le h(E_1)+(1/2)\log
\Delta_1$ en utilisant l'isog\'enie entre $A_\omega$ et $E_1$ (voir
lemme~\ref{septneuf}). Par suite il vient $${\pi\sqrt3\over2}\delta\le3\log{\sqrt3\over2}\delta^2+6H+19,55\le6\log\delta+25,12H.$$L'in{\'e}galit{\'e} $\delta\le\sqrt{233}H$ s'obtient alors en remarquant que  $\log\delta\le((\log\sqrt{233})/\sqrt{233})\delta$ et $$\frac{25,12}{\frac{\pi\sqrt{3}}{2}-\frac{6\log\sqrt{233}}{\sqrt{233}}}<\sqrt{233}.$$On en d\'eduit $\Delta_1\Delta_2\le
D^2\delta^2\le233D^2H^2$ et nous avons donc bien prouv\'e qu'il existe
une isog\'enie entre $E_1$ et $E_2$ de degr\'e au
plus $$233[k':\mathbf{Q}]^2\max\left(1,h_F(E_1)+{1\over2}
\log{[k':\mathbf{Q}]\over12}\right)^2.$$ Avec
$[k':\mathbf{Q}]\le12[k:\mathbf{Q}]$ et $233\times144\le3,4\times10^4$,
ceci montre l'assertion du th\'eor\`eme~\ref{isoprin} dans le cas avec
multiplications complexes.

\subsection{Cas non CM avec une place r{\'e}elle}\label{secreelle}
Soit $k$ un corps de nombres qui poss{\`e}de au moins une place
r{\'e}elle. On note $\sigma_{0}:k\hookrightarrow\mathbf{C}$ un
plongement complexe induit par cette place. Soient $E_{1}$ et $E_{2}$
deux courbes elliptiques sans multiplications complexes, 
d{\'e}finies sur $k$ et isog{\`e}nes.
Pour chaque $j\in\{1,2\}$ le r{\'e}seau des
p{\'e}riodes de $(E_{j})_{\sigma_{0}}$ est de la forme
$\Omega_{j}:=\mathbf{Z}\omega_{j}\oplus\mathbf{Z}\tau_{j}\omega_{j}$
avec $\tau_{j}$ {\'e}l{\'e}ment du domaine fondamental de Siegel (on
notera $y_{j}$ sa partie imaginaire). Le caract{\`e}re r{\'e}el de la
place attach{\'e}e {\`a} $\sigma_{0}$ se traduit par l'{\'e}galit{\'e}
$\overline{(E_{j})_{\sigma_{0}}}=(E_{j})_{\overline{\sigma_{0}}}=(E_{j})_{\sigma_{0}}$,
et en particulier les r{\'e}seaux des p{\'e}riodes sont
identiques. Soit
$\mathsf{f}_{j}:\overline{(E_{j})_{\sigma_{0}}}\to(E_{j})_{\sigma_{0}}$
l'application de conjugaison complexe du
\S~\ref{subsecconju}. D'apr{\`e}s la proposition~\ref{conjugue}, la base
$(\mathrm{d}\mathsf{f}_{j}(\omega_{j}),\mathrm{d}\mathsf{f}_{j}
(\tau_{j}\omega_{j}))$ est une base minimale de $\Omega_j$.
Par antilin{\'e}arit{\'e}, on a $\mathrm{d}\mathsf{f}_{j}(\tau_{j}
\omega_{j})=\overline{\tau_{j}}\mathrm{d}\mathsf{f}_{j}(\omega_{j})$. Ainsi 
$\overline{\tau_{j}}$ et $\tau_{j}$ sont conjugu\'es par
$\SL_{2}(\mathbf{Z})$. En utilisant que $\tau_{j}$ appartient au 
domaine fondamental, on trouve 
$\vert\mathrm{Re}(\tau_{j})\vert\in\{0,1/2\}$ et le r{\'e}seau
$\Omega_{j}':=\mathbf{Z}\omega_{j}\oplus\mathbf{Z}(2\tau_{j}\omega_{j})
=\mathbf{Z}\omega_{j}\oplus\mathbf{Z}(2iy_{j}\omega_{j})$ 
est un sous-r{\'e}seau de $\Omega_{j}$ d'indice $2$. Ainsi, en
consid{\'e}rant une isog{\'e}nie $\varphi:E_{1}\to E_{2}$, on a l'inclusion
$2\mathrm{d}\varphi(\Omega_{1}')\subset
2\mathrm{d}\varphi(\Omega_{1})\subset 2\Omega_{2}\subset\Omega_{2}'$
qui entra{\^{\i}}ne l'existence d'entiers $a,b$ tels
que \begin{equation}\label{relationun}\mathrm{d}\varphi(\omega_{1})=
a\omega_{2}/2+biy_{2}\omega_{2}.\end{equation}Le  
fait que $4iy_{1}\mathrm{d}\varphi(\omega_{1})$ appartienne aussi
{\`a} $\Omega_{2}'$ se traduit par les conditions
$4by_{1}y_{2}\in\mathbf{Z}$ et $ay_{1}y_{2}^{-1}\in\mathbf{Z}$, qui
induisent par produit $4aby_{1}^{2}\in\mathbf{Z}$. Or l'on ne peut
avoir $y_{1}^{2}\in\mathbf{Q}$ car sinon $\tau_{1}$ serait quadratique
et la courbe $E_{1}$ aurait de la multiplication complexe. Ainsi, on a
n{\'e}cessairement $ab=0$ et la relation~\eqref{relationun} montre
que, pour au moins une p{\'e}riode $\omega$ dans l'ensemble
$\{(\omega_{1},\omega_{2}),(\omega_{1},2iy_{2}\omega_{2})\}$, la
sous-vari{\'e}t{\'e} ab{\'e}lienne $A_{\omega}$ de $A=E_{1}\times E_{2}$
est de dimension $1$. Par un raisonnement similaire en permutant
$\Omega_{1}'$ et $\Omega_{2}'$, on peut remplacer
$(\omega_{1},2iy_{2}\omega_{2})$ dans cette paire par
$(2iy_{1}\omega_{1},\omega_{2})$. Pour une p{\'e}riode $\omega$
idoine, la courbe elliptique $A_{\omega}$ est d{\'e}finie sur une
extension $k'$ de $k$ de degr{\'e} $D=[k':\mathbf{Q}]\le
2[k:\mathbf{Q}]$ (th{\'e}or{\`e}me~\ref{lien}) et l'on note
$\sigma_{0}'$ un plongement complexe de $k'$ prolongeant $\sigma_{0}$. Posons
$\Delta_{1}=\card A_{\omega}\cap\{0\}\times E_{2}$ et
$\Delta_{2}=\card A_{\omega}\cap E_{1}\times\{0\}$. Comme dans le cas
CM, nous disposons des lemmes~\ref{lsepthuit} (choix de $\omega$)
et~\ref{septneuf} (d{\'e}monstration inchang{\'e}e). L'obtention d'une
borne pour le degr{\'e} minimal d'isog{\'e}nie $\Delta$ repose alors sur une
majoration du produit $\Delta_{1}\Delta_{2}$ {\`a} partir de notre
lemme matriciel pour les courbes elliptiques, analogue au cas CM. Pour
l'analyse, nous allons distinguer deux cas selon la valeur prise par
$\dim A_{(\omega_{1},\omega_{2})}$.

\subsubsection{Premier cas~: $\dim A_{(\omega_{1},\omega_{2})}=1$}
C'est le cas le plus simple. Il suffit de reprendre la
d{\'e}monstration du cas CM (qui suit le lemme~\ref{septneuf}), avec
$(\omega_{1},\omega_{2})$, en changeant le $12$ au d{\'e}nominateur
dans la d{\'e}finition de $H$ par $2$ (car $D=[k':\mathbf{Q}]\le
2[k:\mathbf{Q}]$). La constante $19,55$ peut {\^{e}}tre remplac{\'e}e
par $14,18$, la valeur $25,12$ par $19,75$ et $233$ par $167$. On trouve
alors $\Delta_{1}\Delta_{2}\le167D^{2}H^{2}$. Dans ce cas, on
a\begin{equation*}\begin{split}\Delta&\le668[k:\mathbf{Q}]^{2}
\max{\left(1,h_{F}(E_{1})+\frac{1}{2}\log[k:\mathbf{Q}]\right)}^{2}\\
&\le1503[k:\mathbf{Q}]^{2}\max{\left(1,h_{F}(E_{1}),\log[k:\mathbf{Q}]
\right)}^{2}.\end{split}\end{equation*} 
\subsubsection{Deuxi{\`e}me cas~: $\dim
A_{(\omega_{1},\omega_{2})}=2$} Nous avons vu que si
$\omega\in\{(\omega_{1},2iy_{2}\omega_{2}),(2iy_{1}\omega_{1},\omega_{2})\}$
alors $\dim A_{\omega}=1$. {\'E}tudions les deux possibilit\'es. 

\paragraph{$\bullet$ $\omega=(\omega_{1},2iy_{2}\omega_{2})$.} On
a $$\rho((A_\omega)_{\sigma_0'},(L_\omega)_{\sigma_0'})^{2}\le\Vert
\omega\Vert_{L_{\omega},\sigma_{0}'}^{2}=\frac{1}{\Delta_{1}y_{1}}
=\frac{4y_{2}}{\Delta_{2}}.$$ Soient $\delta=\Delta_{1}y_{1}/D=
\Delta_{2}/(4Dy_{2})$ et \begin{equation*}H':=\max{\left(1,h_{F}(E_{1}),
\log\frac{D}{2}\right)}.\end{equation*} Nous allons montrer que  
$\delta\le 12,31H'$, ce qui permettra d'obtenir une premi{\`e}re majoration de
$\Delta_{1}\Delta_{2}$ car $\Delta_{1}\Delta_{2}=(2D\delta)^{2}y_{2}/y_{1}$.
On peut supposer $\delta\ge 12,31$. Consid{\'e}rons $T_{\omega}$ la
moyenne des $\rho((A_\omega)_{\sigma},(L_\omega)_{\sigma})^{-2}$. On
dispose de l'in{\'e}galit{\'e}  $T_{\omega}\ge\delta$, qui d\'ecoule
du calcul de la norme de $\omega$ en la place $\sigma_{0}'$. Gr{\^{a}}ce
{\`a} la proposition~\ref{ell} et {\`a} l'isog{\'e}nie entre
$A_\omega$ et $E_1$ de degr{\'e} $\Delta_{1}$ (voir
lemme~\ref{septneuf}), qui donne $h(A_\omega)\le
h(E_1)+(1/2)\log\Delta_1$, on a
\begin{equation}\label{ineqwho}\pi\delta\le3\log\delta+6h(E_{1})+ 
3\log\Delta_{1}+8,66.\end{equation} En observant que
$\Delta_{1}\le(2/\sqrt{3})D\delta$ (car $y_{1}\ge \sqrt{3}/2$) et en
se rappelant que $h_{F}(E_{1})=h(E_{1})-(\log\pi)/2$, on en d{\'e}duit
$\pi\delta\le 6\log\delta+23,61H'$. Comme $\delta\ge12,31$ on a 
$\log\delta\le(\log(12,31)/12,31)\delta$
puis $$\delta\le\frac{23,61H'}{\pi-6\times\frac{\log(12,31)}{12,31}}\le
12,31H',$$qui est le r{\'e}sultat voulu. On en
d{\'e}duit \begin{equation}\label{majorationuni}\sqrt{\Delta}\le
\sqrt{\Delta_{1}\Delta_{2}}\le24,62DH'\sqrt{\frac{y_{2}}{y_{1}}}.\end{equation}

\paragraph{$\bullet$ $\omega=(2iy_{1}\omega_{1},\omega_{2})$.} On
a  $$\rho((A_\omega)_{\sigma_0'},(L_\omega)_{\sigma_0'})^{2}\le\Vert\omega
\Vert_{L_{\omega},\sigma_{0}'}^{2}=\frac{4y_{1}}{\Delta_{1}}=\frac{1}
{\Delta_{2}y_{2}}.$$ Soient
$\delta'=\Delta_{1}/(4Dy_{1})=\Delta_{2}y_{2}/D$, $H'=\max{(1,h_{F}(E_{1}),\log(D/2))}$ et $T_{\omega}$ comme ci-dessus. Nous allons montrer que $\delta'\le18,19H'$ en proc{\'e}dant comme dans le cas pr{\'e}c{\'e}dent. On part de l'in{\'e}galit{\'e}~\eqref{ineqwho} qui reste valide ici avec $\delta'$ (que l'on peut supposer $\ge 18,19$). On majore
$\Delta_{1}=4y_{1}D\delta'$ par $25,8D^{2}\max{(1,h(E_{1}))}\delta'$
gr{\^{a}}ce {\`a} la proposition~\ref{ell} et \`a la
remarque~\ref{yrho}. En rempla{\c{c}}ant dans~\eqref{ineqwho} et en
utilisant $\log a\le a-1$ pour $a>0$, on d{\'e}duit
alors \begin{equation*}\label{treer}\pi\delta'\le6\log\delta'+6\log
D+6h(E_{1})+3\max{(1,h(E_{1}))}+15,42.\end{equation*} En distinguant
les cas $h(E_{1})\le1$ et $h(E_{1})>1$, cette majoration implique
$\pi\delta'\le6\log\delta'+39,74H'$. Comme $\delta'\ge 18,19$ on a
$\log\delta'\le(\log(18,19)/18,19)\delta'$
puis $$\delta'\le\frac{39,74H'}{\pi-6\times\frac{\log(18,19)}{18,19}}\le
18,19H'.$$ On a
alors \begin{equation*}\label{majorationdeuxi}\sqrt{\Delta}\le
\sqrt{\Delta_{1}\Delta_{2}}\le
36,38DH'\sqrt{\frac{y_{1}}{y_{2}}}.\end{equation*} Pour conclure, on
multiplie cette in\'egalit\'e par~\eqref{majorationuni}
:\begin{equation*}\Delta\le895,7
\times[k':\mathbf{Q}]^{2}\max{\left(1,h_{F}(E_{1}),\log
\frac{[k':\mathbf{Q}]}{2}\right)}^{2}\end{equation*}et 
l'on utilise $[k':\mathbf{Q}]\le 2[k:\mathbf{Q}]$ et
$895,7\times4<3583$. Ceci termine la d\'emonstration du
th\'eor\`eme~\ref{isoprin}.

\subsection{Hauteur et invariant modulaire}

Le lemme suivant se trouve dans \cite{silverman} sans explicitation de la
constante. C'est aussi une version plus fine de l'une des
in\'egalit\'es de l'encadrement (51) de \cite{pellarin}.

\begin{lemma}\label{hetj} Pour toute courbe elliptique $E$ d'invariant
$j$ nous avons $$h(E)\le{1\over12}h(j)+2,95.$$\end{lemma}

\begin{proof} Les deux membres sont invariants par extension de corps,
donc nous pouvons supposer $E/k$ semi-stable. Si nous appliquons la
formule (10) donn\'ee par Silverman~\cite{silverman} dans le cas
semi-stable, nous avons : 
$$h(j)={1\over[k:\mathbf{Q}]}\log|N_{k/\mathbf{Q}}\Delta_{E/k}|+{1\over
[k:\mathbf{Q}]}\sum_{\sigma\colon k\hookrightarrow\mathbf{C}}
\log\max(1,|j(\tau_\sigma)|)$$ o\`u nous notons $\tau_\sigma$
l'\'el\'ement du domaine fondamental de Siegel correspondant \`a
$E_\sigma$ (et dans le premier terme 
appara\^\i t la norme du discriminant minimal de $E/k$). Par ailleurs,
en posant $y_\sigma=\mathrm{Im}\tau_\sigma$, la proposition 1.1 de
\cite{silverman} fournit $$h_{F}(E)={1\over12[k:\mathbf{Q}]}
\log|N_{k/\mathbf{Q}}\Delta_{E/k}|-{1\over12[k:\mathbf{Q}]}\sum_{\sigma
\colon k\hookrightarrow\mathbf{C}}\log(|\Delta(\tau_\sigma)|y_\sigma^6).$$
En combinant et en rappelant que $h_{F}(E)=h(E)-(1/2)\log\pi$, nous avons
$$h(E)-{1\over12}h(j)={1\over2}\log\pi-{1\over12[k:\mathbf{Q}]}
\sum_{\sigma\colon k\hookrightarrow\mathbf{C}}{\log\left(\max{(1,\vert
j(\tau_\sigma)\vert)}|\Delta(\tau_\sigma)|y_\sigma^6\right)}.$$
D'apr\`es l'estimation situ\'ee au 
bas de la page 256 de \cite{silverman}, nous pouvons \'ecrire
$$|\Delta(\tau_\sigma)|\ge e^{-1/9-2\pi y_\sigma}(2\pi)^{-12}.$$
Par ailleurs Faisant et Philibert donnent la minoration
$|j(\tau_\sigma)|\ge e^{2\pi y_\sigma}-1193$ (lemme 1, (iii)
de~\cite[p. 187]{faisantphilibertJNT}; la preuve est dans le
texte~\cite[(3) p. 2.6]{upmc}). Par suite, nous
avons $$U_\sigma=\max{(1,\vert j(\tau_\sigma)\vert)}
|\Delta(\tau_\sigma)|y_\sigma^6\ge\max(1,e^{2\pi y_\sigma}-1193)
e^{-1/9-2\pi y_\sigma}(2\pi)^{-12}y_\sigma^6$$
$$\ge e^{-1/9}(2\pi)^{-12}f(y_\sigma)$$ o\`u $f$ est la fonction
donn\'ee par $$f(y)=\max{(y^6e^{-2\pi y},y^6(1-1193e^{-2\pi y}))}.$$
Une rapide \'etude de fonction montre que $f$ est croissante sur
$[\sqrt{3}/2,3/\pi]$ et sur $[(\log 1194)/2\pi,+\infty[$ tandis qu'elle
est d\'ecroissante sur $[3/\pi,(\log1194)/2\pi]$. Comme de plus le
calcul montre que $f((\log 1194)/2\pi)<f(\sqrt{3}/2)$ nous avons pour
tout $y\ge\sqrt{3}/2$ la minoration $f(y)\ge f((\log 1194)/2\pi)$ et
donc pour tout $\sigma$ la quantit\'e $U_\sigma$ est minor\'ee par $1/B$
o\`u $B$ est la
constante $$B=1194\left(2\pi\over\log(1194)\right)^6e^{1/9}(2\pi)^{12}.$$
En revenant au calcul de hauteur nous
avons $$h(E)-{1\over12}h(j)\le{1\over2}\log\pi+{1\over12}\log
B\le2,95$$apr\`es estimation num\'erique.\end{proof}

\subsection{Cas non CM : application}

Nous d\'emontrons le corollaire~\ref{pabi}. Soient $p$ et $E$ comme dans
l'\'enonc\'e. Nous raisonnons par l'absurde en supposant que l'image
de la repr\'esentation galoisienne est contenue dans le normalisateur
d'un sous-groupe de Cartan d\'eploy\'e. Ceci entra{\^\i}ne
notamment que l'invariant modulaire $j$ de $E$
est entier : $j\in\mathbf{Z}$ (voir~\cite[appendice]{parentbilu}). Alors 
le th\'eor\`eme 2.1 de  \cite{parentbilurebolledo} (version explicite
du r\'esultat principal de \cite{parentbilu}) montre
$$\log|j|\le2\pi\sqrt p+6\log p+21{(\log p)^2\over\sqrt p}.$$
De plus, dans la partie 5 de \cite{parentbilu} (voir aussi la partie 3
de \cite{parentbilurebolledo}), on construit deux courbes $E_1$
et $E_2$ de sorte que d'une part $E$ et $E_1$ sont reli\'ees par une
isog\'enie de degr\'e $p$ donc $h(E_1)\le h(E)+(1/2)\log p$ et d'autre part
$E_1$ et $E_2$ sont reli\'ees par une isog\'enie cyclique de degr\'e
$p^2$. Ceci fait que dans les notations des paragraphes~\ref{parsept}
et~\ref{parhuit}, on 
a $\sqrt\Delta=p$ (nos courbes sont toutes sans multiplications
complexes comme $E$). En outre la construction montre que $E_1$ et
$E_2$ ainsi que l'isog\'enie cyclique sont d\'efinies sur un corps $k$
quadratique. Ceci assure que toutes les isog\'enies entre $E_1$ et
$E_2$ sont d\'efinies sur $k$ et donc il en va de m\^eme de
$A_\omega$. Par suite $k'=k$ et $D=2$. Si $k$ est imaginaire, nous
avons d'apr\`es la majoration~\eqref{racinedelta} (qui suit le
lemme~\ref{pentedeux}) $$p\le2\sqrt{2\over3}1778(H+4\log
p+0,5\log H+2,4)$$ o\`u $H=\max(h(E_1),10^3)\le\max(h(E)+(1/2)\log
p,10^3)$. Dans le cas r\'eel, cette estimation est tr\`es largement
vraie (le th\'eor\`eme~\ref{isoprin} montre $p\le2\sqrt{3583}H$).
En combinant le lemme~\ref{hetj} et la majoration de
$\log|j|=h(j)$ donn\'ee ci-dessus, nous trouvons
$$H\le\max\left(1000,{\pi\over6}\sqrt p+\log p+{7(\log p)^2\over4\sqrt
p}+2,95\right)$$ puis 
\begin{equation*}\begin{split}p\le2\sqrt{2\over3}1778&
\left(\max\left(1000,{\pi\over6}\sqrt p+\log p+{7(\log p)^2\over4\sqrt
p}+2,95\right)+4\log p+2,4\right.\\
&\quad+\left.0,5\log\max\left(1000,{\pi\over6}\sqrt p+\log p+{7(\log
p)^2\over4\sqrt p}+2,95\right)\right).\end{split}\end{equation*} Si
l'on divise par 
$p$ de chaque c\^ot\'e on obtient une majoration de la forme 
$1\le f(p)$ pour une fonction $f$ d\'ecroissante sur
$[1,+\infty[$. Le calcul montre que $f(3\,094\,028)<1<f(3\,094\,027)$
et nous en d\'eduisons que l'on a $p\le3\,094\,027$.

\section{Appendice}

L'objectif de cet appendice est de d\'emontrer le th\'eor\`eme
de Bost utilis\'e dans le travail d'Autissier~\cite{autissier}. Il est
\'enonc\'e dans les notes \cite[p. 5]{bost4} et repris dans
\cite[p. 100]{graftieaux} (voir l'in\'egalit\'e (13) et la derni\`ere
\'egalit\'e de la page o\`u l'on corrige l'exposant $g/4$ en $1/4$) mais
aucune d\'emonstration ne semble avoir \'et\'e publi\'ee \`a ce jour.

Soit donc une vari\'et\'e ab\'elienne $A$ d\'efinie sur un corps de
nombres $k$ et munie d'une polarisation principale $L$. Pour tout plongement
$\sigma\colon k\hookrightarrow\mathbf{C}$, la vari\'et\'e ab\'elienne
complexe $A_\sigma$ obtenue par extension des scalaires est
principalement polaris\'ee et donc isomorphe \`a un unique
$\mathbf{C}^g/(\mathbf{Z}^g+\tau_\sigma\mathbf{Z}^g)$ avec $\tau_\sigma$
dans le domaine fondamental de Siegel. 

Notons $y_{\sigma}:=\IM\tau_{\sigma}$. Soit $F_{\sigma}
\colon\mathbf{C}^g\to\mathbf C$ d\'efinie par, si
$z=\tau_{\sigma}p+q\in\mathbf{C}^g$ avec $p,q\in\mathbf{R}^{g}$,
$$F_\sigma(z)=\det(2y_\sigma)^{1/4}\sum_{n\in\mathbf Z^g}\exp
\left(i\pi{}^{\mathrm t}(n+p)\tau_\sigma(n+p)+2i\pi{}^{\mathrm t}nq\right).$$

Le th\'eor\`eme de Bost s'\'ecrit alors sous la forme
suivante. 

\begin{theo}\label{thBost} Soit $a:=-(h(A)+(g/2)\log(2\pi))/2$. On a
$$a\le\frac{1}{[k:\mathbf{Q}]}\sum_{\sigma\colon
k\hookrightarrow\mathbf{C}}{\int_{(\mathbf{R}^{g}/\mathbf{Z}^{g})^{2}}{\log
|F_{\sigma}(\tau_{\sigma}p+q)|\,\mathrm{d}p\mathrm{d}q}}.$$\end{theo} 

Nous commen\c cons par quelques propri\'et\'es de la fonction $F_\sigma$.
Elles font intervenir la donn\'ee d'Appell-Humbert
$(H_\sigma,\chi_\sigma)$ sur $\mathbf{C}^g/\mathbf{Z}^g
+\tau_\sigma\mathbf{Z}^g$ d\'efinie par $$H_\sigma(z,z')={}^\mathrm t\overline
zy_\sigma^{-1}z'\qquad\mathrm{et}\qquad\chi(\tau_\sigma
m+n)=(-1)^{{}^\mathrm tmn}$$ o\`u $z,z'\in\mathbf{C}^g$ et $m,n\in\mathbf{Z}^g$.

\begin{lemma}\label{tht} Soit $\sigma$ un plongement complexe de $k$.
\begin{itemize}\item[$(1)$] Nous
avons $$\int_{(\mathbf{R}^g/\mathbf{Z}^g)^2}|F_\sigma(\tau_\sigma
p+q)|^2\,\mathrm{d}p\mathrm{d}q=1.$$\item[$(2)$]
Si nous posons $\vartheta_\sigma(z)=F_\sigma(z)\exp((\pi/2){}^\mathrm
tzy_\sigma^{-1}z-i\pi{}^\mathrm tp\tau_\sigma p)$ pour $z=\tau_\sigma
p+q$ avec $p,q\in\mathbf{R}^g$ alors
$\vartheta_\sigma\colon\mathbf{C}^g\to\mathbf{C}$ est une fonction th\^eta 
associ\'ee \`a la donn\'ee d'Appell-Humbert $(H_\sigma,\chi_\sigma)$.
\item[$(3)$] Nous avons $|\vartheta_\sigma(z)|=|F_\sigma(z)|
\exp((\pi/2)H_\sigma(z,z))$ pour tout $z\in\mathbf{C}^g$.\end{itemize}\end{lemma}
\begin{proof} Pour (1), en \'evaluant $|F_\sigma|^2$, l'int\'egrale
 \`a calculer vaut \begin{equation*}\begin{split}
&\int_{(\mathbf{R}^g/\mathbf{Z}^g)^2}{\det(2
y_\sigma)^{1/2}\sum_{n,m\in\mathbf{Z}^g}e^{2i\pi{}^\mathrm t(n-m)
q+i\pi{}^\mathrm t(n+p)\tau_\sigma(n+p)-i\pi{}^\mathrm t(m+p)\overline
{\tau_\sigma}(m+p)}\mathrm{d}p\,\mathrm{d}q}\\&=\det(2y_\sigma)^{1/2}
\sum_{n,m\in\mathbf{Z}^g}\int_{\mathbf{R}^g/\mathbf{Z}^g}e^{i\pi
{}^\mathrm t(n+p)\tau_\sigma(n+p)-i\pi{}^\mathrm t(m+p)\overline
{\tau_\sigma}(m+p)}\left(\int_{\mathbf{R}^g/\mathbf{Z}^g}
e^{2i\pi{}^\mathrm t(n-m)q}\mathrm{d}q\right)\mathrm{d}p\\&=\det(2
y_\sigma)^{1/2}\sum_{n\in\mathbf{Z}^g}{\int_{\mathbf{R}^g/\mathbf{Z}^g}
{e^{-2\pi{}^\mathrm t(n+p)y_\sigma(n+p)}\mathrm{d}p}}\\&=\det(2
y_\sigma)^{1/2}\int_{\mathbf{R}^g}
{e^{-2\pi{}^\mathrm tpy_\sigma p}\mathrm{d}p}\\&=1\end{split}
\end{equation*}(pour d{\'e}montrer la derni{\`e}re {\'e}galit{\'e}, on
peut remplacer $2y_{\sigma}$ par l'identit{\'e} \textit{via} un
changement de variables lin{\'e}aire; elle se r{\'e}duit alors {\`a}
$\int_{\mathbf{R}}{e^{-x^{2}}\,\mathrm{d}x}=\sqrt{\pi}$).  Pour (2),
un premier calcul donne pour $z=\tau_\sigma p+q$ et
$\omega=\tau_\sigma m+n$ la relation $F_\sigma(z+\omega)=
F_\sigma(z)\exp(-2i\pi{}^\mathrm tmq)$ o\`u $m,n\in\mathbf{Z}^g$ et
$p,q\in\mathbf{R}^g$. Avec les m\^emes notations ceci montre que
$\vartheta_\sigma(z+\omega)\vartheta_\sigma(z)^{-1}$
est l'exponentielle du nombre complexe $${\pi\over2}\left({}^\mathrm
t(z+\omega)y_\sigma^{-1}(z+\omega)-{}^\mathrm tzy_\sigma^{-1}z\right)
-i\pi{}^\mathrm t(p+m)\tau_\sigma(p+m)
+i\pi{}^\mathrm tp\tau_\sigma p-2i\pi{}^\mathrm tmq.$$ Apr\`es un
calcul \'el\'ementaire, cette quantit\'e se
transforme en $${\pi\over2}{}^\mathrm t\overline\omega y_\sigma^{-1}\omega
+\pi{}^\mathrm t\overline\omega y_\sigma^{-1}z+i\pi{}^\mathrm tmn.$$
Ceci nous fournit la
relation $$\vartheta_\sigma(z+\omega)=\vartheta_\sigma(z)\chi_\sigma(\omega)
\exp\left(\pi H_\sigma(\omega,z)+{\pi\over2}H_\sigma(\omega,\omega)\right)$$ qui
montre bien que $\vartheta_\sigma$ est une fonction th\^eta pour le facteur
d'automorphie introduit au \S~\ref{fnth} (voir aussi~\cite[lemme
3.2.4]{lange} pour un r\'esultat semblable). Le caract\`ere
holomorphe de $\vartheta_\sigma$ se lit sur la relation 
$$e^{-i\pi{}^\mathrm tp\tau_\sigma p}\sum_{n\in\mathbf Z^g}\exp
\left(i\pi{}^{\mathrm t}(n+p)\tau_\sigma(n+p)+2i\pi{}^{\mathrm
t}nq\right)=\sum_{n\in\mathbf Z^g}\exp
\left(i\pi{}^{\mathrm t}n\tau_\sigma n+2i\pi{}^{\mathrm
t}nz\right).$$ Enfin pour (3) il s'agit de voir que le nombre
$(\pi/2){}^\mathrm tzy_\sigma^{-1}z-i\pi{}^\mathrm tp\tau_\sigma
p-(\pi/2)H_\sigma(z,z)$ est un imaginaire pur. On constate alors
simplement qu'il vaut $i\pi{}^\mathrm tpq$.\end{proof} 

Nous en venons maintenant au lemme-clef en vue de la d\'emonstration du
th\'eor\`eme~\ref{thBost} qui relie la hauteur de N\'eron-Tate aux
fonctions $F_\sigma$. Pour pouvoir l'exprimer, nous avons besoin de
pr\'eciser le choix d'isomorphisme entre $A_\sigma$ et $\mathbf{C}^g
/\mathbf{Z}^g+\tau_\sigma\mathbf{Z}^g$. Il est li\'e \`a un choix de
repr\'esentant pour la polarisation $L$. Tout d'abord nous pouvons faire, 
dans l'\'enonc\'e du th\'eor\`eme~\ref{thBost}, une
extension finie du corps de base de mani\`ere transparente. Pour ne
pas alourdir les notations, ici et ci-dessous, nous conservons les
notations $F_\sigma$, $\tau_\sigma$ et $y_\sigma$ pour un plongement
$\sigma$ d'un sur-corps de $k$ : il est entendu que l'on parle en fait
de $F_{\sigma_{|k}}$ et ainsi de suite.

Nous profitons de cette libert\'e pour supposer que $L$ admet sur $k$
un repr\'esentant sym\'etrique et nous le fixons une fois pour
toutes. Nous notons aussi $E$ le diviseur effectif de $A$
associ\'e. Nous pouvons alors fixer de mani\`ere unique l'isomorphisme
entre $A_\sigma$ et
$\mathbf{C}^g/\mathbf{Z}^g+\tau_\sigma\mathbf{Z}^g$ en exigeant que
$L_\sigma$ corresponde au faisceau inversible sym\'etrique de donn\'ee
d'Appell-Humbert $(H_\sigma,\chi_\sigma)$ introduite plus haut. Pour
all\'eger les notations, nous {\em identifions} les vari\'et\'es
ab\'eliennes $A_\sigma$ et
$\mathbf{C}^g/\mathbf{Z}^g+\tau_\sigma\mathbf{Z}^g$. En particulier
$t_{A_\sigma}$ est identifi\'e \`a $\mathbf{C}^g$ et la fonction
$\vartheta_\sigma$ du lemme pr\'ec\'edent est une fonction th\^eta
associ\'ee \`a $L_\sigma$.

\begin{lemma}\label{lemmeappendice} Soient $K$ une extension finie de
$k$ et $x\in A(K)$. Notons $\widehat{h}_{L}(x)$ la hauteur de
N\'eron-Tate de $x$ relative {\`a} $L$. Pour tout plongement
$\sigma\colon K\hookrightarrow\mathbf{C}$, consid\'erons un
logarithme $z_{\sigma}$ de $x$ dans $t_{A_{\sigma}}$~:
$x=\exp_{A_{\sigma}}(z_{\sigma})$. Supposons que $x$ ne
soit pas dans le support du diviseur $E$ associ\'e \`a $L$. Alors on
a $$a\le\widehat{h}_{L}(x)+\frac{1}{[K:\mathbf{Q}]}\sum_{\sigma\colon
K\hookrightarrow\mathbf{C}}{\log\vert F_{\sigma}(z_{\sigma})\vert}.$$\end{lemma}

\begin{proof} Consid\'erons un mod\`ele de Moret-Bailly
$(\mathcal{A},\overline{\mathcal{L}},\epsilon_{x})$ de $(A,L,x)$ sur
une extension finie $K'$ de $K$ (voir \S~\ref{subsectionchoix} et
\cite[\S~4.3]{bostduke})~: 
\begin{enumerate}\item[(i)]$\mathcal{A}\to\spec\mathcal{O}_{K'}$  
est un sch\'ema en groupes lisse, de fibre g\'en\'erique
$A_{K'}$,\item[(ii)] $\overline{\mathcal{L}}$ est un faisceau
inversible hermitien cubiste sur $\mathcal{A}$, de fibre
g\'en\'erique $L_{K'}$, \item[(iii)]
$\epsilon_{x}\colon\spec\mathcal{O}_{K'}\to\mathcal{A}$ est une
section qui rel\`eve $x\in A(K')$.\end{enumerate}Comme dans le
paragraphe~\ref{subsectionchoix}, ce mod\`ele conf\`ere {\`a}
l'espace $H:=\mathrm{H}^{0}(A_{K'},L_{K'})$ une structure de
fibr\'e hermitien ad\'elique $\overline{H}$ sur $K'$. Par
hypoth\`ese, une section $s\in H\setminus\{0\}$ ne s'annule pas en
$x$ et, puisque $\mathrm{h}^{0}(A,L)=1$, on dispose de la
formule $$\widehat{\mu}(\overline{H})=\widehat{\mu}(
\overline{\epsilon_{x}^{*}\mathcal{L}})
+\frac{1}{[K':\mathbf{Q}]}\sum_{v}{[(K')_{v}:
\mathbf{Q}_{v}]\log\frac{\| s(x)\|_{\overline{
\epsilon_{x}^{*}\mathcal{L}},v}}{\| s\|_{\overline{H},v}}}$$
(dans la somme, $v$ parcourt les places de $K'$). La
pente $\widehat{\mu}(\overline{\epsilon_{x}^{*}\mathcal{L}})$
est \'egale \`a $\widehat{h}_{L}(x)$~\cite[th\'eor\`eme $4.10$,
(ii)]{bostduke} tandis que $\widehat{\mu}(\overline{H})=a$
(voir~\eqref{formulepenterestriction}). Dans la somme, on s\'epare les
places ultram\'etriques des places archim\'ediennes. Si $v$ est
ultram\'etrique, consid\'erons une base $s_{v}$ du
$\mathcal{O}_{v}$-module libre (de rang $1$)
$\mathrm{H}^{0}(\mathcal{A},\mathcal{L})\otimes_{\mathcal{O}_{K'}}\mathcal{O}_{v}$
($\mathcal{O}_{v}$ est l'anneau de valuation du compl\'et\'e de
$K'$ en la place $v$). On a alors $$\frac{\|
s(x)\|_{\overline{\epsilon_{x}^{*}\mathcal{L}},v}}{\|
s\|_{\overline{H},v}}=\|
s_{v}(\epsilon_{x})\|_{\overline{\epsilon_{x}^{*}\mathcal{L}},v}\le1.$$
Si $v$ est archim\'edienne et si $\sigma\colon
K'\hookrightarrow\mathbf{C}$ est un plongement complexe associ\'e,
la fonction th\^eta $\vartheta_\sigma\colon t_{A_\sigma}\to\mathbf{C}$ du
lemme~\ref{tht} (2) correspond \`a un \'el\'ement $s_\sigma\in
H\otimes_{\sigma}\mathbf{C}$ avec lequel nous pouvons calculer le
quotient des normes, en utilisant la relation~\eqref{formulesectiontheta}
du~\S~\ref{subsectionchoix} puis le lemme~\ref{tht} (3)~: $$\|
s_\sigma(x)\|_{\overline{\epsilon_{x}^{*}\mathcal{L}},v}=\left|
\vartheta_\sigma(z_{\sigma})\right|e^{-\frac{\pi}{2}\|
z_{\sigma}\|_{L,\sigma}^{2}}=|F_{\sigma}(z_{\sigma})|.$$ Le calcul de la norme $\|
s_\sigma\|_{\overline{H},v}$ se fait en \'elevant cette derni\`ere
formule au carr\'e et en int\'egrant. On trouve donc (lemme~\ref{tht} (1))
$\|s_\sigma\|_{\overline{H},v}=1$
puis $\|s(x)\|_{\overline{\epsilon_x^*\mathcal{L}},v}/\|s\|_{\overline{H},v}
=\|s_\sigma(x)\|_{\overline{\epsilon_x^*\mathcal{L}},v}=|F_{\sigma}(z_{\sigma})|.$
En regroupant toutes ces informations, nous avons la formule voulue
car $F_{\sigma}(z_{\sigma})$ ne d\'epend que de la restriction de
$\sigma$ \`a $K$.\end{proof}

\begin{proof}[D\'emonstration du th\'eor\`eme~\ref{thBost}] 
Soit $X$ un nombre r\'eel. Sur le compact
$(\mathbf{R}^{g}/\mathbf{Z}^{g})^{2}$, la fonction $f_{X,\sigma}$
d\'efinie par $f_{X,\sigma}(p,q)=\max{\{-X,\log|F_{\sigma}(\tau_\sigma
p+q)|\}}$ est continue (\`a valeurs r\'eelles). \'Etant donn\'e un
entier $N\ge 1$, posons $I_{N}:=\{0,1,\ldots,N-1\}^{g}$ et, pour $i\in
I_{N}$, notons $p_{i}$ l'image de $i/N$ dans
$\mathbf{R}^{g}/\mathbf{Z}^{g}$. Alors, pour tout nombre r\'eel
$\varepsilon>0$, il existe $N_{0}(X,\varepsilon)\in\mathbf{N}$ tel
que, pour tout entier $N\ge N_{0}(X,\varepsilon)$, pour tout
plongement $\sigma\colon k\hookrightarrow\mathbf{C}$, l'on
ait \begin{equation*}\frac{1}{N^{2g}}\sum_{(i,j)\in
I_{N}^{2}}{f_{X,\sigma}(p_{i},p_{j})}\le\varepsilon+
\int_{(\mathbf{R}^{g}/\mathbf{Z}^{g})^{2}}{f_{X,\sigma}(p,q)
\,\mathrm{d}p\mathrm{d}q}.\end{equation*}
Faisons alors la moyenne sur les plongements $\sigma$ : 
$$\frac{1}{N^{2g}}\frac{1}{[k:\mathbf{Q}]}
\sum_{\sigma\colon k\hookrightarrow\mathbf{C}}\left(\sum_{(i,j)\in I_{N}^{2}}
f_{X,\sigma}(p_{i},p_{j})\right)\le\varepsilon+\frac{1}
{[k:\mathbf{Q}]}\sum_{\sigma\colon k\hookrightarrow\mathbf{C}}
{\int_{(\mathbf{R}^{g}/\mathbf{Z}^{g})^{2}}{f_{X,\sigma}(p,q)
\,\mathrm{d}p\mathrm{d}q}}.$$
Dans le membre de gauche l'on peut librement remplacer $k$ par une
extension finie. Nous consid\'erons ainsi le corps $K_N$ o\`u sont
rationnels tous les points de $N$-torsion de $A$, not\'es $A[N]$. Pour
$x\in A[N]$ et $\sigma$ un plongement de $K_N$ nous notons $u_{x,\sigma}$
le couple $(p,q)$ correspondant \`a un logarithme de $x$ dans
$A_\sigma$. Lorsque $x$ parcourt $A[N]$, \`a $\sigma$ fix\'e,
$u_{x,\sigma}$ parcourt exactement $I_N^2$. Par suite, nous avons
$$\frac{1}{N^{2g}}\sum_{x\in A[N]}\left(\frac{1}{[k:\mathbf{Q}]}
\sum_{\sigma\colon K_N\hookrightarrow\mathbf{C}}f_{X,\sigma}
(u_{x,\sigma})\right)\le\varepsilon+\frac{1}{[k:\mathbf{Q}]}
\sum_{\sigma\colon k\hookrightarrow\mathbf{C}}
{\int_{(\mathbf{R}^{g}/\mathbf{Z}^{g})^{2}}{f_{X,\sigma}(p,q)
\,\mathrm{d}p\mathrm{d}q}}.$$
Le lemme~\ref{lemmeappendice} montre que, si $x\in A[N]$ n'appartient
pas au diviseur $E$, la parenth\`ese du membre de gauche est plus
grande que $a$. Elle est par ailleurs toujours plus grande que $-X$. 
Notre membre de gauche est donc sup\'erieur \`a
$a(1-t_{N}/N^{2g})-Xt_{N}/N^{2g}$ o\`u $t_{N}=\card(A[N]\cap E)$. Par
le th\'eor\`eme de Raynaud (ex-conjecture de Manin-Mumford,
voir~\cite{raymm}) les points de torsion de $E$ sont contenus dans un
nombre fini de translat\'es de sous-vari\'et\'es ab\'eliennes
strictes de $A$. Comme dans chaque tel translat\'e il y a au plus
$N^{2g-2}$ points de $N$-torsion, nous avons $t_N=O(N^{2g-2})$. En
faisant alors tendre $N$ vers l'infini puis $\varepsilon$ vers $0$, on
obtient $$a\le\frac{1}{[k:\mathbf{Q}]}\sum_{\sigma\colon k\hookrightarrow\mathbf{C}}
{\int_{(\mathbf{R}^{g}/\mathbf{Z}^{g})^{2}}{f_{X,\sigma}(p,q)\,\mathrm{d}p
\mathrm{d}q}}.$$ Pour chaque $\sigma$, la suite d\'ecroissante
$(f_{X,\sigma})_{X\in\mathbf{N}}$ de fonctions mesurables converge vers
$(p,q)\mapsto\log|F_{\sigma}(\tau_\sigma p+q)|$. Par convergence monotone, on
a \begin{equation*}\lim_{X\to\infty}
{\int_{(\mathbf{R}^{g}/\mathbf{Z}^{g})^{2}}{f_{X,\sigma}(p,q)\,
\mathrm{d}p\mathrm{d}q}}=\int_{(\mathbf{R}^{g}/\mathbf{Z}^{g})^{2}}{\log
|F_{\sigma}(\tau_{\sigma}p+q)|\,\mathrm{d}p\mathrm{d}q},\end{equation*}
d'o\`u le r\'esultat.\end{proof}

Pour conclure, rappelons que ce r\'esultat a permis \`a Bost
de d\'emontrer une minoration uniforme de la hauteur d'une vari\'et\'e
ab\'elienne (sans hypoth\`ese de polarisation).

\begin{coro}\label{minB} Pour toute vari\'et\'e ab\'elienne $A$
d{\'e}finie sur  un corps de nombres, on a $h(A)\ge-(1/2)(\dim
A)\log(2\pi)$.\end{coro}

\begin{proof} Dans le cas principalement polaris\'e, il suffit de voir
dans les notations ci-dessus $a\le0$ ou m\^eme, par le
th\'eor\`eme, $$\int_{(\mathbf{R}^g/\mathbf{Z}^g)^2}\log|F_\sigma|\le0.$$
Or, par concavit\'e du logarithme, on a 
$$\int_{(\mathbf{R}^g/\mathbf{Z}^g)^2}\log|F_\sigma|=
\frac{1}{2}\int_{(\mathbf{R}^g/\mathbf{Z}^g)^2}\log|F_\sigma|^2
\le\frac{1}{2}\log\int_{(\mathbf{R}^g/\mathbf{Z}^g)^2}|
F_\sigma|^2=0$$ par le lemme~\ref{tht}. Dans le cas g\'en\'eral on
applique la minoration \`a la vari\'et\'e $A^4\times(\widehat{A})^4$ qui
est principalement polaris\'ee de hauteur $8h(A)$ et de dimension
$8\dim A$ (astuce de Zarhin).\end{proof}

\end{document}